\documentclass[12pt,reqno]{amsart}

\usepackage{xcolor}
\usepackage{comment}
\usepackage{amsmath}
\usepackage{amssymb}
\usepackage{amsthm}
\usepackage[latin1]{inputenc}
\usepackage{eurosym}
\usepackage{graphicx}
\usepackage{epsfig}
\usepackage{hyperref}
\usepackage{dsfont}
\usepackage[nocompress, space]{cite}
\usepackage{caption}
\usepackage{subcaption}
\usepackage[displaymath,mathlines]{lineno}
\usepackage{enumerate}
\usepackage{cancel}

\allowdisplaybreaks

\usepackage{ifthen}

\newcommand{\R}{\mathbb{R}}

\newcommand{\ol}{\overline}

\newcommand{\argmax}{\operatorname{argmax}}

\renewcommand{\div}{\operatorname{div}}

\def\leq{\leqslant}
\def\geq{\geqslant}

\numberwithin{equation}{section}

\newtheoremstyle{thmlemcorr}{10pt}{10pt}{\itshape}{}{\bfseries}{.}{10pt}{{\thmname{#1}\thmnumber{
#2}\thmnote{ (#3)}}}
\newtheoremstyle{thmlemcorr*}{10pt}{10pt}{\itshape}{}{\bfseries}{.}\newline{{\thmname{#1}\thmnumber{
#2}\thmnote{ (#3)}}}
\newtheoremstyle{defi}{10pt}{10pt}{\itshape}{}{\bfseries}{.}{10pt}{{\thmname{#1}\thmnumber{
#2}\thmnote{ (#3)}}}
\newtheoremstyle{remexample}{10pt}{10pt}{}{}{\bfseries}{.}{10pt}{{\thmname{#1}\thmnumber{
#2}\thmnote{ (#3)}}}
\newtheoremstyle{ass}{10pt}{10pt}{}{}{\bfseries}{.}{10pt}{{\thmname{#1}\thmnumber{
A#2}\thmnote{ (#3)}}}

\theoremstyle{thmlemcorr}
\newtheorem{theorem}{Theorem}
\numberwithin{theorem}{section}
\newtheorem{lemma}[theorem]{Lemma}
\newtheorem{corollary}[theorem]{Corollary}
\newtheorem{proposition}[theorem]{Proposition}

\theoremstyle{thmlemcorr*}
\newtheorem{theorem*}{Theorem}
\newtheorem{lemma*}[theorem]{Lemma}
\newtheorem{corollary*}[theorem]{Corollary}
\newtheorem{proposition*}[theorem]{Proposition}
\newtheorem{problem*}[theorem]{Problem}
\newtheorem{conjecture*}[theorem]{Conjecture}

\theoremstyle{defi}
\newtheorem{definition}[theorem]{Definition}

\theoremstyle{remexample}
\newtheorem{remark}{Remark}
\newtheorem{example}[theorem]{Example}

\theoremstyle{plain}

\theoremstyle{ass}

\title[Forced mean curvature flow]{Capillary-type boundary value problems of mean curvature flows with force and transport terms on a bounded domain}

\begin{document}

\author{Jiwoong Jang}
\address[Jiwoong Jang]
{%
	Department of Mathematics, 
	University of Wisconsin Madison, Van Vleck hall, 480 Lincoln drive, Madison, WI 53706, USA}
\email{jjang57@wisc.edu}


\thanks{
	The work of JJ was partially supported by NSF CAREER grant DMS-1843320.
}


\begin{abstract}
In this paper, we study the forced mean curvature flows and the prescribed mean curvature equations of both graphs and level-sets with capillary-type boundary conditions on a $C^3$ bounded domain, which is not necessarily convex. We prove \emph{a priori} gradient estimates locally Lipschitz in time. Under an assumption on the forcing term, we prove that the gradient estimates are globally Lipschitz in time. As a consequence, we obtain the existence theorem of solutions. In our formulation, we recover the known results of the gradient estimates on a strictly convex $C^3$ bounded domain. Next, we study the associated eigenvalue problems for mean curvature flows of both graphs and level-sets. We prove the large time behavior of the solutions of mean curvature flows of graphs on a smooth bounded domain. Finally, we compute the asymptotic speed of the solutions of level-set mean curvature flows and the large time profile of level-sets in the radially symmetric case based on optimal control formula. Examples arising in the radially symmetric case demonstrate that the additional assumption on the forcing term is optimal.
\end{abstract}

\maketitle


\section{Introduction} \label{sec:intro}
In this paper, we study the following two problems
\begin{equation}\label{eq:graph}
\begin{cases}
u_t=\sqrt{1+|D u|^2}\div \left(\frac{D u}{\sqrt{1+|D u|^2}}\right)+c(x,u)\sqrt{1+|D u|^2}-f(x,u) \quad &\text{ in } \Omega\times(0,T),\\
\displaystyle \frac{\partial u}{\partial \Vec{\mathbf{n}}}=\phi(x)(\sqrt{1+|D u|^2})^{1-q}\quad &\text{ on } \partial\Omega\times[0,T),\\
u(x,0)=u_0(x) \quad &\text{ on } \overline{\Omega},
\end{cases}
\end{equation}
and
\begin{equation}\label{eq:levelset}
\begin{cases}
u_t=|D u|\div \left(\frac{D u}{|D u|}\right)+c(x,u)|D u|-f(x,u) \quad &\text{ in } \Omega\times(0,T),\\
\displaystyle \frac{\partial u}{\partial \Vec{\mathbf{n}}}=\phi(x)\quad &\text{ on } \partial\Omega\times[0,T),\\
u(x,0)=u_0(x) \quad &\text{ on } \overline{\Omega},
\end{cases}
\end{equation}
where $q>0$ in \eqref{eq:graph} is a fixed positive number, and $T>0$ denotes values in $(0,\infty]$. Solutions of \eqref{eq:levelset} are understood in the viscosity sense. A forcing term $c=c(x,z)$ and a transport term $f=f(x,z)$ depend on the spatial position $x\in\overline{\Omega}$ and the value $z\in\mathbb{R}$, and they are functions in $C^{1,\alpha}(\overline{\Omega}\times\mathbb{R})$ for a fixed $\alpha\in(0,1)$. The functions $c$ and $f$ of $(x,z)\in\overline{\Omega}\times\mathbb{R}$ are assumed, throughout this paper, to be $C^{1,\alpha}$ functions and to satisfy, for some constant $C$,
\begin{equation}\label{assumtion:c}
|c|\leq C,\quad |D_x c|\leq C,\quad c_z\leq0, 
\end{equation}
and
\begin{equation}\label{assumtion:f}
|f|\leq C,\quad |D_x f|\leq C,\quad f_z\geq0,
\end{equation}
for all arguments $(x,z)\in\overline{\Omega}\times\mathbb{R}$. The vector $\Vec{\mathbf{n}}$ denotes the outward unit normal vector to $\partial\Omega$, and $\phi=\phi(x)\in C^3(\overline{\Omega})$.
Throughout this paper, we assume that the domain $\Omega \subset \R^n$ is bounded and $C^3$-regular. We also assume that $u_0\in C^{2,\alpha}(\overline{\Omega})$ with the same $\alpha\in(0,1)$ as above, and we say the initial condition $u_0$ is compatible with the boundary condition if
$$
\frac{\partial u_0}{\partial \Vec{\mathbf{n}}}=\phi(x)(\sqrt{1+|D u_0|^2})^{1-q}\ \text{on}\ \partial\Omega
$$
in \eqref{eq:graph} and
$$
\frac{\partial u_0}{\partial \Vec{\mathbf{n}}}=\phi(x)\ \text{on}\ \partial\Omega
$$
in \eqref{eq:levelset}, and we always assume the compatibility in this paper. Next, we consider the following forced mean curvature equations 

\begin{equation}\label{eq:eigengraph}
\begin{cases}
-\sum_{i,j=1}^n\left(\delta^{ij}-\frac{w_iw_j}{1+|Dw|^2}\right)w_{ij}-c(x)\sqrt{1+|D w|^2}+f(x)=-\lambda \quad &\text{ in } \Omega,\\
\displaystyle \frac{\partial w}{\partial \Vec{\mathbf{n}}}=\phi(x)(\sqrt{1+|D w|^2})^{1-q}\quad &\text{ on } \partial\Omega 
\end{cases}
\end{equation}
with general capillary-type boundary conditions and
\begin{equation}\label{eq:eigenlevelset}
\begin{cases}
-\sum_{i,j=1}^n\left(\delta^{ij}-\frac{w_iw_j}{|Dw|^2}\right)w_{ij}-c(x)|Dw|+f(x)=-\lambda \quad &\text{ in } \Omega,\\
\displaystyle \frac{\partial w}{\partial \Vec{\mathbf{n}}}=\phi(x)\quad &\text{ on } \partial\Omega,
\end{cases}
\end{equation}
with Neumann boundary conditions. Here, $u_i=u_{x_i},\ u_{ij}=u_{x_ix_j}$ (and the same for $w$) denote the partial derivatives of $u$ in $x_i$, $x_i$ and $x_j$ in order, respectively. The term $\delta^{ij}$ is the $(i,j)$-entry of the $n$ by $n$ identity matrix for $i,j=1,\cdots,n$. Equation \eqref{eq:eigenlevelset} is understood in the viscosity sense. Equations \eqref{eq:eigengraph} and \eqref{eq:eigenlevelset} correspond to \eqref{eq:graph} and \eqref{eq:levelset}, respectively. $\lambda$ is a real number, and it is called an eigenvalue. The stationary problems \eqref{eq:eigengraph} and \eqref{eq:eigenlevelset} are also considered as additive eigenvalue problems.

The four equations above, \eqref{eq:graph}, \eqref{eq:levelset}, \eqref{eq:eigengraph} and \eqref{eq:eigenlevelset},  will be studied by obtaining \emph{a priori} $C^1$ estimates for
\begin{equation}\label{eq}
\begin{cases}
u_t=\sqrt{\eta^2+|D u|^2}\div \left(\frac{D u}{\sqrt{\eta^2+|D u|^2}}\right)+c(x,u)\sqrt{\eta^2+|D u|^2}-f(x,u) \quad &\text{ in } \Omega\times(0,T),\\
\displaystyle \frac{\partial u}{\partial \Vec{\mathbf{n}}}=\phi(x)v^{1-q}\quad &\text{ on } \partial\Omega\times[0,T),\\
u(x,0)=u_0(x) \quad &\text{ on } \overline{\Omega},
\end{cases}
\end{equation}
and \emph{a priori} $C^0$, $C^1$ estimates for 
\begin{equation}\label{eq:eigen}
\begin{cases}
-\sum_{i,j=1}^n\left(\delta^{ij}-\frac{u_iu_j}{\eta^2+|Du|^2}\right)u_{ij}-c(x)\sqrt{\eta^2+|D u|^2}+f(x)=-ku \quad &\text{ in } \Omega,\\
\displaystyle \frac{\partial u}{\partial \Vec{\mathbf{n}}}=\phi(x)v^{1-q}\quad &\text{ on } \partial\Omega 
\end{cases}
\end{equation}
where $v=\sqrt{\eta^2+|D u|^2}$ and $k>0$. The choices $\eta=1,\ q>0$ and $\eta=0,\ q=1$ in \eqref{eq} yield \eqref{eq:graph} and \eqref{eq:levelset}, respectively. The same choices in \eqref{eq:eigen} correspond to \eqref{eq:eigengraph} and \eqref{eq:eigenlevelset}, respectively  after letting $k\to0$. In the choice of $\eta=0,\ q=1$, we first take $\eta\in(0,1]$, and then we let $\eta\to0$, considered as a vanishing viscosity parameter. Whenever we discuss the vanishing viscosity parameter $\eta\in(0,1]$, especially obtaining estimates uniform in $\eta\in(0,1]$, we refer to the case $q=1$.


We note that if $q=0$, \eqref{eq:graph} is the capillary problem, and \eqref{eq:levelset} is the capillary problem formulated as the level-set equation. If $q=1$, \eqref{eq:graph} and \eqref{eq:levelset} are Neumann boundary value problems. We investigate the well-posedness and the large time behavior of the forced mean curvature flow on a $C^3$ bounded domain with general capillary-type boundary conditions, i.e., $q>0$.

The novelty of this paper is threefold; first of all, the multiplier method in \cite{JKMT} can be combined with the method in \cite{WWX} in order to get \emph{a priori} gradient estimates of \eqref{eq} uniform in $\eta\in(0,1]$. The combination of the methods allows us to handle the difficulties coming from the nonconvexity of $\Omega$, a forcing term $c$, a transport term $f$, a nonzero boundary condition with $\phi\not\equiv0$ at the same time. By using the two methods simultaneously, we get a uniform \emph{a priori} gradient estimate, and therefore, we get quite general results. This is the main contribution of this paper. In the gradient estimate, we derive a sufficient condition on a forcing term $c$ to ensure the global Lipschitz regularity, which we call the \emph{coercivity} assumption on $c$. Second of all, we keep the force term $c$ coercive during the interpolation, while we apply the Leray-Schauder fixed point theorem, so that a uniform gradient estimate is maintained. This extra care on the force $c$ is a new step, not arising in \cite{WWX}, and it is necessary and natural since we observe that the coercivity condition is crucial to study the large time behavior. We accordingly are able to study the mean curvature equations \eqref{eq:eigengraph} and \eqref{eq:eigenlevelset}. Finally, by adopting the approaches in \cite{JKMT, GMT}, we discuss the optimality of the coercive condition on $c$, and compute the eigenvalue, the large time profile based on the optimal control formula in the radially symmetric setting of \eqref{eq:levelset}. We also give a dynamics proof in order to deal with the boundary, which does not appear in \cite{GMT}, when we study the asymptotic behavior.

The multiplier method in \cite{JKMT} has been considered new and devised only recently, and it successfully treats the homogeneous Neumann boundary condition. The method is natural, and it explains how the geometry of $\partial\Omega$ affects gradient estimates, which turn out to be sharp. This paper presents as a new contribution that the multiplier method can be generalized to deal with general capillary-type boundary conditions by combining with the method that has been established in \cite{WWX}. The result is general because \eqref{eq:graph} and \eqref{eq:levelset} cover a wide range of equations on a general bounded domain. The process of combining is linear and natural, which justifies that each of the methods is natural. Moreover, the multiplier method highlights the coercivity assumption on the force $c$ with the right angle condition. Another observation of this paper is that we can study the additive eigenvalue problem with this coercivity condition.

We first discuss the literature, which is not an exhaustive list at all, on the capillary problem and the Neumann boundary value problem of mean curvature flows in Subsection \ref{subsec:1.1}. Next, we provide the main results in Subsection \ref{subsec:1.2}, and we outline the approaches of this paper in Subsection \ref{subsec:1.3}.

\medskip

\subsection{Literature}\label{subsec:1.1}
The capillary problem has been an important subject for decades because of motivations and applications in physics, such as wetting phenomena \cite{CQ, GBQ}, behaviors of droplets \cite{AD, CM, DGO, Q}. It also has been investigated with emphasis on obtaining gradient estimates. For instance, \cite{U, SS, Gerhardt, L} study gradient estimates of the mean curvature equation with test function technique. In 1975, the maximum principle was first used to get gradient estimates \cite{S}, and \cite{K, L} are based on the maximum principle. Paper \cite{L} also deals with boundary conditions $q=0$ and $q>1$, and in these cases, boundary gradient estimates have been shown \cite{X} recently with a new proof using the maximum principle. The results when $0<q<1$ have been obtained in \cite{WWX}. For the mean curvature flow, the well-posedness and the large time behavior of solutions has been studied in \cite{AW, Guan}. In particular, \cite{AW} deals with the case when $q=0$ in the dimension $n=2$, and the questions about the well-posedness and the large time behavior in higher dimensions are still open. The vertical capillary problem, i.e., when $\phi(x)=0$ and thus when the problem is also the homogeneous Neumann boundary problem, has been investigated \cite{H}.

The mean curvature flow with Neumann boundary conditions has been of significance on its own. Paper \cite{AC} investigates the mean curvature equation with the homogeneous Neumann condition on a convex domain in the graph case. Recently, the mean curvature flow with general Neumann boundary conditions has been studied \cite{X2}, and a uniform gradient estimate has been obtained for Neumann boundary conditions on a strictly convex domain \cite{MWW}. Also, \cite{MT} studies gradient estimates with Neumann boundary conditions.

The level-set formulation of the mean curvature flow with the homogeneous Neumann boundary condition, understood in the viscosity sense, has been studied \cite{GOS} on a smoothly bounded convex domain, based on the maximum principle. Paper \cite{GOS} also contains an illustration where we lose a global gradient estimate on a nonconvex domain. Note that the illustration justifies the necessity of a nonzero force term in order to have a global gradient estimate on a nonconvex domain. In this context, the results on the forced mean curvature flow with the right angle condition have been obtained \cite{JKMT} recently, which explains the effect of the constraints by the forcing term and by the geometry of the boundary. However, there are no results on the forced mean curvature flow and the forced mean curvature equation with more general boundary conditions on a general bounded domain, for neither the graph case nor the level-set case.

In the context of the above, the main goal of this paper is to study the well-posedness and the large time behavior of solutions of capillary-type boundary value problems, i.e., $q>0$, of the mean curvature flow with a forcing term and a transport term for the graph case, and to study Neumann boundary problems, $q=1$, for the level-set case, on a bounded domain with $C^3$ boundary, which is not necessarily convex. It generalizes \cite{WWX} to capillary-type boundary value problems on a nonconvex domain with a force, and generalizes \cite{JKMT} to nonzero Neumann boundary value problems with a transport term.

\medskip
\subsection{Main results}\label{subsec:1.2} We first list the main results of this paper, and then discuss the main difficulties and the approaches to overcome.

We start with a local gradient estimate.

\begin{theorem}\label{thm:local-grad}
Let $\Omega$ be a $C^3$ bounded domain in $\mathbb{R}^n,\ n\geq2$. Suppose that $c$ and $f$ satisfy \eqref{assumtion:c} and \eqref{assumtion:f}. Then, for each $T\in(0,\infty)$, there exists a unique solution $u\in C^{2,
\sigma}(\overline{\Omega}\times[0,T])\cap C^{3,
\sigma}(\Omega\times(0,T])$ of \eqref{eq:graph} for some $\sigma\in(0,1)$, and there exists a unique viscosity solution $u$ of \eqref{eq:levelset}. For both \eqref{eq:graph} and \eqref{eq:levelset}, moreover, there exists a constant $M>0$ such that and  for each $T\in(0,\infty)$, there exists a constant $R_T>0$ depending only on $T$, $\Omega$, $c$, $f$, $\phi$,  $q$, $u_0$ such that 
\[
\begin{cases}
|u(x,t)-u(x,s)|\leq M|t-s|,\\
|u(x,t)-u(y,t)|\leq R_T|x-y|,
\end{cases}
\quad \text{ for all $x,y\in\overline{\Omega}$, $t,s\in [0,T]$.}
\]
\end{theorem}

For each $x\in\mathbb{R}^n,\ r>0$, we let $B(x,r)$ denote the open ball centered at $x$ with a radius $r$. We recall that for $y\in\partial\Omega$, $\Vec{\mathbf{n}}(y)$ is defined to be the outward unit normal vector to $\partial\Omega$ at $y$. For each $y\in\partial\Omega$, we define the number $K_0(y)$ by
$$
K_0(y)=\sup\{r>0:B(y-r\Vec{\mathbf{n}}(y),r)\subseteq\Omega\}.
$$
Note that the domain $\Omega$ satisfies the uniform interior ball condition since $\Omega$ is a $C^3$ bounded domain. Therefore, there exists a number $\hat{r}>0$ such that $B(y-\hat{r}\Vec{\mathbf{n}}(y),\hat{r})\subseteq\Omega$ for all $y\in\partial\Omega$, which implies $K_0(y)\geq \hat{r}$ for all $y\in\partial\Omega$. We also note that for each $y\in\partial\Omega$, $B(y-K_0(y)\Vec{\mathbf{n}}(y),K_0(y))\subseteq\Omega$, and $B(y-(K_0(y)+\varepsilon)\Vec{\mathbf{n}}(y),K_0(y)+\varepsilon)\nsubseteq\Omega$ for any $\varepsilon>0$. 

For each $y\in\partial\Omega$, we define the number $C_{0}(y)$ by
$$
C_0(y)=\max\{\lambda:\lambda\ \textrm{is an eigenvalue of}\ -\kappa\},
$$
where $\kappa:=\left(\kappa^{\ell j}\right)_{\ell,j=1}^{n-1}$ is the curvature matrix of $\partial\Omega$ at $y$.

Next we show that a solution $u$ is globally Lipschitz under further conditions on the forcing term $c$.

\begin{theorem}\label{thm:global-grad}
Let $\Omega$ be a $C^3$ bounded domain in $\mathbb{R}^n,\ n\geq2$. Let
\[
\begin{cases}
C_0=\sup\{C_0(y):y\in\partial\Omega\},\\
K_0=\inf\{K_0(y):y\in\partial\Omega\}.
\end{cases}
\]
Suppose that $c$ and $f$ satisfy \eqref{assumtion:c} and \eqref{assumtion:f}. Suppose that there exists $\delta>0$ such that 
\begin{equation}\label{condition:c}
\frac{1}{n-1}c(x,z)^2-|Dc(x,z)|-\delta>\max\left\{0,\ C_0|c(x,z)|+\frac{(n-1)C_0}{K_0}+(1+q)\mathrm{sgn}(C_0)C_0^2\right\}
\end{equation}
for all $(x,z)\in\overline{\Omega}\times\mathbb{R}$, where $\mathrm{sgn}(C_0)$ is the sign of the real number $C_0$. Let $T\in(0,\infty)$, and let $u\in C^{2,
\sigma}(\overline{\Omega}\times[0,T])\cap C^{3,
\sigma}(\Omega\times(0,T])$ be the unique solution of \eqref{eq:graph}, $\sigma\in(0,1)$, and with abuse of notations, let $u$ be the unique viscosity solution $u$ of \eqref{eq:levelset}. In both cases, there exist constants $M,L>0$, depending only on $\Omega$, $c$, $f$, $\phi$, $q$, $u_0$ such that
\[
\begin{cases}
|u(x,t)-u(x,s)|\leq M|t-s|,\\
|u(x,t)-u(y,t)|\leq L |x-y|,
\end{cases}
\quad \text{ for all $x,y\in\overline{\Omega}$, $t,s\in[0,T]$.}
\]
\end{theorem}

We can relax the conditions \eqref{assumtion:c} and \eqref{assumtion:f} quite a bit if we have \emph{a priori} $C^0$ estimate on $u$. For instance, $\widetilde{f}(x,z)=f(x)+kz$, $k>0$, is not bounded as $z$ runs over $\mathbb{R}$. However, if we know that a solution $u$ is bounded \emph{a priori}, then $\widetilde{f}(x,u)=f(x)+ku$ is bounded as well. Therefore, once we get \emph{a priori} $C^0$ estimate on $u$, we can drop the assumptions $|c|\leq C,\ |f|\leq C$ in \eqref{assumtion:c}, \eqref{assumtion:f}, respectively, for Theorem \ref{thm:local-grad} and Theorem \ref{thm:global-grad}.

The condition \eqref{condition:c} serves as a coercivity assumption, which appears in the classical Bernstein method. In this sense, we sometimes call the forcing term $c$ \emph{coercive} if $c$ satisfies \eqref{condition:c}. One more remark is that the coercivity condition \eqref{condition:c} is an open condition, in the sense that it remains true even if we perturb the force $c$ a little bit.

When the domain $\Omega$ is convex so that $C_0\leq0$, the condition \eqref{condition:c} is equivalent to taking only zero on the right hand of \eqref{condition:c} into account. On the other hand, if the domain $\Omega$ is nonconvex so that $C_0>0$, the condition \eqref{condition:c} considers only $C_0|c(x,z)|+\frac{(n-1)C_0}{K_0}+(1+q)\mathrm{sgn}(C_0)C_0^2$, and moreover, this condition is stronger than the convex case. In other words, we require a stronger coercivity condition on the force to deal with the nonconvex boundary $\partial\Omega$. We may refer to the example on a nonconvex domain suggested in \cite[Section 6]{JKMT}.

The condition \eqref{condition:c} is slightly better than the one given in \cite[Theorem 1.2]{JKMT} in the case when $\phi\equiv0$ so that the boundary condition is the homogeneous Neumann boundary condition, or the right angle condition equivalently. More precisely, when $\phi\equiv0$, one can see easily that the condition \eqref{condition:c} with $q=0$ follows from the condition in \cite[Theorem 1.2]{JKMT}. Thus, the condition in \cite[Theorem 1.2]{JKMT} is assuming more. We also note that the condition \eqref{condition:c} with $q=0$ works as a sufficient condition by following the proof of Theorem \ref{thm:global-grad}.

We note that $C_0$ measures the curvature on the boundary $\partial\Omega$, and $K_0$ measures the width of the domain $\Omega$ with inscribed balls. The appearance of the fraction $\frac{C_0}{K_0}$ in \eqref{condition:c} reflects the battle of the two constraints, namely, from the normal velocity $V=k_1+c$ and from the boundary condition $\frac{\partial u}{\partial \Vec{\mathbf{n}}}=\phi(x)v^{1-q}$, where $k_1$ is $(n-1)$ times of the mean curvature of a level-set of $u$. 

We also note that if $\Omega$ is strictly convex, then $C_0<0$ so that $C_0|c(x,z)|+\frac{(n-1)C_0}{K_0}-(1+q)C^2_0<0$. This implies that there is a room for improvement of estimates if $\Omega$ is strictly convex, and indeed it turns out that we can recover a global gradient estimate if $c(x,z)\equiv0$. We state the following corollary for $c\equiv0$, which is \cite[Theorem 1.1]{WWX} for \eqref{eq:graph}, together with the corresponding conclusion for \eqref{eq:levelset}.

\begin{corollary}\label{cor:strcvxlevelset}
Let $\Omega$ be a strictly convex $C^3$ bounded  domain in $\mathbb{R}^n,\ n\geq2.$ Let $c\equiv0$. Suppose that the term $f$ satisfies \eqref{assumtion:f}. Then, for each $T\in(0,\infty)$, there exists a unique solution $u\in C^{2,
\sigma}(\overline{\Omega}\times[0,T])\cap C^{3,
\sigma}(\Omega\times(0,T])$ of \eqref{eq:graph} for some $\sigma\in(0,1)$, and there exists a unique viscosity solution $u$ of \eqref{eq:levelset}, with abuse of notations. In both cases, moreover, there exist constants $M,L>0$ depending only on $\Omega$, $c$, $f$, $\phi$, $q$, $u_0$ such that $|u(x,t)-u(x,s)|\leq M|t-s|$, $|u(x,t)-u(y,t)|\leq L |x-y|$ for all $x,y\in\overline{\Omega}$, $t,s\in[0,T]$.
\end{corollary}

As we have obtained gradient estimates, we next study the additive eigenvalue problems \eqref{eq:eigengraph} and \eqref{eq:eigenlevelset} under the assumption \eqref{condition:c} on the forcing term $c$. In the additive eigenvalue problems, we will consider the terms $c=c(x)$ and $f=f(x)$ that depend only on $x\in\overline{\Omega}$. That being said, the $z$-dependence in the estimates obtained so far plays a role in the additive eigenvalue problems.

Before we introduce the next results, we explain how the additive eigenvalue problem is approached briefly. First of all, we get uniform $C^0$ \emph{a priori} estimates of $|ku|$ in \eqref{eq:eigen} by the maximum principle. Then, we establish uniform $C^1$ \emph{a priori} estimates of \eqref{eq:eigen}. Applying Leray-Schauder fixed point theorem (see \cite{LT}), we get the existence of solutions of \eqref{eq:eigen}. Finding a pair of an eigenvalue and an eigenfunction of \eqref{eq:eigengraph} and \eqref{eq:eigenlevelset} is called additive eigenvalue problems, which have been extensively studied. The problems naturally appear in ergodic optimal control theory, in the homogenization of Hamilton-Jacobi equations, in the large time behavior of the Cauchy problem of Hamilton-Jacobi equations and in weak KAM theory. See \cite{BIM, F, Hung, LPV} and the references therein. We also leave the references \cite{DS, F, FS, I2} for the \emph{Aubry} set, as it is treated separately as an important set in this paper.

\begin{theorem}\label{thm:eigengraph}
Let $\Omega$ be a $C^{\infty}$ bounded domain in $\mathbb{R}^n,\ n\geq2,$ and let $q>0$. Suppose that $c=c(x)$ satisfies \eqref{condition:c}. For $\phi\in C^{\infty}(\overline{\Omega})$, there exists a unique $\lambda\in\mathbb{R}$ such that there exists a solution $u\in C^{\infty}(\overline{\Omega})$ of \eqref{eq:eigengraph}. Moreover, a solution $u$ is unique upto an additive constant.
\end{theorem}

Moreover, we get the following result on the large time behavior of solutions of \eqref{eq:graph} by following the argument in \cite{MWW, WWX, SH}.

\begin{theorem}\label{thm:asympgraph}
Let $\Omega$ be a $C^{\infty}$ bounded domain in $\mathbb{R}^n,\ n\geq2,$ and let $q>0$. Suppose that $c,f,\phi\in C^{\infty}(\overline{\Omega})$, and that $c$ satisfies \eqref{condition:c}. Let $u^i$, $i=1,2$, be the solution of
\begin{equation}
\begin{cases}
u_t=\sqrt{1+|D u|^2}\div \left(\frac{D u}{\sqrt{1+|D u|^2}}\right)+c(x)\sqrt{1+|D u|^2}-f(x) \quad &\text{ in } \Omega\times(0,\infty),\\
\displaystyle \frac{\partial u}{\partial \Vec{\mathbf{n}}}=\phi(x)(\sqrt{1+|D u|^2})^{1-q}\quad &\text{ on } \partial\Omega\times[0,\infty),\\
u(x,0)=u_0^i(x) \quad &\text{ on } \overline{\Omega},
\end{cases}
\end{equation}
with initial data $u_0^i$ compatible with the boundary condition, respectively for $i=1,2$. Then $\lim_{t\to\infty}|u^1-u^2|_{C^{\infty}(\overline{\Omega})}=0$. In particular, for the solution $u$ of \eqref{eq:graph} and the solution $(\lambda,w)$ of \eqref{eq:eigengraph}, it holds that $\lim_{t\to\infty}|u(x,t)-\lambda t-w(x)|_{C^{\infty}(\overline{\Omega})}=0$.
\end{theorem}

We also study the large time behavior of solutions of \eqref{eq:levelset}. We go though the same procedure as we do in Theorem \ref{thm:eigengraph}. During the limit process in which we send $k$ to $0$, the gradient estimates remain uniform in the viscosity parameter $\eta\in(0,1]$, which allows us to find a viscosity solution of the stationary problem \eqref{eq:eigenlevelset}.

\begin{theorem}\label{thm:eigenlevelset}
Let $\Omega$ be a $C^{\infty}$ bounded domain in $\mathbb{R}^n,\ n\geq2$. Suppose that $c$ satisfies \eqref{condition:c}. For $\phi\in C^{\infty}(\overline{\Omega})$, there exists a unique $\lambda\in\mathbb{R}$ such that there exists a viscosity solution $w$ of \eqref{eq:eigenlevelset}. Moreover, $\lambda=\lim_{t\to\infty}\frac{u(x,t)}{t}$ and the convergence as $t\to\infty$ is uniform in $x\in\overline{\Omega}$, where $u$ is the unique viscosity solution of \eqref{eq:levelset} with $T=\infty$.
\end{theorem}

The questions on classifying viscosity solutions $w$ of \eqref{eq:eigenlevelset}, and on whether or not $u(x,t)-\lambda t$ converges to a stationary solution $w$ as $t\to\infty$ are challenging, and they are still widely open. For partial resolutions, we refer to \cite{GOS, JKMT}, where a Lyapunov function is used.

In the radially symmetric setting, we can prove the convergence of $u(x,t)-\lambda t$ to a stationary solution $w$ as $t\to\infty$. Moreover, we are able to compute the eigenvalue $\lambda$ and the large time profile $w$ of the solution $u$ based on the optimal control formula. We will see in Chapter \ref{sec:radial} that the curves $c(r)$ and $\frac{n-1}{r}$ meet at most one point on $[0,R]$ because of the coercivity assumption \eqref{condition:c} on $c$. This fact allows us to follow the argument in \cite{GMT} overall, with the dynamics suggested in \cite{I2}, called the Skorokhod problem.

We also note that the eigenvalue $\lambda=\lim_{t\to\infty}\frac{u(x,t)}{t}$ is constant in $x\in\overline{\Omega}$, but this is under the condition \eqref{condition:c}. We will find an example in the radially symmetric setting, where the limit $\lim_{t\to\infty}\frac{u(x,t)}{t}$ is not constant, which thus disobeys \eqref{condition:c}. It turns out this example demonstrates that the condition \eqref{condition:c} is optimal, which we will discuss in Section \ref{sec:radial}.

\begin{theorem}\label{thm:radprofile}
Assume the radially symmetric setting \eqref{con:radial}. Assume \eqref{condition:c}. Let $u=u(r,t)$ be the unique radial viscosity solution of \eqref{eq:levelset}, and let $(\lambda,w)$ be a pair of a real number and a Lipschitz continuous function satisfying \eqref{eq:eigenlevelset} in the sense of viscosity solutions. Then,

\noindent (i) $u(r,t)-\lambda t\to w(r)$ as $t\to\infty$ uniformly in $r\in[0,R]$, and

\noindent (ii) the asymptotic speed $\lambda$ and the asymptotic profile $w$ are described as follows; if the curves $r\mapsto c(r)$ and $r\mapsto\frac{n-1}{r}$ cross at $r\in[0,R]$, then such numbers $r$ are unique, which we call $r_{cr}$. If the curves do not cross on the interval $[0,R]$, we let $r_{cr}:=\infty$. Then,
\begin{equation}\label{formula:eigenvalue}
\lambda=\sup\left\{-f(r)+\delta(r-R)\phi(R)\left(\frac{n-1}{R}+\mathrm{sgn}(\phi(R))c(R)\right):\ r\geq r_{cr}\textrm{ or }r=R\right\},  
\end{equation}
where $\delta$ is the function on $\mathbb{R}$ having its value 1 at the origin, 0 elsewhere, and the asymptotic profile $w$ is given by
\begin{equation}\label{formula:asympprofile}
w(r)=\max\left\{d(r,s)+w_0(s):\ s\in\widetilde{\mathcal{A}}\right\}.
\end{equation}
Here,
\begin{equation}\label{formula:distance}
d(r_0,r_1):=\sup\left\{\int_{0}^{t}-f(\eta(s))-\phi(\eta(s))l(s)ds:\ t\geq0,\  (\eta,l)\in\mathcal{C}(0,t;r_0,r_1)\right\}
\end{equation}
for any $r_0,r_1\in[0,R]$, where we set
\begin{multline*}
\mathcal{C}(0,t;r_0,r_1):=\left\{(\eta,l)\in\mathrm{AC}([0,t];(0,R])\times L^{\infty}([0,t]):\right.\\
\left.\ \eta(0)=r_0,\ \eta(t)=r_1,
\ (\eta,v,l)\in\mathrm{SP}(r_0)\right\}, 
\end{multline*}
and $\mathrm{SP}(r)$ denotes the Skorokhod problem, and
\begin{align*}
&w_0(r):=\max\left\{d(r,\rho)+u_0(\rho):\ \rho\in[0,R]\right\},\\
\widetilde{\mathcal{A}}:=\left\{r\geq r_{cr}\right.&\left.:\ \textrm{the supremum of \eqref{formula:eigenvalue} is attained}\right\}\qquad\text{if }r_{cr}<\infty.
\end{align*}
If $r_{cr}=\infty$, we let $\widetilde{\mathcal{A}}:=\{R\}$.
\end{theorem}

\medskip

\subsection{Discussions and our main ideas}\label{subsec:1.3} In the following, we first discuss the necessity of a nonzero force in order to get a global gradient estimate and its geometric interpretation. Next, we outline the approaches taken to obtain the results of this paper. 

We start with the special case of \eqref{eq:levelset} when $c(x,z)\equiv0,\ f(x,z)\equiv0,\ \phi(x)\equiv0$, which corresponds to the homogeneous Neumann boundary problem with zero force. Paper \cite{GOS} obtains a global gradient estimate for the problem on a convex domain, and additionally, \cite{GOS} describes an example, which is constructed rigorously in \cite{O} as well, on a nonconvex domain where the global gradient estimate fails. In this context, \cite{JKMT} provides the computation realizing the description, which means we need a nonzero force on a nonconvex domain to get a global estimate. Also, \cite{JKMT} studies the problem with a nonzero force $c=c(x)$, and it generally investigates the competition between the two geometric constraints, one from the normal velocity $V=k_1+c$ where $k_1$ is $(n-1)$ times of the mean curvature, the other from the right angle condition of surfaces and $\partial\Omega$ given by the boundary condition.

\medskip

We now describe the approaches of this paper. We overall rely on the maximum principle to get \emph{a priori} gradient estimates. The difficult case is when a maximizer is on the boundary, where we cannot expect the maximum principle to hold as it is inside the domain. In \cite{WWX}, the difficulty is overcome by considering a slanted gradient in order to get rid of $u_{nn}$, the second derivative of a solution in the normal direction, which is hard to know from the maximum principle. In \cite{JKMT}, the difficulty is handled by considering a multiplier which allows us to put the maximizer inside, so that we can apply the maximum principle. This idea is the crux of the multiplier method, which plays a main role in the estimates in \cite{JKMT}. Moreover, the multiplier method explains how the geometry of the domain affects the estimates, which is natural and geometric. It ultimately enables us to generalize the results of \cite{WWX} on nonconvex domains in a natural way for a wide class of equations \eqref{eq:graph} and \eqref{eq:levelset}. This is how we overcome the difficulty, and it is the main novelty of this paper.

To outline the structure of gradient estimates, we start by observing that both of the methods are relying on the same major term coming from the square norm of the second fundamental form. This is the reason why it is possible to apply the two methods at the same time, and why the process of mix is linear and natural. The whole chain of inequalities starts with applying the maximum principle, and is basically an expansion of a polynomial in $v=\sqrt{\eta^2+|Du|^2}$. Finally, we focus on the coefficient of the highest power of $v$, which yields the coercivity condition \eqref{condition:c} on $c$. We also note that we can get rid of bad terms in the linearized equation.

After we get a global gradient estimate, we next study the mean curvature equations and the large time behavior, as suggested in \cite{WWX}. The part different from \cite{WWX} is where we apply Leray-Schauder fixed point theorem for the mean curvature equations. As we deal with the additional term concerning a nonzero force, we interpolate \eqref{eq:eigen} with a carefully chosen equation so that we keep the force $c$ coercive during the interpolation. A force that is being kept coercive yields a uniform $C^1$ estimate by the gradient estimate obtained above. As an exchange for keeping coercivity in the interpolation, we change the transport term $f$, and this is allowed as long as it is \emph{a priori} bounded. We then follow \cite{WWX} to verify the asymptotic behavior for the graph case, and go through vanishing viscosity process as $\eta\to0$ for the level-set case.

For the level-set mean curvature flow, we compute the eigenvalue and the large time profile, and prove the asymptotic behavior in the radially setting. Equation \eqref{eq:levelset} is reduced to a first-order singular Hamilton-Jacobi equation with Neumann boundary conditions. Based on the optimal control formula \cite{I2}, we are able to compute the eigenvalue. By providing an example where the eigenvalue is not constant, we discuss the optimality of the condition \eqref{condition:c}, which serves as the most important condition to ensure global gradient estimates. The use of the optimal control formula for computing the limit and for an example in this way follows \cite{JKMT}, and it is extended to an equation with a transport term and nonzero boundary conditions. Then, by observing the monotonicity on the Aubry set as in \cite{GMT}, we prove the asymptotic behavior. To deal with the boundary, which does not appear in \cite{GMT}, we instead give a dynamics proof for the monotonicity, written in the style of \cite{DS}.

\medskip


\subsection*{Organization of the paper}

In Section \ref{sec:gradestim}, we prove the existence of solutions of \eqref{eq:graph} and \eqref{eq:levelset} by giving \emph{a priori} local and global gradient estimates. We also recover \cite[Theorem 1.1]{WWX} and the corresponding result for \eqref{eq:levelset} when the domain $\Omega$ is strictly convex. In Section \ref{sec:homogenization}, we prove the existence of solutions of \eqref{eq:eigengraph} and \eqref{eq:eigenlevelset} through homogenization. In Section \ref{sec:radial}, we compute the eigenvalue and the large time profile, and prove the asymptotic behavior of the solution of \eqref{eq:levelset} in the radially symmetric setting. In Appendix, we provide the definitions and the results on the comparison principle and on the stability of viscosity solutions of \eqref{eq:levelset}.


\section{Gradient estimates}\label{sec:gradestim}
In this section, we give \emph{a priori} local gradient estimates of \eqref{eq}, and under the condition \eqref{condition:c} on the forcing term $c$, we prove \emph{a priori} global gradient estimates. Throughout this section, we assume that the conditions \eqref{assumtion:c} and \eqref{assumtion:f} hold, and that $\Omega$ is bounded with $C^3$ boundary.

We leave a remark that for the choice $\eta=1,\ q>0$ in \eqref{eq:graph}, the function $u_0$ serves as an initial data that is compatible with the boundary condition. In \eqref{eq:levelset}, by setting $q=1$, we see that the function $u_0$, which is independent of $\eta\in(0,1]$, serves as an initial data that is compatible with the boundary condition even if $\eta\in(0,1]$ varies. We understand its viscosity solution as the limit of solutions of \eqref{eq} as $\eta\to0$. We also note from the compatibility condition that $|\phi v^{-q}|<1$ on the boundary $\partial\Omega$.

The following lemma states that the time derivative of a solution of \eqref{eq} is bounded. 

\begin{lemma}\label{lem:time-grad}
Suppose that $u^{\eta}$ is the unique solution of \eqref{eq} for each $\eta\in(0,1]$. Suppose \eqref{assumtion:f} and \eqref{assumtion:c}. Fix $T\in(0,\infty)$. 
Then, there exists $M>0$ depending only on $\Omega$, $c$, $f$, $\phi$, $q$, $u_0$  such that
$$
\lVert u^{\eta}_t\rVert_{L^{\infty}(\overline{\Omega}\times[0,T])}\leq\lVert u^{\eta}_t(\cdot,0)\rVert_{L^{\infty}(\overline{\Omega})} \leq M.
$$
\end{lemma}

\begin{proof}
The proof follows the argument in \cite[Lemma 2.1]{WWX}.
\end{proof}

Now we state \emph{a priori} gradient estimates.

\begin{proposition}\label{prop:exuniquegra}
Let $T\in(0,\infty),\ \eta\in(0,1]$. Suppose that a solution $u^{\eta}$ of \eqref{eq} exists and it is of class $C^{2,\sigma}(\overline{\Omega}\times[0,T])\cap C^{3,\sigma}(\Omega\times(0,T])$ for some $\sigma\in(0,1)$. Suppose that the force $c$ satisfies \eqref{condition:c}. Then $u^{\eta}$ satisfies that
$$
\lVert Du^{\eta}\rVert_{L^{\infty}(\overline{\Omega}\times[0,T])}\leq R,
$$
where $R>1$ is a constant depending only on $\Omega, c, f, \phi, q, u_0$.
\end{proposition}

Once we prove Proposition \ref{prop:exuniquegra} (and Proposition \ref{prop:exuniquegratime} introduced later), we obtain the existence of solutions $u=u^{\eta}$ to \eqref{eq} with the bound $\lVert Du^{\eta}\rVert_{L^{\infty}(\overline{\Omega}\times[0,T])}\leq R$ (and therefore prove Theorem \ref{thm:global-grad}), due to the standard theory of quasilinear uniformly parabolic equations, for which we refer to \cite{LSU}. See \cite[Section 5]{MT} for the usage of \cite{LSU}, \cite[Theorem 8.8]{L2}. We also briefly describe the existence from \emph{a priori} estimates in Appendix for completeness.

Before getting into the proof of Proposition \ref{prop:exuniquegra}, we introduce the notations for scalars, vectors, and matrices. After that, we state Lemma \ref{lem:boundary} and Lemma \ref{lem:afternotation} for later use, whose proofs are provided in Appendix.

We set notations. Let $p,q\in\mathbb{R}^n$ be column vectors and $M$ be a symmetric $n$ by $n$ matrix. A real number $p\cdot q$ is the scalar obtained from the standard inner product of $\mathbb{R}^n$, and we let $|p|=\sqrt{p\cdot p}$. A vector $Mp$ is the vector obtained from the standard matrix product. Let $\alpha=\left(\alpha^{ij}\right)_{i,j=1}^n,\beta=\left(\beta^{ij}\right)_{i,j=1}^n$ be two $n$ by $n$ matrices that are not necessarily symmetric. We let $\alpha\beta$ denote the matrix obtained from the standard matrix multiplication of $\alpha$ in the left and $\beta$ in the right. We write $\textrm{tr}\{\alpha\beta^{\textrm{Tr}}\}=\sum_{i,j=1}^n\alpha^{ij}\beta^{ij}$, where $\textrm{tr}\{\cdot\}$ denotes the trace, and $\textrm{Tr}$ denotes the transpose. We let $\|\alpha\|=\sqrt{\textrm{tr}\{\alpha\alpha^{\textrm{Tr}}\}}$.

For a $C^1$ function $\mu$ in $x=(x_1,\cdots,x_n)$, we let $\mu_i$ denote the partial derivative $\mu_{x_i}$ of $\mu$ in $x_i$ for each $i=1,\cdots,n$, and we let $D\mu=\left(\mu_1,\cdots,\mu_n\right)^{\textrm{Tr}}$ be the gradient of $\mu$. For a $C^2$ function, say $\mu$ again, in $x=(x_1,\cdots,x_n)$, we let $\mu_{ij}$ denote the second order partial derivative $\mu_{x_ix_j}$ of $\mu$ in $x_i$ and $x_j$ in order for each $i,j=1,\cdots,n$, and we let $D^2\mu=\left(\mu_{ij}\right)_{i,j=1}^n$ be the Hessian of $\mu$. For a $C^3$ function $\mu$ and a vector $\xi=(\xi^1,\cdots,\xi^n)^{\textrm{Tr}}$, we let $\mu_{\ell ij}$ denote the third order partial derivative $\mu_{x_{\ell}x_ix_j}$ of $\mu$ in $x_{\ell}$, $x_i$ and $x_j$ in order for each $\ell,i,j=1,\cdots,n$, and we let $D^3\mu\odot\xi$ denote the matrix $\left(\sum_{\ell=1}^n\mu_{\ell ij}\xi^{\ell}\right)_{i,j=1}^n$. For $\nu=(\nu^1,\cdots,\nu^n)^{\textrm{Tr}}$, $\nu^i$ a $C^1$ function for each $i=1,\cdots,n$, we let $D\nu$ denote the matrix $\left(\nu^i_{x_j}\right)_{i,j=1}^n$. Then, for a $C^2$ function $\mu$, we check that $D^2\mu=D(D\mu)$.

We define the matrix $a=a(p)$ by $a(p)=I_n-\frac{p\otimes p}{\eta^2+|p|^2}$, where $p\otimes p$ denotes the matrix $\left(p^ip^j\right)_{i,j=1}^n$ for $p=\left(p^1,\cdots,p^n\right)^{\textrm{Tr}}$, and $I_n$ denotes the $n$ by $n$ identity matrix. We let $p\otimes q$ denotes the matrix $\left(p^iq^j\right)_{i,j=1}^n$ for $p=\left(p^1,\cdots,p^n\right)^{\textrm{Tr}},q=\left(q^1,\cdots,q^n\right)^{\textrm{Tr}}\in\mathbb{R}^n$. For a vector $\xi=(\xi^1,\cdots,\xi^n)^{\textrm{Tr}}$, we let $D_pa\odot\xi$ denote the matrix
$$D_pa\odot\xi=\left(\sum_{\ell=1}^na_{p^{\ell}}^{ij}\xi^{\ell}\right)_{i,j=1}^n,$$
where $a_{p^{\ell}}^{ij}=a_{p^{\ell}}^{ij}(p)$ is the partial derivative of $a^{ij}$, the $(i,j)$-entry of the matrix $a$ for $i,j=1,\cdots,n$, in its $\ell$-th variable $p^{\ell}$ of $p=(p^1,\cdots,p^n)^{\textrm{Tr}}$.

Now, we give the setup for Lemma \ref{lem:boundary}. Suppose that $x_0=(0,\cdots,0)\in\partial\Omega$, and that $\Vec{\mathbf{n}}(x_0)=(0,\cdots,0,-1)$. Then, there exist an open neighborhood $U_1$ of $x_0$ in $\mathbb{R}^n$ and a $C^3$ function $\varphi$ defined on $\{x'=(x_1,\cdots,x_{n-1}):(x',0)\in U_1\}$ such that $x=(x',x_n)\in\partial\Omega$ if and only if $x_n=\varphi(x')$. The eigenvalues $\kappa_1,\cdots,\kappa_{n-1}$ of the matrix $D^2\varphi(x_0')$ are called the principal curvatures of $\partial\Omega$ at $x_0$, where $x_0'=(0,\cdots,0)\in\mathbb{R}^{n-1}$, and the corresponding eigenvectors are called the principal directions of $\partial\Omega$ at $x_0$.

By applying a rotation of coordinates to $x'=(x_1,\cdots,x_{n-1})$, we may assume that the $x_{\ell}-$axis lies along a principal direction corresponding to $\kappa_{\ell}$, $\ell=1,\cdots,n-1$, respectively. We call such a coordinate system a principal coordinate system of $\partial\Omega$ at $x_0$. The Hessian matrix $D^2\varphi(x_0)$ with respect to a principal coordinate system of $\partial\Omega$ at $x_0$ is given by the diagonal matrix, as
$$
D^2\varphi(x_0)=
\begin{bmatrix}
\kappa_{1} & &0 \\
& \ddots & \\
0& & \kappa_{n-1}
\end{bmatrix}.
$$

We state Lemma \ref{lem:boundary}, which provides a local parametrization $y'=(y_1,\cdots,y_{n-1})$ of the surface $\partial\Omega$ around $(0,\cdots,0)$ and the derivatives of $C^1$ (or $C^2$) functions in $y=(y_1,\cdots,y_n)$. See \cite[Lemma 14.16]{GL} for the reference of Lemma \ref{lem:boundary}.

\begin{lemma}\label{lem:boundary}
Let $x_0\in\partial\Omega$. For a coordinate $x=(x_1,\cdots,x_n)$ of $\mathbb{R}^n$, suppose that $x_0=(0,\cdots,0)$, and that $\Vec{\mathbf{n}}(x_0)=(0,\cdots,0,-1)$. Suppose also that $x'=(x_1,\cdots,x_{n-1})$ is a principal coordinate system of $\partial\Omega$ at $x_0$, i.e., the $x_{\ell}-$axis lies along a principal direction corresponding to a principal curvature $\kappa_{\ell}$ of $\partial\Omega$ at $x_0$, $\ell=1,\cdots,n-1$, respectively.

Then, there are open neighborhoods $U,V$ of $(0,\cdots,0)$ in $\mathbb{R}^n$ and a $C^2$ diffeomorphism $g:U\to V$, and there is a number $\sigma>0$ satisfying the following properties;

(i) It holds that $g(0,\cdots,0)=(0,\cdots,0)$, and that
$$
\{g(y',0):|y'|<\sigma\}\subseteq\partial\Omega\quad\textrm{ and }\quad\{g(y',y_n):|y'|+|y_n|<\sigma,\ y_n>0\}\subseteq\Omega.
$$
where $y'=(y_1,\cdots,y_{n-1})\in\mathbb{R}^{n-1}$, and

(ii)
$g$ is the identity function on the line $\{(0,\cdots,0,y_n):|y_n|<\sigma\}$.

\noindent If we write $x=g(y),\ y\in U,\ x\in V$, then

(iii)
$$
\frac{\partial\overline{\zeta}}{\partial y_{\ell}}=(1-\kappa_{\ell}y_n)\frac{\partial\zeta}{\partial x_{\ell}}\qquad\textrm{for }\ell=1,\cdots,n,
$$
on the line $\{(0,\cdots,0,y_n):|y_n|<\sigma\}$, which is a subset of $U$. Here, $\zeta=\zeta(x)$ is a $C^1$ function defined on $V$, $\overline{\zeta}(y)$ is the $C^1$ function defined by $\zeta(g(y))$ on $U$, and $\kappa_n$ is set to be $0$. The number $\sigma>0$ satisfies $\sigma^{-1}>\max\{|\kappa_1|,\cdots,|\kappa_{n-1}|\}$.

(iv)
$$
\frac{\partial}{\partial y_n}\left(\frac{\partial \overline{\zeta}}{\partial y_{\ell}}\right)=(1-\kappa_{\ell}y_n)\frac{\partial}{\partial y_n}\left(\frac{\partial\zeta}{\partial x_{\ell}}\right)-\frac{\kappa_{\ell}}{1-\kappa_{\ell}y_n}\frac{\partial\overline{\zeta}}{\partial y_{\ell}}\qquad\textrm{for }\ell=1,\cdots,n,
$$
on the line $\{(0,\cdots,0,y_n):|y_n|<\sigma\}$ if the functions $\zeta,\overline{\zeta}$ given as above are $C^2$ functions.

\end{lemma}

We introduce the following lemma in advance, which will be used in the proof of Proposition \ref{prop:exuniquegra}.

\begin{lemma}\label{lem:afternotation}
Let $u\in C^{2,\sigma}(\overline{\Omega}\times[0,T])\cap C^{3,\sigma}(\Omega\times(0,T])$, and let $v=\sqrt{\eta^2+|Du|^2}$ for $T\in(0,\infty),\ \eta\in(0,1]$. Let $\xi\in\mathbb{R}^n$. Then,
\begin{align}\label{auxiliaryr}
v\textrm{tr}\{(D_pa(Du)\odot\xi)D^2u\}+2\textrm{tr}\{a(Du)(\xi\otimes Dv)\}=0.
\end{align}
\end{lemma}




\begin{proof}[Proof of Proposition \ref{prop:exuniquegra}]
The proof of Proposition \ref{prop:exuniquegra} follows the classical Bernstein method by applying the maximum principle to the function $w:=v^{q+1}-(q+1)\phi Du\cdot Dh$, where $v:=\sqrt{\eta^2+|Du|^2}$.

Let $T\in(0,\infty)$, $\eta\in(0,1]$. Let $u=u^{\eta}\in C^{2,\sigma}(\overline{\Omega}\times[0,T])\cap C^{3,\sigma}(\Omega\times(0,T])$ be a solution to \eqref{eq} for some $\sigma\in(0,1)$. We need to show that $\lVert v\rVert_{L^{\infty}(\overline{\Omega}\times[0,T])}\leq R$ for some constant $R>1$ independent of $T\in(0,\infty)$ and of $\eta\in(0,1]$. Throughout the proof, $R>1$ will denote constants which vary line by line and which do not depend on $T\in(0,\infty)$ and also on $\eta\in(0,1]$. Note that $\eta$ is fixed to be $1$ when $q>0$, and $\eta\in(0,1]$ when $q=1$. Accordingly, $\eta\in(0,1]$ in all cases. Also, $C>0$ will denote constants which vary line by line throughout the proof and also which do not depend on $T\in(0,\infty)$ and also on $\eta\in(0,1]$.

We drop the super and subscript regarding $\eta$, but we are still dealing with \eqref{eq} together with the $\eta$-dependence when $q=1$, which is of importance for \eqref{eq:levelset}. Once we obtain bounds uniform in $\eta\in(0,1]$, we also drop the $\eta$-dependence throughout the estimate.

\medskip

Let $h$ be a function in $C^3(\overline{\Omega})$ such that $h\equiv C,\ Dh=\Vec{\mathbf{n}}$ on the boundary $\partial\Omega$ for some constant $C$. Let $$w=v^{q+1}-(q+1)\phi Du\cdot Dh$$ on $\overline{\Omega}\times[0,T]$. The reason why we choose this $w$ instead of $v=\sqrt{\eta^2+|Du|^2}$ is that we want to cancel out terms involving $\frac{\partial^2u}{\partial \Vec{\mathbf{n}}^2}$, the second derivative of $u$ in the normal direction on the boundary. The reason will be explained with more details when the cancellation occurs.

Fix $(x_0,t_0)\in\argmax_{\overline{\Omega}\times[0,T]}w$. The goal is to show that $v(x_0,t_0)\leq R$ for some constant $R>1$ independent of $T\in(0,\infty),\ \eta\in(0,1]$. Once it is shown, then we obtain $\lVert v\rVert_{L^{\infty}(\overline{\Omega}\times[0,T])}\leq R$, which completes the proof. This is seen by the fact that
$$
w\leq v^{q+1}+(q+1)\|\phi\|_{C^0(\overline{\Omega})}\|h\|_{C^1(\overline{\Omega})}
$$
at $(x_0,t_0)$, and by the fact that
$$
v^{q+1}-(q+1)\|\phi\|_{C^0(\overline{\Omega})}\|h\|_{C^1(\overline{\Omega})}\leq w\leq w(x_0,t_0)\leq R
$$
at $(x,t)\in\overline{\Omega}\times[0,T]$.

If $t_0=0$, we get a uniform bound $v(x_0,t_0)\leq R$, so we are done. It remains the case when $t_0>0$, and we divide the proof into two cases: $x_0\in\Omega$ and $x_0\in\partial\Omega$.

\medskip

\noindent {\bf Case 1: $x_0\in\Omega$.}

\medskip

\emph{Step 1.} We apply the maximum principle at $(x_0,t_0)$ and simplify the resulting inequality.

\medskip

As $x_0\in\Omega,\ t_0>0,$ the maximum principle yields $D^2w\leq0,\ w_t\geq0$ at $(x_0,t_0)$. Therefore, together with the fact that $a(p)\geq0$ as a matrix, we obtain
\begin{equation}\label{maxprinciple}
0\geq\frac{1}{q+1}\left(\textrm{tr}\{a(Du)D^2w\}-w_t\right)\quad\text{ at }(x_0,t_0).
\end{equation}
This is the point where we start a chain of inequalities.

Write $u_t=G+cv-f$, where $G:=\textrm{tr}\{a(Du)D^2u\}$, so that \eqref{maxprinciple} becomes
\begin{align}
0&\geq\frac{1}{q+1}\left(\textrm{tr}\{a(Du)D^2w\}-w_t\right)\nonumber\\
&=\textrm{tr}\{a(Du)D(v^qDv)\}-\textrm{tr}\{a(Du)D^2(\phi Du\cdot Dh)\}-(v^qv_t-\phi Du_t\cdot Dh)\nonumber\\
&=\textrm{tr}\{a(Du)D(v^qDv)\}-\textrm{tr}\{a(Du)D^2(\phi Du\cdot Dh)\}\nonumber\\
&\qquad\qquad\qquad+(-v^{q-1}Du+\phi Dh)\cdot DG+(-v^{q-1}Du+\phi Dh)\cdot D(cv-f)\label{maxprincipleexpandedr}
\end{align}
at $(x_0,t_0)$. Here, we have used the fact that $vv_t=Du\cdot Du_t$.

For the first term $\textrm{tr}\{a(Du)D(v^qDv)\}$ of \eqref{maxprincipleexpandedr}, we substitute $D(v^qDv)=qv^{q-1}Dv\otimes Dv+v^qD^2v$ to get
$$
\textrm{tr}\{a(Du)D(v^qDv)\}=qv^{q-1}\textrm{tr}\{a(Du)Dv\otimes Dv\}+v^q\textrm{tr}\{a(Du)D^2v\}.
$$
We first check that $vD^2v=Q a(Du) D^2u+D^3u\odot Du$ with $Q=D^2u$. Differentiating $vDv=D^2u Du$, and using the fact that $p\otimes q=pq^{\textrm{Tr}}$ for two vectors $p,q$, we get
\begin{align*}
vD^2v&=D^3u\odot Du+(D^2u)^2-Dv\otimes Dv\\
&=(D^2u)^2-\frac{D^2u Du}{v}\otimes\frac{D^2u Du}{v}+D^3u\odot Du\\
&=QI_nQ-Q\left(\frac{Du}{v}\otimes\frac{Du}{v}\right) Q+D^3u\odot Du\\
&=Q a(Du) Q+D^3u\odot Du.
\end{align*}
Therefore,
\begin{align}
\textrm{tr}\{a(Du)D(v^qDv)\}=v^{q-1}\textrm{tr}\{(a(Du)D^2u)^2\}+qV+X_1,\label{term1}
\end{align}
where $V:=v^{q-1}\textrm{tr}\{a(Du)Dv\otimes Dv\}$ and $X_1:=v^{q-1}\textrm{tr}\{a(Du)(D^3u\odot Du)\}$.

To compute the second term of \eqref{maxprincipleexpandedr}, we expand $D^2(\phi Du\cdot Dh)$ so that
\begin{align*}
D^2(\phi Du\cdot Dh)&=(Du\cdot Dh)D^2\phi+(D^2uDh+D^2hDu)\otimes D\phi+D\phi\otimes(D^2uDh+D^2hDu)\\
&\qquad\qquad\qquad\qquad+\phi(D^3u\odot Dh+D^2uD^2h+D^3h\odot Du+D^2hD^2u).
\end{align*}
Since $\textrm{tr}\{a(p)(q\otimes r)\}=\textrm{tr}\{a(p)(r\otimes q)\},\ \textrm{tr}\{a(p)AB\}=\textrm{tr}\{a(p)BA\}$ for $p,q,r\in\mathbb{R}^n$, symmetric matrices $A,B$, we obtain
$$
\textrm{tr}\{a(Du)D^2(\phi Du\cdot Dh)\}=2\textrm{tr}\{a(Du)(D\phi\otimes (D^2uDh))\}+2\phi\textrm{tr}\{a(Du)D^2uD^2h\}+X_2+J_1,
$$
where $X_2:=\phi\textrm{tr}\{a(Du)(D^3u\odot Dh)\}$ and
$$
J_0:=(Du\cdot Dh)\textrm{tr}\{a(Du)D^2\phi\}+2\textrm{tr}\{a(Du)(D\phi\otimes(D^2hDu))\}+\phi\textrm{tr}\{a(Du)(D^3h\odot Du)\}.
$$
Applying Cauchy-Schwarz inequality to the terms of $J_0$, we see that there exists a constant $C>0$ independent of $T\in(0,\infty),\ \eta\in(0,1]$ such that
\begin{align*}
J_0&=(Du\cdot Dh)\textrm{tr}\{a(Du)D^2\phi\}+2\textrm{tr}\{a(Du)(D\phi\otimes(D^2hDu))\}+\phi\textrm{tr}\{a(Du)(D^3h\odot Du)\}\\
&\leq |Du||Dh|\|a\|\|D^2\phi\|+2\|a\||D\phi||D^2h Du|+|\phi|\|a\|\|D^3h\odot Du\|\\
&\leq Cv\|a\|\\
&\leq C\left(\frac{\eta^2}{v}+v\right).
\end{align*}
We have used the fact that $\|a\|=\left(\frac{\eta^4}{v^4}+n-1\right)^{1/2}\leq\frac{\eta^2}{v^2}+n-1$, that $\|p\otimes q\|=|p||q|$ for $p,q\in\mathbb{R}^n$. We also have used the fact that, seen again by Cauchy-Schwarz inequality,
\begin{align*}
|D^2h Du|&=\sqrt{\|(D^2h Du)\otimes (D^2h Du)\|}=\sqrt{\|D^2h Du Du^{\textrm{Tr}} D^2h^{\textrm{Tr}}\|}\\
&\leq\sqrt{\|D^2h\|\|Du Du^{\textrm{Tr}}\|\|D^2h^{\textrm{Tr}}\|}=\|D^2h\||Du|\leq\|D^2h\|v,
\end{align*}
and
\begin{align*}
\|D^3h\odot Du\|=\sqrt{\sum_{i,j=1}^n\left(\sum_{\ell=1}^nh_{ij\ell}u_{\ell}\right)^2}\leq\sqrt{\sum_{i,j=1}^n\left(\sum_{\ell=1}^nh_{ij\ell}^2\right)\left(\sum_{\ell=1}^nu_{\ell}^2\right)}\leq C|Du|\leq Cv,
\end{align*}
where $C>0$ is a constant depending on $\|h\|_{C^3(\overline{\Omega})}$. Since $\eta\in(0,1]$, we see that there exist constants $R>1,\ C>0$ independent of $T\in(0,\infty),\ \eta\in(0,1]$ such that
$$
J_0\leq Cv
$$
whenever $v>R$, and therefore that
\begin{align}
-\textrm{tr}\{a(Du)D^2(\phi Du\cdot Dh)\}&\geq-2\textrm{tr}\{a(Du)(D\phi\otimes (D^2uDh))\}-2\phi\textrm{tr}\{a(Du)D^2uD^2h\}\nonumber\\
&\qquad\qquad\qquad\qquad\qquad\qquad\qquad\qquad-X_2-Cv.\label{term2}
\end{align}
whenever $v>R$.

We compute the third term $(-v^{q-1}Du+\phi Dh)\cdot DG$ of \eqref{maxprincipleexpandedr}. By differentiating $G=\textrm{tr}\{a(Du)D^2u\}$ and taking inner product, we obtain
\begin{align*}
Du\cdot DG&=\textrm{tr}\{(D_pa(Du)\odot(D^2uDu))D^2u\}+\textrm{tr}\{a(Du)(D^3u\odot Du)\}\\
&=v\textrm{tr}\{(D_pa(Du)\odot Dv)D^2u\}+\textrm{tr}\{a(Du)(D^3u\odot Du)\}
\end{align*}
and
\begin{align*}
Dh\cdot DG=\textrm{tr}\{(D_pa(Du)\odot(D^2uDh))D^2u\}+\textrm{tr}\{a(Du)(D^3u\odot Dh)\}.
\end{align*}
Therefore,
\begin{align}
(-v^{q-1}Du+\phi Dh)\cdot DG&=-v^q\textrm{tr}\{(D_pa(Du)\odot Dv)D^2u\}\nonumber\\
&\ \ \ \ \ +\phi\textrm{tr}\{(D_pa(Du)\odot(D^2uDh))D^2u\}-X_1+X_2.\label{term3}
\end{align}
Recall that $X_1=v^{q-1}\textrm{tr}\{a(Du)(D^3u\odot Du)\}$ and $X_2=\phi\textrm{tr}\{a(Du)(D^3u\odot Dh)\}$.

Now, we compute and estimate the fourth term $(-v^{q-1}Du+\phi Dh)\cdot D(cv-f)$ of \eqref{maxprincipleexpandedr}. By expansion,
\begin{align*}
(-v^{q-1}Du+\phi Dh)\cdot D(cv-f)&=(-c_zv+f_z)(v^{q-1}|Du|^2-\phi Du\cdot Dh)\\
&\ \ \ +(-v^{q-1}Du+\phi Dh)\cdot(vDc-Df)+cDv\cdot(-v^{q-1}Du+Dh).
\end{align*}
Since $\eta\in(0,1]$, there exist constants $R>1,\ C>0$ independent of $T\in(0,\infty),\ \eta\in(0,1]$ such that
$$
v^{q-1}|Du|^2-\phi Du\cdot Dh\geq v^{q+1}-\eta^2v^{q-1}-\|\phi\|_{C^{0}(\overline{\Omega})}\|h\|_{C^{1}(\overline{\Omega})}v\geq0
$$
if $v>R$, and therefore that
$$
(-c_zv+f_z)(v^{q-1}|Du|^2-\phi Du\cdot Dh)\geq0
$$
if $v>R$. Here, we have used the assumption that $c_z\leq0,\ f_z\geq0$ from \eqref{assumtion:c}, \eqref{assumtion:f}. Also, again by \eqref{assumtion:c}, \eqref{assumtion:f}, there exists a constant $C>0$ independent of $T\in(0,\infty),\ \eta\in(0,1]$ such that
\begin{align*}
(-v^{q-1}Du+\phi Dh)\cdot(vDc-Df)&\geq-v^{q}|Du||Dc|-v^{q-1}|Du|\|Df\|_{C^0(\overline{\Omega}\times\mathbb{R})}\\
&\qquad-\|h\|_{C^{1}(\overline{\Omega})}|Dc|v-\|h\|_{C^{1}(\overline{\Omega})}\|Df\|_{C^0(\overline{\Omega}\times\mathbb{R})}\\
&\geq-|Dc|v^{q+1}-C(v+v^q).
\end{align*}
Therefore, there exist constants $R>1,\ C>0$ independent of $T\in(0,\infty),\ \eta\in(0,1]$ such that at $(x_0,t_0)$
\begin{align}
(-v^{q-1}Du+\phi Dh)\cdot D(cv-f)&\geq-|Dc|v^{q+1}-C(v+v^q)\nonumber\\
&\qquad\qquad\qquad+cDv\cdot(-v^{q-1}Du+\phi Dh)\label{term4}
\end{align}
whenever $v>R$. We will give a bound of the term $cDv\cdot(-v^{q-1}Du+\phi Dh)$ at $(x_0,t_0)$ later.

All in all, by the estimates \eqref{term1}, \eqref{term2}, \eqref{term3}, \eqref{term4}, we obtain that there exist constants $R>1,\ C>0$ independent of $T\in(0,\infty),\ \eta\in(0,1]$ such that at $(x_0,t_0)$,
\begin{align}
0&\geq\frac{1}{q+1}\left(\textrm{tr}\{a(Du)D^2w\}-w_t\right)\nonumber\\
&\geq J_1+J_2-|Dc|v^{q+1}+(q+1-\varepsilon)V-C(v+v^q)\label{maxprinciplesimplifiedr}
\end{align}
if $v>R$, where
\begin{align*}
J_1&:=(1-\varepsilon)v^{q-1}\textrm{tr}\{(a(Du)D^2u)^2\}-\frac{1}{2}v^q\textrm{tr}\{(D_pa(Du)\odot Dv)D^2u\}\\
&\qquad\qquad\qquad\qquad\qquad\qquad\qquad\qquad\qquad+cDv\cdot(-v^{q-1}Du+\phi Dh)\\
J_2&:=\varepsilon v^{q-1}\textrm{tr}\{(a(Du)D^2u)^2\}-\frac{1}{2}\varepsilon v^q\textrm{tr}\{(D_pa(Du)\odot Dv)D^2u\}\\
&\qquad\qquad-2\textrm{tr}\{a(Du)(D\phi\otimes (D^2uDh))\}-2\phi\textrm{tr}\{a(Du)D^2uD^2h\}\\
&\qquad\qquad\qquad\qquad\qquad\qquad\qquad\qquad+\phi\textrm{tr}\{(D_pa(Du)\odot(D^2uDh))D^2u\}.
\end{align*}
Here, $\varepsilon\in(0,1)$ is a number to be determined, and we have used the fact, from Lemma \ref{lem:afternotation} with $\xi=Dv$, that
$$
-\frac{1}{2}v^q\textrm{tr}\{(D_pa(Du)\odot Dv)D^2u\}=v^{q-1}\textrm{tr}\{a(Du)Dv\otimes Dv=V.
$$

\medskip

\emph{Step 2.} We estimate $J_1$.

\medskip
We first write, with $Q=D^2u$,
\begin{align*}
\textrm{tr}\{(a(Du)D^2u)^2\}&=\textrm{tr}\left\{\left(I_n-\frac{Du\otimes Du}{v^2}\right)Qa(Du)Q\right\}\\
&=\textrm{tr}\{a(Du)Q^2\}-\textrm{tr}\left\{a(Du)\left(\frac{D^2u Du}{v}\otimes\frac{D^2u Du}{v}\right)\right\}\\
&=\textrm{tr}\{a(Du)(D^2u)^2\}-\textrm{tr}\{a(Du)Dv\otimes Dv\}.
\end{align*}
Apply Cauchy-Schwarz inequality $\|\alpha\|^2\|\beta\|^2\geq\textrm{tr}\{\alpha\beta^{\textrm{Tr}}\}^2$ for $\textrm{tr}\{a(Du)(D^2u)^2\}$ with $\alpha=\sqrt{a} D^2u$, $\beta=\sqrt{a}$ to obtain
\begin{align*}
\textrm{tr}\{a(Du)(D^2u)^2\}=\|\alpha\|^2&\geq\frac{\textrm{tr}\{\alpha\beta^{\textrm{Tr}}\}^2}{\|\beta\|^2}=\frac{G^2}{n-1+\frac{\eta^2}{v^2}}\nonumber\\
&=\left(\frac{1}{n-1}-\frac{\eta^2}{v^2(n-1)\left(n-1+\frac{\eta^2}{v^2}\right)}\right)(u_t-cv+f)^2\\
&\geq\frac{1}{n-1}c^2v^2-Cv
\end{align*}
for some constant $C>0$ depending only on $\|f\|_{C^0(\overline{\Omega}\times\mathbb{R})},\|c\|_{C^0(\overline{\Omega}\times\mathbb{R})}$ and $M>0$ in Lemma \ref{lem:time-grad}. We have used Lemma \ref{lem:time-grad}, the assumptions \eqref{assumtion:c}, \eqref{assumtion:f} and the fact that $\eta\in(0,1]$. Therefore, there exist constants $R>1,\ C>0$ independent of $T\in(0,\infty),\ \eta\in(0,1]$ such that
$$
\textrm{tr}\{a(Du)(D^2u)^2\}\geq\frac{1}{n-1}c^2v^2-Cv
$$
if $v>R$, and thus such that
\begin{align}
(1-\varepsilon)v^{q-1}\textrm{tr}\{(a(Du)D^2u)^2\}\geq\frac{1-\varepsilon}{n-1}c^2v^{q+1}-(1-\varepsilon)V-Cv^q\label{term1J1}
\end{align}
if $v>R$, $\varepsilon\in(0,1)$. The number $\varepsilon\in(0,1)$ will be explicitly chosen later. We note that the term $\textrm{tr}\{(a(Du)D^2u)^2\}$ is used to derive the term $\frac{1}{n-1}c^2v^{q+1}$ as a lower bound, which is crucial to obtain the bound $v\leq R$.

For the third term of $J_1$, we claim that at $(x_0,t_0)$, it holds that
\begin{align}
|cDv\cdot(-v^{q-1}Du+\phi Dh)|\leq Cv\label{term3J1}
\end{align}
for some constant $C>0$ independent of $T\in(0,\infty),\ \eta\in(0,1]$. Note that $Dw=0$ at $(x_0,t_0)$, so that
\begin{align*}
0&=\frac{1}{q+1}Dw\cdot Du\\
&=v^qDu\cdot Dv-(Du\cdot D\phi)(Du\cdot Dh)-\phi (D^2uDu)\cdot Dh-\phi (D^2hDu)\cdot Du.
\end{align*}
This implies that at $(x_0,t_0)$,
$$
cDv\cdot(-v^{q-1}Du+\phi Dh)=-\frac{c}{v}\left((Du\cdot D\phi)(Du\cdot Dh)+\phi (D^2hDu)\cdot Du\right),
$$
and thus that at $(x_0,t_0)$,
\begin{align*}
|cDv\cdot(-v^{q-1}Du+\phi Dh)|&\leq \|c\|_{C^0(\overline{\Omega}\times\mathbb{R})}\left(\|\phi\|_{C^{1}(\overline{\Omega})}\|h\|_{C^{1}(\overline{\Omega})}+\|h\|_{C^{2}(\overline{\Omega})}\|\phi\|_{C^{0}(\overline{\Omega})}\right)\frac{1}{v}|Du|^2\\
&\leq Cv
\end{align*}
for some constant $C>0$ independent of $T\in(0,\infty),\ \eta\in(0,1]$. We have used the fact that $|Du|\leq v$ and the assumptions \eqref{assumtion:c}, \eqref{assumtion:f}.

Together with the fact that
\begin{align*}
-\frac{1}{2}v^q\textrm{tr}\{(D_pa(Du)\odot Dv)D^2u\}=v^{q-1}\textrm{tr}\{a(Du)Dv\otimes Dv\}=V,
\end{align*}
and with \eqref{term1J1}, \eqref{term3J1}, we conclude that there exist constants $R>1,\ C>0$ independent of $T\in(0,\infty),\ \eta\in(0,1]$ such that
\begin{align}
J_1\geq\frac{1-\varepsilon}{n-1}c^2v^{q+1}+\varepsilon V-C(v+v^q)\label{J1r}
\end{align}
if $v>R$.

\medskip

\emph{Step 3.} We estimate $J_2$.

\medskip

Before we start the estimate of $J_2$, we rotate the axes at $x_0$ and compute the second derivatives of $u$ with respect to these axes. Take axes at $x_0$ such that
\begin{equation}\label{rotation1}
u_1=|Du|,\quad\quad u_i=0,\ i=2,\cdots,n,\quad\quad (u_{ij})_{2\leq i,j\leq n}\text{ is diagonal}.
\end{equation}
Then, $a^{ij}=a^{ij}(Du)$ is simplified as
\begin{equation}\label{rotation2}
a^{11}=\frac{\eta^2}{v^2},\quad\quad a^{ii}=1,\ i=2,\cdots n,\quad\quad a^{ij}=0,\ i\neq j.
\end{equation}
Using $Dw=0$ at $(x_0,t_0)$, we obtain, at $(x_0,t_0)$,
$$
v^{q-1}u_1u_{1i}-\phi \sum_{\ell=1}^nu_{\ell i}h_{\ell}=(\phi_ih_1+\phi h_{1i})u_1,\quad i=1,\cdots, n.
$$
For $i\geq2,$
$$
v^{q-1}u_1u_{1i}-\phi u_{1i}h_1-\phi u_{ii}h_i=(\phi_ih_1+\phi h_{1i})u_1,\quad i=2,\cdots, n,
$$
and thus,
\begin{align}
u_{1i}=E_iu_1+F_iu_{ii},\qquad i=2,\cdots,n,\label{secondderivative1}
\end{align}
where
\begin{align}
E_i:=\frac{\phi_ih_1+\phi h_{1i}}{v^{q-1}u_1-\phi h_1},\ \ \  F_i:=\frac{\phi h_{i}}{v^{q-1}u_1-\phi h_1},\qquad i=2,\cdots,n.\label{E_iF_i}
\end{align}
For $i=1$,
$$
v^{q-1}u_1u_{11}-\phi h_1u_{11}-\phi\sum_{\ell=2}^nh_{\ell}u_{1\ell}=(\phi_1h_1+\phi h_{11})u_1.
$$
As above, we get
\begin{align}
u_{11}&=E_1u_1+\sum_{\ell=2}^nF_{\ell}^2u_{\ell\ell},\label{secondderivative21}
\end{align}
where
\begin{align}
E_1:=\frac{\phi_1h_1+\phi h_{11}}{v^{q-1}u_1-\phi h_1}+\frac{\phi}{v^{q-1}u_1-\phi h_1}\sum_{{\ell}=2}^nh_{\ell}E_{\ell}.\label{E_1}
\end{align}

Now, we write $J_2=\varepsilon v^{q-1}\textrm{tr}\{(a(Du)D^2u)^2\}+S_1+S_2$, where
\begin{align*}
S_1&:=-2\textrm{tr}\{a(Du)(D\phi\otimes (D^2uDh))\}-2\phi\textrm{tr}\{a(Du)D^2uD^2h\},\\
S_2&:=-\frac{1}{2}\varepsilon v^q\textrm{tr}\{(D_pa(Du)\odot Dv)D^2u\}+\phi\textrm{tr}\{(D_pa(Du)\odot(D^2uDh))D^2u\},
\end{align*}
and we bound $S_1,\ S_2$.

We start with $S_1$. By expansion,
\begin{align*}
S_1&=-2\left(\frac{\eta^2}{v^2}\phi_1Du_1\cdot Dh+\sum_{\ell=2}^n\phi_{\ell}Du_{\ell}\cdot Dh+\frac{\eta^2}{v^2}\phi Du_1\cdot Dh_1+\phi\sum_{\ell=2}^nDu_{\ell}\cdot Dh_{\ell}\right)\\
&=-2\left(\frac{\eta^2}{v^2}Du_1\cdot(\phi_1Dh+\phi Dh_1)+\sum_{\ell=2}^nDu_{\ell}\cdot(\phi_{\ell}Dh+\phi Dh_{\ell})\right).
\end{align*}
Let $H_{\ell i}:=\phi_{\ell}h_i+\phi h_{\ell i}$ for each $\ell,i=1,\cdots,n$. Then,
\begin{align*}
S_1&=-2\left(\frac{\eta^2}{v^2}\sum_{\ell=1}^nu_{1\ell}H_{1\ell}+\sum_{\ell=2}^n(u_{1\ell}H_{\ell1}+u_{\ell\ell}H_{\ell\ell})\right)\\
&=-2\left(\frac{\eta^2}{v^2}u_{11}H_{11}+\sum_{\ell=2}^nu_{1\ell}\left(\frac{\eta^2}{v^2}H_{1\ell}+H_{\ell1}\right)+\sum_{\ell=2}^nu_{\ell\ell}H_{\ell\ell}\right).
\end{align*}
Using \eqref{secondderivative1}, \eqref{secondderivative21}, we get
\begin{align*}
S_1&=-2\left(\left(\frac{\eta^2}{v^2}H_{11}E_{1}+\sum_{\ell=2}^n\left(\frac{\eta^2}{v^2}H_{1\ell}+H_{\ell1}\right)E_{\ell}\right)u_1\right.\\
&\left.\qquad\qquad\qquad+\sum_{\ell=2}^n\left(\frac{\eta^2}{v^2}H_{11}F_{\ell}^2+\left(\frac{\eta^2}{v^2}H_{1\ell}+H_{\ell1}\right)F_{\ell}+H_{\ell\ell}\right)u_{\ell\ell}\right)
\end{align*}
Note that since $\eta\in(0,1]$,
$$
\left|\frac{\eta^2}{v^2}H_{11}E_{1}+\sum_{\ell=2}^n\left(\frac{\eta^2}{v^2}H_{1\ell}+H_{\ell1}\right)E_{\ell}\right|\leq Cv^{-q},
$$
$$
\left|\frac{\eta^2}{v^2}H_{11}F_{\ell}^2+\left(\frac{\eta^2}{v^2}H_{1\ell}+H_{\ell1}\right)F_{\ell}+H_{\ell\ell}\right|\leq C
$$
for $v>1$, for some constant $C>0$ that depends only on $\|\phi\|_{C^1(\overline{\Omega})},\ \|h\|_{C^2(\overline{\Omega})}$. Therefore, there exist constants $R>1,\ C>0$ independent of $T\in(0,\infty),\ \eta\in(0,1]$ such that
\begin{align}
S_1\geq-C\left(v^{1-q}+\sum_{\ell=2}^n|u_{\ell\ell}|\right)\label{S1}
\end{align}
for $v>R$.

Now, we estimate $S_2$. Applying Lemma \ref{lem:afternotation} with $\xi=Dv$ and with $\xi=D^2uDh$, and by expansion, we see that
\begin{align*}
S_2&=\varepsilon v^{q-1}\textrm{tr}\{a(Du)Dv\otimes Dv\}-\frac{2\phi}{v}\textrm{tr}\{a(Du)(D^2uDh\otimes Dv)\}\\
&=\varepsilon v^{q-1}\left(\frac{\eta^2}{v^2}v_1^2+\sum_{\ell=2}^nv_{\ell}^2\right)-\frac{2\phi}{v}\left(\frac{\eta^2}{v^2}v_1(Du_1\cdot Dh)+\sum_{\ell=2}^nv_{\ell}(Du_{\ell}\cdot Dh)\right).
\end{align*}
Let $K_{\ell}:=\phi v^{-1}(Du_{\ell}\cdot Dh)$ for each $\ell=1,\cdots,n$. Then,
\begin{align*}
S_2&=\frac{\eta^2}{v^2}(\varepsilon v^{q-1}v_1^2-2K_1v_1)+\sum_{\ell=2}^n(\varepsilon v^{q-1}v_{\ell}^2-2K_{\ell}v_{\ell})\\
&=\frac{\eta^2}{v^2}\left(\varepsilon v^{q-1}\left(v_1-\frac{K_1}{\varepsilon v^{q-1}}\right)^2-\frac{K_1^2}{\varepsilon v^{q-1}}\right)+\sum_{\ell=2}^n\left(\varepsilon v^{q-1}\left(v_{\ell}-\frac{K_{\ell}}{\varepsilon v^{q-1}}\right)^2-\frac{K_{\ell}^2}{\varepsilon v^{q-1}}\right)\\
&\geq-\varepsilon^{-1}v^{-1-q}K_1^2-\varepsilon^{-1}v^{1-q}\sum_{\ell=2}^nK_{\ell}^2.
\end{align*}
In the last inequality, we have used the fact that $\eta\in(0,1]$. By expansion and \eqref{secondderivative1}, \eqref{secondderivative21}, we have
\begin{align*}
K_1=\phi v^{-1}\left(h_1u_{11}+\sum_{\ell=2}^nh_{\ell}u_{1\ell}\right)=K_{11}u_1+\sum_{\ell=2}^nK_{1\ell}u_{\ell\ell},
\end{align*}
where
\begin{align}
K_{11}:=\phi v^{-1}\sum_{\ell=1}^nh_{\ell}E_{\ell},\ \ \ K_{1\ell}:=\phi v^{-1}(h_1F_{\ell}^2+h_{\ell}F_{\ell}),\ \ \ \textrm{for }\ell=2,\cdots,n.\label{K_11K_1l}
\end{align}
Then, by Cauchy-Schwarz inequality,
\begin{align*}
-K_1^2=-\left(K_{11}u_1+\sum_{\ell=2}^nK_{1\ell}u_{\ell\ell}\right)^2\geq-nK_{11}^2u_1^2-\sum_{\ell=2}^nnK_{1\ell}^2u_{\ell\ell}^2.
\end{align*}
Similarly, it holds that, for $\ell=2,\cdots,n$,
\begin{align*}
K_{\ell}=\phi v^{-1}\left(h_1u_{1\ell}+h_{\ell}u_{\ell\ell}\right)=K_{\ell1}u_1+K_{\ell\ell}u_{\ell\ell},
\end{align*}
where
\begin{align}
K_{\ell1}:=\phi v^{-1}h_1E_{\ell},\ \ \ K_{\ell\ell}:=\phi v^{-1}(h_1F_{\ell}+h_{\ell}),\ \ \ \textrm{for }\ell=2,\cdots,n,\label{K_l1K_ll}
\end{align}
and, by Cauchy-Schwarz inequality, for $\ell=2,\cdots,n$,
\begin{align*}
-K_{\ell}^2=-\left(K_{\ell1}u_1+K_{\ell\ell}u_{\ell\ell}\right)^2\geq-2K_{\ell1}^2u_1^2-2K_{\ell\ell}^2u_{\ell\ell}^2.
\end{align*}
Therefore,
\begin{align}
S_2&\geq-\varepsilon^{-1}v^{-1-q}K_1^2-\varepsilon^{-1}v^{1-q}\sum_{\ell=2}^nK_{\ell}^2\nonumber\\
&\geq-\varepsilon^{-1}S_{21}u_{1}^2-\varepsilon^{-1}\sum_{\ell=2}^nS_{2\ell}u_{\ell\ell}^2,\label{S2}
\end{align}
where
\begin{align}
S_{21}:=nv^{-1-q}K_{11}^2+2v^{1-q}\sum_{\ell=2}^nK_{\ell1}^2,\ S_{2\ell}:=nv^{-1-q}K_{1\ell}^2+2v^{1-q}K_{\ell\ell}^2,\ \ \textrm{for }\ell=2,\cdots,n.\label{S_21S_2l}
\end{align}

By \eqref{S1}, \eqref{S2}, we see that there exist constants $R>1,\ C>0$ independent of $T\in(0,\infty),\ \eta\in(0,1]$ such that
$$
J_2\geq\varepsilon v^{q-1}\sum_{\ell=2}^nu_{\ell\ell}^2-C(v^{1-q}+\sum_{\ell=2}^n|u_{\ell\ell}|)-\varepsilon^{-1}S_{21}u_1^2-\varepsilon^{-1}\sum_{\ell=2}^nS_{2\ell}u_{\ell\ell}^2
$$
for $v>R$. Note that there exists a constant $C>0$ that depends only on $\|\phi\|_{C^1(\overline{\Omega})},\ \|h\|_{C^2(\overline{\Omega})}$ such that, by \eqref{E_iF_i}, \eqref{E_1},
$$
|E_1|+\sum_{\ell=2}^n|E_{\ell}|+\sum_{\ell=2}^n|F_{\ell}|\leq Cv^{-q},
$$
for $v>1$, and in turn, by \eqref{K_11K_1l}, \eqref{K_l1K_ll}
$$
|K_{11}|+\sum_{\ell=2}^n|K_{1\ell}|+\sum_{\ell=2}^n|K_{\ell1}|+\sum_{\ell=2}^n|K_{\ell\ell}|\leq Cv^{-1-q},
$$
for $v>1$, and thus such that, by \eqref{S_21S_2l}
$$
|S_{21}|+\sum_{\ell=2}^n|S_{2\ell}|\leq Cv^{-1-3q}.
$$
Therefore, there exist constants $R>1,\ C>0$ independent of $T\in(0,\infty),\ \eta\in(0,1]$ such that
\begin{align*}
J_2&\geq\varepsilon v^{q-1}\sum_{\ell=2}^nu_{\ell\ell}^2-C(v^{1-q}+\sum_{\ell=2}^n|u_{\ell\ell}|)-C\varepsilon^{-1}v^{1-3q}-C\varepsilon^{-1}\sum_{\ell=2}^nv^{-1-3q}u_{\ell\ell}^2\\
&\geq-Cv^{1-q}-C\varepsilon^{-1}v^{1-3q}+\sum_{\ell=2}^n\left(v^{q-1}\left(\varepsilon-C\varepsilon^{-1}v^{-4q}\right)u_{\ell\ell}^2-C|u_{\ell\ell}|\right)
\end{align*}
for $v>R$.

For a given $\varepsilon\in(0,1)$, choose $R_{\varepsilon}>1$ that may depend on $\varepsilon\in(0,1)$ but not on $T\in(0,\infty),\ \eta\in(0,1]$ such that $C\varepsilon^{-1}v^{-4q}<\frac{\varepsilon}{2}$ if $v>R_{\varepsilon}$. Then, for $v>R_{\varepsilon}$,
\begin{align*}
J_2&\geq-Cv^{1-q}-C\varepsilon^{-1}v^{1-3q}+\sum_{\ell=2}^n\left(\frac{\varepsilon}{2}v^{q-1}u_{\ell\ell}^2-C|u_{\ell\ell}|\right)\\
&=-Cv^{1-q}-C\varepsilon^{-1}v^{1-3q}+\sum_{\ell=2}^n\left(\frac{\varepsilon}{2}v^{q-1}\left(|u_{\ell\ell}|-\frac{C}{\varepsilon v^{q-1}}\right)^2-\frac{C^2}{2\varepsilon v^{q-1}}\right).
\end{align*}

All in all, for each $\varepsilon\in(0,1)$, there exist a constant $C>0$ independent of $T\in(0,\infty),\ \eta\in(0,1]$ (also of $\varepsilon\in(0,1)$) and a constant $R_{\varepsilon}>1$ that may depend on $\varepsilon\in(0,1)$ but not on $T\in(0,\infty),\ \eta\in(0,1]$ such that
\begin{align}
J_2\geq-Cv^{1-q}(1+\varepsilon^{-1})\label{J2r}
\end{align}
for $v>R_{\varepsilon}$.

\medskip

\emph{Step 4.} We finish Case 1.

\medskip

We come back to the maximum principle \eqref{maxprinciplesimplifiedr} applied at $(x_0,t_0)$. By \eqref{J1r}, \eqref{J2r}, we see that there exist a constant $C>0$ independent of $T\in(0,\infty),\ \eta\in(0,1]$ (also of $\varepsilon\in(0,1)$) and a constant $R_{\varepsilon}>1$ that may depend on $\varepsilon\in(0,1)$ but not on $T\in(0,\infty),\ \eta\in(0,1]$ such that
$$
0\geq\left(\frac{1-\varepsilon}{n-1}c^2-|Dc|\right)v^{q+1}+(q+1)V-C(v+v^q)-Cv^{1-q}\varepsilon^{-1}
$$
for $v>R_{\varepsilon}$.
Now, we apply the condition \eqref{condition:c}; take
$$
\varepsilon=\frac{1}{2}\min\left\{1,\frac{(n-1)\delta}{\|c\|^2_{L^{\infty}(\overline{\Omega}\times\mathbb{R})}}\right\}.
$$
For this choice of $\varepsilon\in(0,1)$, it holds that $\frac{1-\varepsilon}{n-1}c^2-|Dc|\geq\frac{1}{2}\delta$, and $R_{\varepsilon},\ \varepsilon^{-1}$ are fixed. Therefore, by taking this choice of $\varepsilon\in(0,1)$, we see that there exist constants $R>1,\ C>0$ independent of $T\in(0,\infty),\ \eta\in(0,1]$ such that at $(x_0,t_0)$,
$$
0\geq\frac{\delta}{2}v^{q+1}-C(v+v^q)
$$
if $v>R$. Here, we have used the fact that $V\geq0$. On the other hand, there is also a constant $R_0>R$ independent of $T\in(0,\infty),\ \eta\in(0,1]$ such that
$$
0<\frac{\delta}{2}v^{q+1}-C(v+v^q)
$$
if $v>R_0$. Therefore, it must hold that $v=v(x_0,t_0)\leq R_0$, which completes Case 1.

\medskip

\noindent {\bf Case 2: $x_0\in\partial\Omega$.}

\medskip

\emph{Step 5.} We bound the normal derivative of $w$ at $(x_0,t_0)$ with a geometric constant.

\medskip

Recall that $C_0(x_0)=\max\{\lambda:\lambda\ \textrm{is an eigenvalue of}\ -\kappa\}$, where $\kappa:=\left(\kappa^{\ell j}\right)_{\ell,j=1}^{n-1}$ is the curvature matrix of $\partial\Omega$ at $x_0$, and that $C_0=\sup\{C_0(y):y\in\partial\Omega\}$. For $\varepsilon_0\in(0,1)$, we let $L=(q+1)\left(C_0+\varepsilon_0\right)$. The goal of this step is to prove that for any given number $\varepsilon_0\in(0,1)$, there exists a constant $R_{\varepsilon_0}>0$ which depends on $\varepsilon_0$ but not on $T\in(0,\infty),\ \eta\in(0,1]$ (also not on $x_0\in\partial\Omega$) such that $w>0$ and $\frac{\partial w}{\partial \Vec{\mathbf{n}}}<Lw$ at $(x_0,t_0)$ whenever $v>R_{\varepsilon_0}$.


Changing a coordinate on $\mathbb{R}^n$, we may assume without loss of generality that $x_0=(0,\cdots,0)$, $\Vec{\mathbf{n}}(x_0)=(0,\cdots,0,-1)$, and that $x'=(x_1,\cdots,x_{n-1})$ is a principal coordinate system of $\partial\Omega$ at $x_0$. We may assume that the $x_{\ell}-$axis lies along a principal direction corresponding to $\kappa_{\ell}$, $\ell=1,\cdots,n-1$, respectively. By Lemma \ref{lem:boundary}, there are open neighborhoods $U,V$ of $(0,\cdots,0)$ in $\mathbb{R}^n$ and a $C^2$ diffeomorphism $g:U\to V$, and there is a number $\sigma>0$ satisfying the properties (i), $\cdots$, (iv) of Lemma \ref{lem:boundary}. For each function $\zeta=u,v,w,\phi,h$ on $V\cap\overline{\Omega}$, we define the function $\overline{\zeta}$ on $U\cap g^{-1}(\overline{\Omega})=\{y=(y_1,\cdots,y_n):y\in U,\ y_n\geq0\}$ by $\overline{\zeta}=\zeta\circ g$. We let $y_0=g^{-1}(x_0)$. The different characters $x_0,y_0$ are used to distinguish where they belong to, i.e., the domains $V,U$ of definitions, respectively, though the both are the origin.

We introduce notations to denote vectors and derivatives in $y=(y_1,\cdots,y_n)$. For a $C^1$ function $\overline{\zeta}$ defined on $U$, let
\begin{align*}
\nabla\overline{\zeta}:=\left(\frac{\partial\overline{\zeta}}{\partial y_1},\cdots,\frac{\partial\overline{\zeta}}{\partial y_n}\right)^{\textrm{Tr}},\ \ \ \nabla'\overline{\zeta}:=\left(\frac{\partial\overline{\zeta}}{\partial y_1},\cdots,\frac{\partial\overline{\zeta}}{\partial y_{n-1}}\right)^{\textrm{Tr}},
\end{align*}
and for the $C^1$ function $\zeta:=\overline{\zeta}\circ g^{-1}$ on $V$, let
\begin{align*}
D\zeta:=\left(\frac{\partial\zeta}{\partial x_1},\cdots,\frac{\partial\zeta}{\partial x_n}\right)^{\textrm{Tr}},\ \ \ D'\zeta:=\left(\frac{\partial\zeta}{\partial x_1},\cdots,\frac{\partial\zeta}{\partial x_{n-1}}\right)^{\textrm{Tr}}.
\end{align*}
If $\zeta$ is a $C^2$ function on $V$, we let
\begin{align*}
\frac{\partial}{\partial y_n}\left(\nabla\overline{\zeta}\right)&:=\left(\frac{\partial}{\partial y_n}\left(\frac{\partial\overline{\zeta}}{\partial y_1}\right),\cdots,\frac{\partial}{\partial y_n}\left(\frac{\partial\overline{\zeta}}{\partial y_n}\right)\right)^{\textrm{Tr}},\\
\frac{\partial}{\partial y_n}\left(\nabla'\overline{\zeta}\right)&:=\left(\frac{\partial}{\partial y_n}\left(\frac{\partial\overline{\zeta}}{\partial y_1}\right),\cdots,\frac{\partial}{\partial y_n}\left(\frac{\partial\overline{\zeta}}{\partial y_{n-1}}\right)\right)^{\textrm{Tr}},\\
\frac{\partial}{\partial y_n}\left(D\zeta\right)&:=\left(\frac{\partial}{\partial y_n}\left(\frac{\partial\zeta}{\partial x_1}\right),\cdots,\frac{\partial}{\partial y_n}\left(\frac{\partial\zeta}{\partial x_n}\right)\right)^{\textrm{Tr}},\\
\frac{\partial}{\partial y_n}\left(D'\zeta\right)&:=\left(\frac{\partial}{\partial y_n}\left(\frac{\partial\zeta}{\partial x_1}\right),\cdots,\frac{\partial}{\partial y_n}\left(\frac{\partial\zeta}{\partial x_{n-1}}\right)\right)^{\textrm{Tr}}.
\end{align*}
We use the same notation, $\cdot$, for the inner product in $\mathbb{R}^n$, now including the vectors in $\mathbb{R}^n$ just introduced above. By abuse of notations, we use the notation, $\cdot$, for the inner product in $\mathbb{R}^{n-1}$, also including the above vectors in $\mathbb{R}^{n-1}$ just introduced.
We write the curvature $\kappa$ as
$$
\kappa=
\begin{bmatrix}
\kappa_{1} & &0 \\
& \ddots & \\
0& & \kappa_{n-1}
\end{bmatrix},
$$
and we let
$$
\widetilde{\kappa}=
\begin{bmatrix}
\kappa_1 &  &  &0 \\ 
& \ddots &  & \\ 
&  &  \kappa_{n-1} & \\ 
0&  &   & \kappa_n 
\end{bmatrix},
$$
with $\kappa_n=0$ for convenience for later.

With the above notations, Lemma \ref{lem:boundary} states that
$$
\nabla\overline{\zeta}=(I_n-y_n\widetilde{\kappa})D\zeta,
$$
and
$$
\frac{\partial}{\partial y_n}(\nabla\overline{\zeta})=(I_n-y_n\widetilde{\kappa})\frac{\partial}{\partial y_n}(D\zeta)-(I_n-y_n\widetilde{\kappa})^{-1}\widetilde{\kappa}\nabla\overline{\zeta}
$$
on the line $\{(0,\cdots,0,y_n)\in U:0\leq y_n<\sigma\}$, in the setting of Lemma \ref{lem:boundary}.

We start the estimate of $\frac{\partial w}{\partial \Vec{\mathbf{n}}}(x_0,t_0)$. In order to estimate $\frac{\partial w}{\partial \Vec{\mathbf{n}}}(x_0,t_0)(=-\frac{\partial w}{\partial x_n}(x_0,t_0)=-\frac{\partial\overline{w}}{\partial y_n}(y_0,t_0))$, we first compute $\frac{\partial \overline{v}}{\partial y_n}$, $\nabla'\overline{v}$, $\nabla'\overline{u}\cdot\frac{\partial}{\partial y_n}(\nabla'\overline{u})$ in turn. Note that for the normal derivatives, we have the additional negative sign, since $\Vec{\mathbf{n}}(x_0)$ denotes the outward unit normal vector at $x_0$, while the inward unit normal vector at $x_0$ and the inward unit normal vector at $y_0$ lie on the positive $x_n-$axis and the positive $y_n-$axis, respectively.

To compute $\frac{\partial \overline{v}}{\partial y_n}$, we differentiate $\overline{v}^2=\eta^2+|Du|^2$ on the line $\{(0,\cdots,0,y_n):0\leq y_n<\sigma\}$ in $y_n$ to obtain
\begin{align*}
2\overline{v}\frac{\partial\overline{v}}{\partial y_n}&=2\frac{\partial}{\partial y_n}\left((I_n-y_n\widetilde{\kappa})^{-1}\nabla\overline{u}\right)\cdot(I_n-y_n\widetilde{\kappa})^{-1}\nabla\overline{u}\\
&=2\left((I_n-y_n\widetilde{\kappa})^{-3}\widetilde{\kappa}\nabla\overline{u}\right)\cdot\nabla\overline{u}+2\left((I_n-y_n\widetilde{\kappa})^{-2}\frac{\partial}{\partial y_n}(\nabla\overline{u})\right)\cdot\nabla\overline{u}
\end{align*}
on the line $\{(0,\cdots,0,y_n):0\leq y_n<\sigma\}$. Since $\frac{\partial\overline{u}}{\partial y_n}=-\overline{\phi}\overline{v}^{1-q}$ at $(y_0,t_0)$ and $\kappa_n=0$, we obtain
\begin{align}\label{v_n}
\frac{\partial\overline{v}}{\partial y_n}=\frac{1}{\overline{v}}\frac{\partial}{\partial y_n}(\nabla'\overline{u})\cdot\nabla'\overline{u}-\overline{\phi}\overline{v}^{-q}\frac{\partial^2\overline{u}}{\partial y_n^2}+\frac{1}{\overline{v}}(\kappa\nabla'\overline{u})\cdot\nabla'\overline{u}.
\end{align}
at $(y_0,t_0)$.

We compute $\nabla'\overline{v}$ at $(y_0,t_0)$. Since $y_0$ is a maximizer of $\overline{w}(\cdot,t_0)$ on $U\cap g^{-1}(\partial\Omega)=\{y=(y',0)\in U: y'=(y_1,\cdots,y_{n-1})\}$, it holds that $\nabla'\overline{w}(y_0,t_0)=0$. Note also that $\overline{w}=\overline{v}^{q+1}-(q+1)\overline{\phi}^2\overline{v}^{1-q}$ on $\left(U\cap g^{-1}(\partial\Omega)\right)\times\{t_0\}$. Hence, at $(y_0,t_0)$,
\begin{align*}
0=\frac{1}{q+1}\nabla'\overline{w}=\overline{v}^q\nabla'\overline{v}-2\overline{\phi}\overline{v}^{1-q}\nabla'\overline{\phi}-(1-q)\overline{\phi}^2\overline{v}^{-q}\nabla'\overline{v},
\end{align*}
which gives
\begin{align}\label{v_l}
\nabla'\overline{v}=\frac{2\overline{\phi}\overline{v}^{1-q}}{\overline{v}^q-(1-q)\overline{\phi}^2\overline{v}^{-q}}\nabla'\overline{\phi}
\end{align}
at $(y_0,t_0)$. Here, we are assuming $(\overline{v}(y_0,t_0)=)v(x_0,t_0)>\left(|1-q|\|\phi\|^2_{C^0(\partial\Omega)}\right)^{1/2q}$ so that $\overline{v}^q-(1-q)\overline{\phi}^2\overline{v}^{-q}>0$. In the other case when $v=v(x_0,t_0)\leq\left(|1-q|\|\phi\|^2_{C^0(\partial\Omega)}\right)^{1/2q}$, we already achieve our goal.

We compute $\nabla'\overline{u}\cdot\frac{\partial}{\partial y_n}(\nabla'\overline{u})$ before getting into the estimate of $\frac{\partial w}{\partial \Vec{\mathbf{n}}}$ at $(x_0,t_0)$. We differentiate $\frac{\partial\overline{u}}{\partial y_n}=-\overline{\phi}\overline{v}^{1-q}$ on $\left(U\cap g^{-1}(\partial\Omega)\right)\times\{t_0\}$ in $y_{\ell}$, $\ell=1,\cdots,n-1$, to have
\begin{align*}
\frac{\partial}{\partial y_n}\left(\nabla'\overline{u}\right)=\nabla'\left(\frac{\partial\overline{u}}{\partial y_n}\right)=-\overline{v}^{1-q}\nabla'\overline{\phi}-(1-q)\overline{\phi} \overline{v}^{-q}\nabla'\overline{v}.
\end{align*}
By \eqref{v_l}, we obtain
\begin{align}
\nabla'\overline{u}\cdot\frac{\partial}{\partial y_n}\left(\nabla'\overline{u}\right)&=-\overline{v}^{1-q}\nabla'\overline{u}\cdot\nabla'\overline{\phi}-\frac{2(1-q)\overline{\phi}^2\overline{v}^{1-2q}}{\overline{v}^q-(1-q)\overline{\phi}^2\overline{v}^{-q}}\nabla'\overline{u}\cdot\nabla'\overline{\phi}\nonumber\\
&=-\frac{\overline{v}+(1-q)\overline{\phi}^2\overline{v}^{1-2q}}{\overline{v}^q-(1-q)\overline{\phi}^2\overline{v}^{-q}}\nabla'\overline{u}\cdot\nabla'\overline{\phi}.\label{u_lu_nl}
\end{align}

We now estimate $\frac{\partial w}{\partial \Vec{\mathbf{n}}}$ at $(x_0,t_0)$. On the line $\{(0,\cdots,0,y_n):0\leq y_n<\sigma\}$, we have 
\begin{align*}
\frac{1}{q+1}\frac{\partial\overline{w}}{\partial y_n}&=\frac{1}{q+1}\frac{\partial}{\partial y_n}\left(\overline{v}^{q+1}-(q+1)\overline{\phi}(I_n-y_n\widetilde{\kappa})^{-2}\nabla\overline{u}\cdot\nabla\overline{h}\right)\\
&=\overline{v}^q\frac{\partial\overline{v}}{\partial y_n}-\frac{\partial\overline{\phi}}{\partial y_n}(I_n-y_n\widetilde{\kappa})^{-2}\nabla\overline{u}\cdot\nabla\overline{h}-2\overline{\phi}(I_n-y_n\widetilde{\kappa})^{-3}\widetilde{\kappa}\nabla\overline{u}\cdot\nabla\overline{h}\\
&\quad\quad-\overline{\phi}(I_n-y_n\widetilde{\kappa})^{-2}\frac{\partial}{\partial y_n}(\nabla\overline{u})\cdot\nabla\overline{h}-\overline{\phi}(I_n-y_n\widetilde{\kappa})^{-2}\nabla\overline{u}\cdot\frac{\partial}{\partial y_n}(\nabla\overline{h}).
\end{align*}
Note that $\kappa_n=0$ and that $\nabla'\overline{h}=0$, $\frac{\partial\overline{h}}{\partial y_n}=-1$ at $(y_0,t_0)$. Also, $\nabla'\left(\frac{\partial\overline{h}}{\partial y_n}\right)=0$ on $U\cap g^{-1}(\partial\Omega)$ since $\frac{\partial\overline{h}}{\partial y_n}=-1$ on $U\cap g^{-1}(\partial\Omega)$. Therefore, at $(y_0,t_0)$, we get
\begin{align*}
\frac{1}{q+1}\frac{\partial\overline{w}}{\partial y_n}
=\overline{v}^q\frac{\partial\overline{v}}{\partial y_n}+\frac{\partial\overline{\phi}}{\partial y_n}\frac{\partial\overline{u}}{\partial y_{n}}+\overline{\phi}\frac{\partial^2\overline{u}^2}{\partial y_n^2}-\overline{\phi}\frac{\partial\overline{u}}{\partial y_{n}}\frac{\partial^2\overline{h}}{\partial y_{n}^2}.
\end{align*}
By \eqref{v_n}, \eqref{u_lu_nl} and the boundary condition that $\frac{\partial\overline{u}}{\partial y_n}=-\overline{\phi}\overline{v}^{1-q}$ on $\left(U\cap g^{-1}(\partial\Omega)\right)\times\{t_0\}$, we obtain, at $(y_0,t_0)$,
\begin{align}
\frac{1}{q+1}\frac{\partial\overline{w}}{\partial y_n}
&=\overline{v}^{q-1}\left(\nabla'\overline{u}\cdot\frac{\partial}{\partial y_n}(\nabla'\overline{u})+(\kappa\nabla'\overline{u})\cdot\nabla'\overline{u}\right)-\overline{\phi}\frac{\partial^2\overline{u}}{\partial y_n^2}\nonumber\\
&\qquad\qquad\qquad\qquad\qquad\qquad+\frac{\partial\overline{\phi}}{\partial y_n}\frac{\partial\overline{u}}{\partial y_{n}}+\overline{\phi}\frac{\partial^2\overline{u}^2}{\partial y_n^2}-\overline{\phi}\frac{\partial\overline{u}}{\partial y_{n}}\frac{\partial^2\overline{h}}{\partial y_{n}^2}\nonumber\\
&=-\frac{\overline{v}^q+(1-q)\overline{\phi}^2\overline{v}^{-q}}{\overline{v}^q-(1-q)\overline{\phi}^2\overline{v}^{-q}}\nabla'\overline{u}\cdot\nabla'\overline{\phi}+\overline{v}^{q-1}(\kappa\nabla'\overline{u})\cdot\nabla'\overline{u}\nonumber\\
&\qquad\qquad\qquad\qquad\qquad\qquad\qquad\qquad-\overline{\phi}\frac{\partial\overline{\phi}}{\partial y_n}\overline{v}^{1-q}-\overline{\phi}\frac{\partial\overline{u}}{\partial y_{n}}\frac{\partial^2\overline{h}}{\partial y_{n}^2}\label{wy_n}.
\end{align}
At this point, we emphasize the cancellation of the terms $\pm\overline{\phi}\frac{\partial^2\overline{u}}{\partial y_n^2}$ while we compute the normal derivative $\frac{\partial \overline{w}}{\partial y_n}$ at $(y_0,t_0)$. The term $\frac{\partial^2\overline{u}}{\partial y_n^2}$ is the hardest term to get information among the terms in the Hessian $D^2\overline{u}$ of $\overline{u}$.


We recall the definitions of $C_0(x_0),\ C_0$;
\begin{align*}
C_0(x_0)&=\max\{-\lambda:\lambda\ \textrm{is an eigenvalue of}\ \kappa\textrm{ at }x_0\},\\
C_0&=\sup\{C_0(x_0):x_0\in\partial\Omega\}.
\end{align*}
Also, if $v>\left(2|1-q|\|\phi\|^2_{C^0(\partial\Omega)}\right)^{1/2q}=:R_0$, then $|(1-q)\phi^2v^{-2q}|<\frac{1}{2}$ at $(x_0,t_0)$, and thus,
$$
\frac{1}{3}<\frac{v^q+(1-q)\phi^2v^{-q}}{v^q-(1-q)\phi^2v^{-q}}(x_0,t_0)<3.
$$
Note that $R_0$ is independent of $T\in(0,\infty),\ \eta\in(0,1],\ x_0\in\partial\Omega$. Lastly, we check that $\frac{\partial^2\overline{h}}{\partial y_n^2}(y_0)=\frac{\partial^2h}{\partial x_n^2}(x_0)$ since the coordinate change $g:U\to V$ is the identity on the line $\{(0,\cdots,0,y_n):|y_n|<\sigma\}$. 

Finally, if $v=v(x_0,t_0)>R_0$, and also if $\eta\in(0,1]$, then
\begin{align*}
&\ \ \frac{1}{q+1}\frac{\partial w}{\partial \Vec{\mathbf{n}}}(x_0,t_0)\\
&\leq3\left|Du(x_0,t_0)\right||D\phi(x_0)|+C_0v(x_0,t_0)^{q-1}|D'u(x_0,t_0)|^2\\
&\qquad\qquad\qquad\qquad\qquad\qquad+|\phi(x_0)||D\phi(x_0)|v(x_0,t_0)^{1-q}+|\phi(x_0)|Du(x_0,t_0)||D^2h(x_0)||,
\end{align*}
from the fact that $\frac{1}{q+1}\frac{\partial w}{\partial \Vec{\mathbf{n}}}(x_0,t_0)=-\frac{1}{q+1}\frac{\partial\overline{w}}{\partial y_n}(y_0,t_0)$ and \eqref{wy_n}.
By the boundary condition $\frac{\partial u}{\partial x_n}=-\phi v^{1-q}$ at $(x_0,t_0)$, we see that
$$
|D'u(x_0,t_0)|^2=v(x_0,t_0)^2-\left(\frac{\partial u}{\partial x_n}(x_0,t_0)\right)^2-\eta^2=v(x_0,t_0)^2-\phi(x_0)^2v(x_0,t_0)^{2-2q}-\eta^2.
$$
Together with the fact that $\eta\in(0,1]$ and that
$$
C_0v(x_0,t_0)^{q-1}(-\phi(x_0)^2v(x_0,t_0)^{2-2q}-\eta^2)\leq |C_0|\|\phi\|^2_{C^0(\partial\Omega)}v(x_0,t_0)^{1-q}+|C_0|v(x_0,t_0)^{q-1},
$$
we obtain that
\begin{align*}
&\quad\ \frac{1}{q+1}\frac{\partial w}{\partial \Vec{\mathbf{n}}}(x_0,t_0)\\
&\leq3\|D\phi\|_{C^0(\partial\Omega)}v(x_0,t_0)+C_0v(x_0,t_0)^{q+1}+|C_0|\|\phi\|^2_{C^0(\partial\Omega)}v(x_0,t_0)^{1-q}+|C_0|v(x_0,t_0)^{q-1}\\
&\qquad\qquad\qquad\qquad+\|\phi\|_{C^0(\partial\Omega)}\left\|D\phi\right\|_{C^0(\partial\Omega)}v(x_0,t_0)^{1-q}+\|\phi\|_{C^0(\partial\Omega)}\|h\|_{C^2(\partial\Omega)}v(x_0,t_0)
\end{align*}
Therefore, if $v=v(x_0,t_0)>R_0$, and also if $\eta\in(0,1]$, then
\begin{align*}
&\quad\ \frac{1}{q+1}\frac{\partial w}{\partial \Vec{\mathbf{n}}}\leq L_1v^{q+1}
\end{align*}
at $(x_0,t_0)$, where
\begin{align*}
L_1&:=C_0+3\|D\phi\|_{C^0(\partial\Omega)}v^{-q}+|C_0|\|\phi\|^2_{C^0(\partial\Omega)}v^{-2q}+|C_0|v^{-2}\\
&\qquad\qquad\qquad\qquad\qquad+\|\phi\|_{C^0(\partial\Omega)}\|D\phi\|_{C^0(\partial\Omega)}v^{-2q}+\|\phi\|_{C^0(\partial\Omega)}\|h\|_{C^2(\partial\Omega)}v^{-q},
\end{align*}
with $v=v(x_0,t_0)$.

Note that for a given $\varepsilon_0'\in(0,1)$, it holds that $1-\varepsilon_0'<1-(q+1)\phi^2v^{-2q}<1+\varepsilon_0'$ when $v>\max\left\{1,R_0,\left((q+1)\|\phi\|^2_{C^0(\partial\Omega)}(\varepsilon_0')^{-1}\right)^{1/2q}\right\}$. Thus, for a given $\varepsilon_0'\in(0,1)$, there exists $R_{\varepsilon_0'}>1$ that may depend on $\varepsilon_0'$ but not on $T\in(0,\infty),\ \eta\in(0,1],\ x_0\in\partial\Omega$ such that $C_0\leq L_1<C_0+\varepsilon_0'$, and
$1-\varepsilon_0'<1-(q+1)\phi^2v^{-2q}<1+\varepsilon_0'$ and that $\frac{1}{q+1}\frac{\partial w}{\partial \Vec{\mathbf{n}}}\leq L_1v^{q+1}$ whenever $v>R_{\varepsilon_0'}$. Also, $w=v^{q+1}-(q+1)\phi^2v^{1-q}>(1-\varepsilon_0')v^{q+1}>0$ on $\partial\Omega\times\{t_0\}$ whenever $v>R_{\varepsilon_0'}$.

For a given $\varepsilon_0'\in(0,1)$ and for $v=v(x_0,t_0)>R_{\varepsilon_0'},\ \eta\in(0,1]$, we have
\begin{align*}
\frac{1}{q+1}\frac{\partial w}{\partial \Vec{\mathbf{n}}}&\leq L_1v^{q+1}\\
&=L_1\frac{v^{q+1}}{v^{q+1}-(q+1)\phi^2v^{1-q}}w\\
&=\frac{L_1}{1-(q+1)\phi^2v^{-2q}}w.
\end{align*}
at $(x_0,t_0)$. If $C_{0}+\varepsilon_0'\geq0$,
$$
\frac{L_1}{1-(q+1)\phi^2v^{-2q}}<\frac{C_0+\varepsilon_0'}{1-\varepsilon_0'},
$$
and if $C_{0}+\varepsilon_0'<0$,
$$
\frac{L_1}{1-(q+1)\phi^2v^{-2q}}<\frac{C_0+\varepsilon_0'}{1+\varepsilon_0'}.
$$
For a given $\varepsilon_0\in(0,1)$, there exists $\varepsilon_0'\in(0,1)$ that depends only on $\varepsilon_0$ such that
$$
\frac{C_0+\varepsilon_0'}{1-\varepsilon_0'}<C_0+\varepsilon_0\quad\textrm{ and }\quad\frac{C_0+\varepsilon_0'}{1+\varepsilon_0'}<C_0+\varepsilon_0.
$$
Therefore, for a given $\varepsilon_0\in(0,1)$, there exists a constant $R_{\varepsilon_0}>1$ that may depend on $\varepsilon_0$ but not on $T\in(0,\infty),\ \eta\in(0,1]$ and also not on $x_0\in\partial\Omega$ such that at $(x_0,t_0)$, $w>0$ for $v>R_{\varepsilon_0}$, and
\begin{align*}
\frac{1}{q+1}\frac{\partial w}{\partial \Vec{\mathbf{n}}}<\left(C_0+\varepsilon_0\right)w,
\end{align*}
or
\begin{align}\label{w_n}
\frac{\partial w}{\partial \Vec{\mathbf{n}}}<Lw
\end{align}
for $v>R_{\varepsilon_0}$, where $L=(q+1)(C_0+\varepsilon_0)$. Note that we relied on the fact that $x_0$ is a maximizer of $w$ on $\partial\Omega\times\{t_0\}$, and this condition will be emphasized in future applications in the estimate on the boundary.



We claim that if $C_0<0$, then $v(x_0,t_0)\leq R$ for some constant $R>1$ that does not depend on $T\in(0,\infty),\ \eta\in(0,1]$ and also not on $x_0\in\partial\Omega$. This is because if we choose $\varepsilon_0=\frac{1}{2}\min\{\frac{1}{2},-\frac{1}{2}C_0\}$, then there is a constant $R=R_{\varepsilon_0}$, which is now fixed by the choice of $\varepsilon_0$, such that $w>0$ and
\begin{align*}
\frac{1}{q+1}\frac{\partial w}{\partial \Vec{\mathbf{n}}}<\left(C_0+\varepsilon_0\right)w
\end{align*}
if $v(x_0,t_0)=v>R=R_{\varepsilon_0}$. If it really were that $v(x_0,t_0)>R=R_{\varepsilon_0}$, then we would have
\begin{align*}
\frac{1}{q+1}\frac{\partial w}{\partial \Vec{\mathbf{n}}}<\left(C_0+\varepsilon_0\right)w<0.
\end{align*}
However, this is a contradiction, since $x_0$ is a maximizer of $w$ on $\overline{\Omega}\times\{t_0\}$, it must hold that $\frac{\partial w}{\partial \Vec{\mathbf{n}}}\geq0$ at $(x_0,t_0)$. Therefore, $v(x_0,t_0)\leq R$ for some constant $R>0$ that does not depend on $T\in(0,\infty),\ \eta\in(0,1]$ (also not on $x_0\in\partial\Omega$). Since our goal is to prove the bound $v(x_0,t_0)\leq R$, we are done in the case when $C_0<0$, and this argument verifies Theorem \ref{thm:global-grad} in the case when $C_0<0$ under the assumption \eqref{condition:c} with $C_0<0$.

It remains the case when $C_0\geq0$. From now on, we assume that $C_0\geq0$, and thus that $L\geq0$.

\medskip

\emph{Step 6.} For a new function $\psi:=\rho w$, we get a new maximizer $(x_1,t_1)$ of $\psi$ with $x_1\in\Omega,t_1>0$ by choosing a specific multiplier $\rho$. We apply the maximum principle to $\psi$ at $(x_1,t_1)$ in order to bound $v(x_1,t_1)$.

\medskip

Let $\psi:=\rho w$ with a multiplier $\rho=\rho(x)$ that is smooth on $\mathbb{R}^n$. We require that $\rho(x_0)=1$, $\frac{\partial \rho}{\partial \Vec{\mathbf{n}}}(x_0)=-L$. Let $B=B(x_c,K_0)$ be the open ball with the center $x_c:=x_0-K_0\Vec{\mathbf{n}}(x_0)$ so that $B\subseteq\Omega$ and $\overline{B}\cap(\mathbb{R}^n\setminus\Omega)=\{x_0\}$. Choose
$$
\rho(x):=-\frac{L}{2K_0}|x-x_c|^2+\frac{LK_0}{2}+1.
$$
Since we assume $L\geq0$, it holds that $\rho\geq1$ in $B$. Also, $\rho$ is a quadratic function in $|x-x_c|$, and $\rho(x_0)=1,\ \frac{\partial \rho}{\partial \Vec{\mathbf{n}}}(x_0)=-L$. Then, by \eqref{w_n},
$$
\frac{\partial \psi}{\partial \Vec{\mathbf{n}}}=\rho\frac{\partial w}{\partial \Vec{\mathbf{n}}}+w\frac{\partial \rho}{\partial \Vec{\mathbf{n}}}=\frac{\partial w}{\partial \Vec{\mathbf{n}}}+(-L)w<0,\qquad\text{at }(x_0,t_0),
$$
if $v(x_0,t_0)>R_{\varepsilon_0}$ for a given $\varepsilon_0\in(0,1)$. 

For a given $\varepsilon_0\in(0,1)$, assume $v(x_0,t_0)>R_{\varepsilon_0}$. Say the maximum of $\psi=\rho w$ on $\overline{B}\times[0,T]$ occurs at $(x_1,t_1)\in\overline{B}\times[0,T]$. If $t_1=0$, then
$$
w(x_0,t_0)=\rho(x_0)w(x_0,t_0)\leq\rho(x_1)w(x_1,t_1)=\rho(x_1)w(x_1,0)\leq R,
$$
for some constant $R>1$ independent of $T\in(0,\infty),\ \eta\in(0,1]$. Thus, it proves that $w(x_0,t_0)\leq R$ in this case. Using the fact that $v^{q+1}-Cv\leq w$ for some constant $C>0$ depending only on $\|\phi\|_{C^0(\overline{\Omega})},\ \|h\|_{C^1(\overline{\Omega})}$, we see that $v(x_0,t_0)\leq R$, and we reach our goal. Therefore, we now consider the case when $t_1>0$, and we assume $t_1>0$ from now on. If $x_1\in\partial B$, then $\rho(x_1)=\rho(x_0)$, and thus,
$$
\rho(x_1)w(x_1,t_1)\leq\rho(x_0)w(x_0,t_0).
$$
However, $\rho(x_0)w(x_0,t_0)<\rho(x)w(x,t_0)$ for some $x\in B$ since $\frac{\partial (\rho w)}{\partial \Vec{\mathbf{n}}}(x_0,t_0)<0$. It contradicts with the choice of $(x_1,t_1)\in\argmax_{\overline{B}\times[0,T]}\psi$. Therefore, $x_1\in B$, and it suffices to consider the case $(x_1,t_1)\in B\times(0,T]$. 

For a given $\varepsilon_0\in(0,1)$, we always assume from now on that $v(x_0,t_0)>R_{\varepsilon_0}$ so that $w>0 and $\eqref{w_n} are valid. Also, we assume that a maximizer $(x_1,t_1)\in\argmax_{\overline{B}\times[0,T]}\psi$ happens in $B\times(0,T]$, since we achieve the goal, i.e., to prove $v(x_0,t_0)\leq R$, in the other cases from the above argument. Fix $(x_1,t_1)\in\argmax_{\overline{B}\times[0,T]}\psi\cap(B\times(0,T])$.

Before we move on the next step, we check that there exists a constant $C>0$ depending only on $\|\phi\|_{C^0(\overline{\Omega})},\ \|h\|_{C^1(\overline{\Omega})}$ such that the condition $v(x_0,t_0)>R_{\varepsilon_0}$ with $R_{\varepsilon_0}>(8C)^{\frac{1}{q+1}}$ implies the condition $v(x_1,t_1)>\left(\frac{1}{4C}\right)^{\frac{1}{q+1}}R_{\varepsilon_0}=:R'_{\varepsilon_0}$. This is because there exists a constant $C>0$ depending only on $\|\phi\|_{C^0(\overline{\Omega})},\ \|h\|_{C^1(\overline{\Omega})}$ such that
$$
v(x_0,t_0)^{q+1}-Cv(x_0,t_0)\leq w(x_0,t_0)\leq\rho(x_1,t_1)w(x_1,t_1)\leq C(v(x_1,t_1)^{q+1}+v(x_1,t_1)).
$$
Moreover, if $v(x_0,t_0)>R_{\varepsilon_0}$ with $R_{\varepsilon_0}>(8C)^{\frac{1}{q+1}}$, then
$$
\frac{1}{2}R_{\varepsilon_0}^{q+1}<\frac{1}{2}v(x_0,t_0)^{q+1}\leq v(x_0,t_0)^{q+1}-Cv(x_0,t_0)\leq C(v(x_1,t_1)^{q+1}+v(x_1,t_1)).
$$
If $v(x_1,t_1)\leq1$, then we would have $\frac{1}{2}R_{\varepsilon_0}^{q+1}<2C$, which contradicts to $R_{\varepsilon_0}>(8C)^{\frac{1}{q+1}}$. Thus, $v(x_1,t_1)>1$, which gives $\frac{1}{2}R_{\varepsilon_0}^{q+1}<2Cv(x_1,t_1)^{q+1}$ and the conclusion that $v(x_1,t_1)>R'_{\varepsilon_0}$. We note that this is true whenever we replace the constant $C>0$ by a larger one.

Writing $R'_{\varepsilon_0}=\left(\frac{1}{4C}\right)^{\frac{1}{q+1}}R_{\varepsilon_0},$ $R_{\varepsilon_0}=\left(4C\right)^{\frac{1}{q+1}}R'_{\varepsilon_0}$ (and also for $R,\ R'$ similarly), we can state the above equivalently that if $v(x_1,t_1)\leq R'_{\varepsilon_0}$, then $v(x_0,t_0)\leq \max\left\{R_{\varepsilon_0},(8C)^{\frac{1}{q+1}}\right\}$. Accordingly, we change our goal from verifying $v(x_0,t_0)\leq R$ to proving $v(x_1,t_1)\leq R'$.


\medskip

By the maximum principle, $D^2\psi\leq0,\ \psi_t\geq0$ at $(x_1,t_1)$, and thus,
\begin{equation}\label{maxprinciplerhor}
0\geq \frac{1}{(q+1)\rho}\left(\textrm{tr}\{a(Du)D^2\psi\}-\psi_t\right)
\end{equation}
at $(x_1,t_1)$.
Substituting the derivatives of $\psi$ with those of $\rho$ and $w$, we obtain, at $(x_1,t_1)$,
\begin{align}
0
\geq\frac{w}{(q+1)\rho}\textrm{tr}\{a(Du)D^2\rho\}+\frac{2}{(q+1)\rho}\textrm{tr}\{a(Du)Dw\otimes D\rho\}+\frac{1}{q+1}(\textrm{tr}\{a(Du)D^2w\}-w_t).\label{maxprincipleexpandedrhor}
\end{align}
Following the computations up to \eqref{maxprinciplesimplifiedr} in Step 1, we see that there exist a constant $C>0$ independent of $T\in(0,\infty),\ \eta\in(0,1],\ \varepsilon_0\in(0,1)$ and a constant $R'_{\varepsilon_0}>1$ that may depend on $\varepsilon_0\in(0,1)$ but not on $T\in(0,\infty),\ \eta\in(0,1]$ such that, at $(x_1,t_1)$,
\begin{align}
0&\geq\frac{1}{q+1}\left(\textrm{tr}\{a(Du)D^2w\}-w_t\right)\nonumber\\
&\geq J_1+J_2-|Dc|v^{q+1}+(q+1-\varepsilon_0)V-C(v+v^q)\label{maxprinciplesimplifiedrhor}
\end{align}
if $v=v(x_1,t_1)>R'_{\varepsilon_0}$, with the same definitions of $J_1,\ J_2$ ($\varepsilon$ replaced by $\varepsilon_0$).

We check for a moment that, at $(x_1,t_1)$,
\begin{align}
V\geq V_1+V_2,\label{Vexpansion}
\end{align}
where
\begin{align*}
V_1&:=v^{-q-1}\textrm{tr}\left\{a(Du)\left(\frac{w}{(q+1)\rho}D\rho\right)\otimes\left(\frac{w}{(q+1)\rho}D\rho\right)\right\},\\
V_2&:=-2v^{-q-1}\textrm{tr}\left\{a(Du)\left(\frac{w}{(q+1)\rho}D\rho\right)\otimes\left((Du\cdot Dh)D\phi+\phi D^2uDh+\phi D^2hDu\right)\right\}.
\end{align*}
At $(x_1,t_1)$, we have that $D\psi=wD\rho+\rho Dw=0$ so that
$$
-\frac{w}{\rho}D\rho=(q+1)(v^qDv-(Du\cdot Dh)D\phi-\phi D^2uDh-\phi D^2hDu).
$$
By putting
$$
v^qDv=-\frac{w}{(q+1)\rho}D\rho+(Du\cdot Dh)D\phi+\phi D^2uDh+\phi D^2hDu
$$
into $V=v^{q-1}\textrm{tr}\{a(Du)Dv\otimes Dv\}=v^{-q-1}\textrm{tr}\{a(Du)(v^qDv)\otimes(v^qDv)\}$, we obtain \eqref{Vexpansion}.

By \eqref{maxprincipleexpandedrhor}, \eqref{maxprinciplesimplifiedrhor}, \eqref{Vexpansion}, there exist a constant $C>0$ independent of $T\in(0,\infty),\ \eta\in(0,1],\ \varepsilon_0\in(0,1)$ and a constant $R'_{\varepsilon_0}>1$ that may depend on $\varepsilon_0\in(0,1)$ but not on $T\in(0,\infty),\ \eta\in(0,1]$ such that, at $(x_1,t_1)$,
\begin{align}
0&\geq\frac{w}{(q+1)\rho}\textrm{tr}\{a(Du)D^2\rho\}+\left(\frac{2}{(q+1)\rho}\textrm{tr}\{a(Du)Dw\otimes D\rho\}+(q+1)V_1\right)\nonumber\\
&\qquad\qquad\qquad\qquad\qquad\qquad\qquad\qquad\qquad+J'_1+J'_2-|Dc|v^{q+1}-C(v+v^q)\label{maxprinciplerhofinalr}
\end{align}
if $v=v(x_1,t_1)>R'_{\varepsilon_0}$, where
\begin{align*}
J'_1:&=J_1-\varepsilon_0V\\
&=(1-\varepsilon_0)v^{q-1}\textrm{tr}\{(a(Du)D^2u)^2\}-\frac{1}{2}v^q\textrm{tr}\{(D_pa(Du)\odot Dv)D^2u\}\\
&\qquad\qquad\qquad\qquad\qquad\qquad\qquad\qquad+cDv\cdot(-v^{q-1}Du+\phi Dh)-\varepsilon_0V,\\
J'_2:&=J_2+(q+1)V_2\\
&=\varepsilon_0 v^{q-1}\textrm{tr}\{(a(Du)D^2u)^2\}-\frac{1}{2}\varepsilon_0 v^q\textrm{tr}\{(D_pa(Du)\odot Dv)D^2u\}\\
&\qquad\qquad\qquad-2\textrm{tr}\{a(Du)(D\phi\otimes (D^2uDh))\}-2\phi\textrm{tr}\{a(Du)D^2uD^2h\}\\
&\qquad\qquad\qquad\qquad\qquad\qquad+\phi\textrm{tr}\{(D_pa(Du)\odot(D^2uDh))D^2u\}+(q+1)V_2.
\end{align*}

\medskip

\emph{Step 7.} We estimate the terms of \eqref{maxprinciplerhofinalr}.

\medskip

We start with the first term of \eqref{maxprinciplerhofinalr}. By the fact that $D^2\rho=-\frac{L}{K_0}I_n$ and $\rho\geq1$ in $B$, we see that
\begin{align*}
&\ \ \ \frac{w}{(q+1)\rho}\textrm{tr}\{a(Du)D^2\rho\}\nonumber\\
&=\frac{w}{(q+1)\rho}\left(-\frac{L}{K_0}\right)\left(\frac{\eta^2}{v^2}+n-1\right)\nonumber\\
&\geq-\frac{L}{(q+1)K_0}(v^{q+1}+(q+1)\|\phi\|_{C^0(\overline{\Omega})}\|h\|_{C^1(\overline{\Omega})}v)\left(\frac{\eta^2}{v^2}+n-1\right).\nonumber
\end{align*}
Therefore, there exists a constant $C>0$ independent of $T\in(0,\infty),\ \eta\in(0,1],\ \varepsilon_0\in(0,1)$ such that, at $(x_1,t_1)$,
\begin{align}
\frac{w}{(q+1)\rho}\textrm{tr}\{a(Du)D^2\rho\}\geq-\frac{(n-1)(C_0+\varepsilon_0)}{K_0}v^{q+1}-C(v+v^q)\label{term1rho}
\end{align}
if $v=v(x_1,t_1)>1$. Here, we have used the fact that $\eta\in(0,1]$.

We bound the second term of \eqref{maxprinciplerhofinalr}. Since $Dw=-\frac{w}{\rho}D\rho$ at $(x_1,t_1)$, we obtain
\begin{align*}
\frac{2}{(q+1)\rho}\textrm{tr}\{a(Du)Dw\otimes D\rho\}+(q+1)V_1=\frac{(wv^{-1-q}-2)w}{q+1}\textrm{tr}\left\{a(Du)\frac{D\rho}{\rho}\otimes\frac{D\rho}{\rho}\right\}
\end{align*}
at $(x_1,t_1)$. From
$$
v^{q+1}-(q+1)\|\phi\|_{C^0(\overline{\Omega})}\|h\|_{C^1(\overline{\Omega})}v\leq w\leq v^{q+1}+(q+1)\|\phi\|_{C^0(\overline{\Omega})}\|h\|_{C^1(\overline{\Omega})}v,
$$
we see that there exists a constant $R'_{\varepsilon_0}>1$ that may depend on $\varepsilon_0\in(0,1)$ but not on $T\in(0,\infty),\ \eta\in(0,1]$ such that $|wv^{-q-1}-1|<\varepsilon_0$ for $v>R'_{\varepsilon_0}$. Using the fact that
$$
0\leq\textrm{tr}\left\{a(Du)\frac{D\rho}{\rho}\otimes\frac{D\rho}{\rho}\right\}\leq \left|\frac{D\rho}{\rho}\right|^2=\frac{L^2}{K_0^2}|x_1-x_c|^2\leq(C_0+\varepsilon_0)^2,
$$
and the fact that
$$
w\leq v^{q+1}+(q+1)\|\phi\|_{C^0(\overline{\Omega})}\|h\|_{C^1(\overline{\Omega})}v
$$
once again, we see that there exist a constant $C>0$ independent of $T\in(0,\infty),\ \eta\in(0,1],\ \varepsilon_0\in(0,1)$ and a constant $R'_{\varepsilon_0}>1$ that may depend on $\varepsilon_0\in(0,1)$ but not on $T\in(0,\infty),\ \eta\in(0,1]$ such that, at $(x_1,t_1)$,
\begin{align}
\frac{2}{(q+1)\rho}\textrm{tr}\{a(Du)Dw\otimes D\rho\}+(q+1)V_1\geq-(q+1)(C_0+\varepsilon_0)^2(1+\varepsilon_0)v^{q+1}-Cv.\label{term2rho}
\end{align}
if $v=v(x_1,t_1)>R'_{\varepsilon_0}$.

We give an estimate of the term $J_1'$ of \eqref{maxprinciplerhofinalr}. Following the same computation of $J_1$, we have \eqref{term1J1} with $\varepsilon_0$ instead of $\varepsilon$, and thus, we see that there exist a constant $C>0$ independent of $T\in(0,\infty),\ \eta\in(0,1],\ \varepsilon_0\in(0,1)$ and a constant $R'_{\varepsilon_0}>1$ that may depend on $\varepsilon_0\in(0,1)$ but not on $T\in(0,\infty),\ \eta\in(0,1]$ such that
\begin{align}
(1-\varepsilon_0)v^{q-1}\textrm{tr}\{(a(Du)D^2u)^2\}-\frac{1}{2}v^q\textrm{tr}\{(D_pa(Du)\odot Dv)D^2u\}\geq\frac{1-\varepsilon_0}{n-1}c^2v^{q+1}+\varepsilon_0V-Cv^q.\label{term1J1'}
\end{align}
if $v>R'_{\varepsilon_0}$.

We claim that at $(x_1,t_1)$, it holds that, for $v>1$,
\begin{align}
|cDv\cdot(-v^{q-1}Du+\phi Dh)|\leq Cv+(C_0+\varepsilon_0)|c|v^{q+1}\label{term3J1'}
\end{align}
for some constant $C>0$ independent of $T\in(0,\infty),\ \eta\in(0,1],\ \varepsilon_0\in(0,1)$. Since $D\psi=0$ at $(x_1,t_1)$,
\begin{align*}
0&=\frac{1}{(q+1)\rho}D\psi\cdot Du\\
&=v^qDu\cdot Dv-(Du\cdot D\phi)(Du\cdot Dh)-\phi (D^2uDu)\cdot Dh-\phi (D^2hDu)\cdot Du\\
&\qquad\qquad\qquad\qquad\qquad\qquad\qquad\qquad\qquad\qquad\qquad\qquad\qquad+\frac{w}{(q+1)\rho}D\rho\cdot Du.
\end{align*}
This implies that at $(x_1,t_1)$,
$$
cDv\cdot(-v^{q-1}Du+\phi Dh)=-\frac{c}{v}\left((Du\cdot D\phi)(Du\cdot Dh)+\phi (D^2hDu)\cdot Du-\frac{w}{(q+1)\rho}D\rho\cdot Du\right),
$$
and thus that at $(x_1,t_1)$,
\begin{align*}
|cDv\cdot(-v^{q-1}Du+\phi Dh)|&\leq Cv+\frac{L|c|}{K_0(q+1)}|x_1-x_c|(v^{q+1}+(q+1)\|\phi\|_{C^0(\overline{\Omega})}\|h\|_{C^1(\overline{\Omega})}v)\\
&\leq Cv+(C_0+\varepsilon_0)|c|v^{q+1}
\end{align*}
for some constant $C>0$ independent of $T\in(0,\infty),\ \eta\in(0,1],\ \varepsilon_0\in(0,1)$.

By \eqref{term1J1'}, \eqref{term3J1'}, we conclude that there exist a constant $C>0$ independent of $T\in(0,\infty),\ \eta\in(0,1],\ \varepsilon_0\in(0,1)$ and a constant $R'_{\varepsilon_0}>1$ that may depend on $\varepsilon_0\in(0,1)$ but not on $T\in(0,\infty),\ \eta\in(0,1]$ such that
\begin{align}
J'_1\geq\left(\frac{1-\varepsilon_0}{n-1}c^2-(C_0+\varepsilon_0)|c|\right)v^{q+1}-C(v+v^q)\label{J1'}
\end{align}
if $v>R'_{\varepsilon_0}$.

Now, we bound the term $J_2'$ of \eqref{maxprinciplerhofinalr}. Taking the axes at $x_1$ so that \eqref{rotation1} holds, and calculating $u_{1i},\ i=2,\cdots,n,\ u_{11}$ using $\rho Dw+wD\rho=0$ at $(x_1,t_1)$, we obtain
\begin{align}
u_{1i}=E_iu_1+F_iu_{ii}+G_iw,\qquad i=2,\cdots,n,\label{secondderivative1rho}
\end{align}
where
\begin{align}
E_i:=\frac{\phi_ih_1+\phi h_{1i}}{v^{q-1}u_1-\phi h_1},\ \ \ F_i:=\frac{\phi h_{i}}{v^{q-1}u_1-\phi h_1},\qquad i=2,\cdots,n\nonumber
\end{align}
and
\begin{align}
G_i:=-\frac{\rho_{i}}{(q+1)\rho(v^{q-1}u_1-\phi h_1)},\qquad i=2,\cdots,n.\nonumber
\end{align}
For $i=1$, we get
\begin{align}
u_{11}&=E_1u_1+\sum_{\ell=2}^nF_{\ell}^2u_{\ell\ell}+G_1w,\label{secondderivative21rho}
\end{align}
where
\begin{align}
E_1:=\frac{\phi_1h_1+\phi h_{11}}{v^{q-1}u_1-\phi h_1}+\frac{\phi}{v^{q-1}u_1-\phi h_1}\sum_{{\ell}=2}^nh_{\ell}E_{\ell},\nonumber
\end{align}
and
\begin{align}
G_1:=-\frac{\rho_1}{(q+1)\rho(v^{q-1}u_1-\phi h_1)}+\frac{\phi}{v^{q-1}u_1-\phi h_1}\sum_{\ell=2}^nh_{\ell}G_{\ell}.\nonumber
\end{align}
The definitions of $E_i$'s and $F_i$'s are the same as before, and we display them to recall. Note that the denominator $v^{q-1}u_1-\phi h_1$ is nonzero for $v>R'$ for some constant $R'>1$ independent of $T\in(0,\infty),\ \eta\in(0,1],\ \varepsilon_0\in(0,1)$.

Write $$
J_2'=\varepsilon_0 v^{q-1}\textrm{tr}\{(a(Du)D^2u)^2\}+S'_1+S'_2-\frac{2wv^{-q-1}}{\rho}\textrm{tr}\{a(Du)D\rho\otimes((Du\cdot Dh)D\phi+\phi D^2hDu)\},
$$
where
\begin{align*}
S'_1:&=S_1-\frac{2wv^{-q-1}\phi}{\rho}\textrm{tr}\{a(Du)D\rho\otimes(D^2uDh)\}\\
&=-2\textrm{tr}\{a(Du)(D\phi\otimes (D^2uDh))\}-2\phi\textrm{tr}\{a(Du)D^2uD^2h\}\\
&\qquad\qquad\qquad\qquad\qquad\qquad\qquad-\frac{2wv^{-q-1}\phi}{\rho}\textrm{tr}\{a(Du)D\rho\otimes(D^2uDh)\},\\
S'_2:&=S_2\\
&=-\frac{1}{2}\varepsilon_0 v^q\textrm{tr}\{(D_pa(Du)\odot Dv)D^2u\}+\phi\textrm{tr}\{(D_pa(Du)\odot(D^2uDh))D^2u\},
\end{align*}
with $S_1,\ S_2$ defined as in Case 1 ($\varepsilon$ replaced by $\varepsilon_0$).

Computing $S'_1$ in a similar manner as before, we get
\begin{align*}
S'_1&=-2\left(\left(\frac{\eta^2}{v^2}H'_{11}E_{1}+\sum_{\ell=2}^n\left(\frac{\eta^2}{v^2}H'_{1\ell}+H'_{\ell1}\right)E_{\ell}\right)u_1\right.\\
&\left.\qquad\qquad\qquad+\sum_{\ell=2}^n\left(\frac{\eta^2}{v^2}H'_{11}F_{\ell}^2+\left(\frac{\eta^2}{v^2}H'_{1\ell}+H'_{\ell1}\right)F_{\ell}+H'_{\ell\ell}\right)u_{\ell\ell}\right)\\
&\qquad\qquad\qquad\qquad\qquad\qquad-2\left(\frac{\eta^2}{v^2}H'_{11}G_1+\sum_{\ell=2}^nG_{\ell}\left(\frac{\eta^2}{v^2}H'_{1\ell}+H'_{\ell1}\right)\right)w,
\end{align*}
where $H'_{\ell i}:=H_{\ell i}+\frac{wv^{-q-1}\phi}{\rho}\rho_{\ell}h_i=\phi_{\ell}h_i+\phi h_{\ell i}+\frac{wv^{-q-1}\phi}{\rho}\rho_{\ell}h_i$ for each $\ell,i=1,\cdots,n$.
Note that since $\eta\in(0,1]$,
$$
\left|\frac{\eta^2}{v^2}H'_{11}E_{1}+\sum_{\ell=2}^n\left(\frac{\eta^2}{v^2}H'_{1\ell}+H'_{\ell1}\right)E_{\ell}\right|\leq Cv^{-q},
$$
$$
\left|\frac{\eta^2}{v^2}H'_{11}F_{\ell}^2+\left(\frac{\eta^2}{v^2}H'_{1\ell}+H'_{\ell1}\right)F_{\ell}+H'_{\ell\ell}\right|\leq C,
$$
$$
\left|\frac{\eta^2}{v^2}H'_{11}G_1+\sum_{\ell=2}^nG_{\ell}\left(\frac{\eta^2}{v^2}H'_{1\ell}+H'_{\ell1}\right)\right|\leq Cv^{-q}
$$
for $v>R'_{\varepsilon_0}$. Here, $R'_{\varepsilon_0}>1$ is some constant that may depend on $\varepsilon_0\in(0,1)$ but not on $T\in(0,\infty),\ \eta\in(0,1]$, and $C>0$ is another constant independent of $T\in(0,\infty),\ \eta\in(0,1],\ \varepsilon_0\in(0,1)$. Using the fact that $|w|\leq v^{q+1}+(q+1)\|\phi\|_{C^0(\overline{\Omega})}\|h\|_{C^1(\overline{\Omega})}v$, we see that there exist a constant $C>0$ independent of $T\in(0,\infty),\ \eta\in(0,1],\ \varepsilon_0\in(0,1)$ and a constant $R'_{\varepsilon_0}>1$ that may depend on $\varepsilon_0\in(0,1)$ but not on $T\in(0,\infty),\ \eta\in(0,1]$ such that
\begin{align}
S'_1\geq-C\left(v+\sum_{\ell=2}^n|u_{\ell\ell}|\right)\label{S1'}
\end{align}
for $v>R'_{\varepsilon_0}$.

Following the same computation of $S_2$, we have
\begin{align*}
S'_2\geq-\varepsilon_0^{-1}v^{-1-q}K_1^2-\varepsilon_0^{-1}v^{1-q}\sum_{\ell=2}^nK_{\ell}^2,
\end{align*}
where $K_{\ell}:=\phi v^{-1}(Du_{\ell}\cdot Dh)$ for each $\ell=1,\cdots,n$. By expansion and \eqref{secondderivative1rho}, \eqref{secondderivative21rho}, we have
\begin{align*}
K_1=K_{11}u_1+\sum_{\ell=2}^nK_{1\ell}u_{\ell\ell}+M_1w,
\end{align*}
where
\begin{align}
K_{11}:=\phi v^{-1}\sum_{\ell=1}^nh_{\ell}E_{\ell},\ \ \ K_{1\ell}:=\phi v^{-1}(h_1F_{\ell}^2+h_{\ell}F_{\ell}),\ \ \ \textrm{for }\ell=2,\cdots,n.\nonumber
\end{align}
and
\begin{align}
M_1:=\phi v^{-1}\sum_{\ell=1}^nh_{\ell}G_{\ell}.\nonumber
\end{align}
For $\ell=2,\cdots,n$,
\begin{align*}
K_{\ell}=K_{\ell1}u_1+K_{\ell\ell}u_{\ell\ell},
\end{align*}
where
\begin{align}
K_{\ell1}:=\phi v^{-1}h_1E_{\ell},\ \ \ K_{\ell\ell}:=\phi v^{-1}(h_1F_{\ell}+h_{\ell}),\ \ \ \textrm{for }\ell=2,\cdots,n,\nonumber
\end{align}
and
\begin{align}
M_{\ell}:=\phi v^{-1}G_{\ell},\ \ \ \textrm{for }\ell=2,\cdots,n.\nonumber
\end{align}
Applying Cauchy-Schwarz inequality as before in $S_2$, we obtain
\begin{align}
S'_2\geq-\varepsilon_0^{-1}S_{21}u_{1}^2-\varepsilon_0^{-1}\sum_{\ell=2}^nS_{2\ell}u_{\ell\ell}^2-\varepsilon_0^{-1}Mw^2,\label{S2'}
\end{align}
where
\begin{align}
S_{21}:=nv^{-1-q}K_{11}^2+2v^{1-q}\sum_{\ell=2}^nK_{\ell1}^2,\ S_{2\ell}:=nv^{-1-q}K_{1\ell}^2+2v^{1-q}K_{\ell\ell}^2,\ \ \textrm{for }\ell=2,\cdots,n,\nonumber
\end{align}
and
\begin{align}
M:=(n+1)M_1^2v^{-1-q}+3v^{1-q}\sum_{\ell=2}^nM_{\ell}^2.\nonumber
\end{align}

We note that there exist a constant $C>0$ independent of $T\in(0,\infty),\ \eta\in(0,1],\ \varepsilon_0\in(0,1)$ and a constant $R'_{\varepsilon_0}>1$ that may depend on $\varepsilon_0\in(0,1)$ but not on $T\in(0,\infty),\ \eta\in(0,1]$ such that, at $(x_1,t_1)$,
\begin{align}
&\ \ \ \left|-\frac{2wv^{-q-1}}{\rho}\textrm{tr}\{a(Du)D\rho\otimes((Du\cdot Dh)D\phi+\phi D^2hDu)\}\right|\nonumber\\
&\leq C\|a\||D\rho|(|D\phi||Dh||Du|+|\phi|\|D^2h\||Du|)\nonumber\\
&\leq Cv.\label{V2second}
\end{align}
Here, we have used the fact that $|wv^{-q-1}-1|<\varepsilon_0$ for $v>R'_{\varepsilon_0}$ (making $R'_{\varepsilon_0}>1$ larger if necessary), that $\rho\geq1$ at $x_1\in B$ and that $\|a\|=\left(\frac{\eta^4}{v^4}+n-1\right)^{1/2}\leq\frac{\eta^2}{v^2}+n-1\leq C$ for $v>1,\ \eta\in(0,1]$. Also, the constants $C>0,\ R'_{\varepsilon_0}>1$ can be taken in a way that they may depend on $\|\rho\|_{C^1(\overline{\Omega})},\ \|\phi\|_{C^1(\overline{\Omega})},\ \|h\|_{C^2(\overline{\Omega})}$, but not on a specific position $x_1\in\overline{\Omega}$.

By \eqref{S1'}, \eqref{S2'}, \eqref{V2second}, we see that there exist a constant $C>0$ independent of $T\in(0,\infty),\ \eta\in(0,1],\ \varepsilon_0\in(0,1)$ and a constant $R'_{\varepsilon_0}>1$ that may depend on $\varepsilon_0\in(0,1)$ but not on $T\in(0,\infty),\ \eta\in(0,1]$ such that
$$
J'_2\geq\varepsilon_0 v^{q-1}\sum_{\ell=2}^nu_{\ell\ell}^2-C(v+\sum_{\ell=2}^n|u_{\ell\ell}|)-\varepsilon_0^{-1}S_{21}u_1^2-\varepsilon_0^{-1}\sum_{\ell=2}^nS_{2\ell}u_{\ell\ell}^2
$$
for $v>R'_{\varepsilon_0}$. As before, there exists a constant $C>0$ that depends only on $\|\phi\|_{C^1(\overline{\Omega})},\ \|h\|_{C^2(\overline{\Omega})}$ such that, for $v>R'_{\varepsilon_0}$
$$
|S_{21}|+\sum_{\ell=2}^n|S_{2\ell}|+|M|\leq Cv^{-1-3q}.
$$
Using the fact that $|w|\leq v^{q+1}+(q+1)\|\phi\|_{C^0(\overline{\Omega})}\|h\|_{C^1(\overline{\Omega})}v$, we see that there exist a constant $C>0$ independent of $T\in(0,\infty),\ \eta\in(0,1],\ \varepsilon_0\in(0,1)$ and a constant $R'_{\varepsilon_0}>1$ that may depend on $\varepsilon_0\in(0,1)$ but not on $T\in(0,\infty),\ \eta\in(0,1]$ such that
\begin{align*}
J'_2\geq-Cv-C\varepsilon_0^{-1}v^{1-q}+\sum_{\ell=2}^n\left(v^{q-1}\left(\varepsilon_0-C\varepsilon_0^{-1}v^{-4q}\right)u_{\ell\ell}^2-C|u_{\ell\ell}|\right)
\end{align*}
for $v>R'_{\varepsilon_0}$.

As before, by choosing $R'_{\varepsilon_0}>1$ that may depend on $\varepsilon_0\in(0,1)$ but not on $T\in(0,\infty),\ \eta\in(0,1]$ such that $C\varepsilon_0^{-1}v^{-4q}<\frac{\varepsilon_0}{2}$ if $v>R'_{\varepsilon_0}$. Then, for $v>R'_{\varepsilon_0}$,
\begin{align*}
J'_2\geq-Cv-C\varepsilon_0^{-1}v^{1-q}+\sum_{\ell=2}^n\left(\frac{\varepsilon_0}{2}v^{q-1}\left(|u_{\ell\ell}|-\frac{C}{\varepsilon_0 v^{q-1}}\right)^2-\frac{C^2}{2\varepsilon_0 v^{q-1}}\right).
\end{align*}
All in all, for each $\varepsilon_0\in(0,1)$, there exist a constant $C>0$ independent of $T\in(0,\infty),\ \eta\in(0,1],\ \varepsilon_0\in(0,1)$ and a constant $R'_{\varepsilon_0}>1$ that may depend on $\varepsilon_0\in(0,1)$ but not on $T\in(0,\infty),\ \eta\in(0,1]$ such that
\begin{align}
J'_2\geq-Cv-C\varepsilon_0^{-1}v^{1-q}.\label{J2'r}
\end{align}
for $v>R'_{\varepsilon_0}$.

\medskip

\emph{Step 8.} We finish Case 2.

\medskip

All in all, by \eqref{maxprinciplerhofinalr}, \eqref{term1rho}, \eqref{term2rho}, \eqref{J1'}, \eqref{J2'r}, we see that there exist a constant $C>0$ independent of $T\in(0,\infty),\ \eta\in(0,1],\ \varepsilon_0\in(0,1)$ and a constant $R'_{\varepsilon_0}>1$ that may depend on $\varepsilon_0\in(0,1)$ but not on $T\in(0,\infty),\ \eta\in(0,1]$ such that, at $(x_1,t_1)$,
\begin{align*}
0&\geq\left(\frac{1-\varepsilon_0}{n-1}c^2-|Dc|-(C_0+\varepsilon_0)|c|-\frac{(n-1)(C_0+\varepsilon_0)}{K_0}\right.\\
&\qquad\qquad\qquad\qquad\qquad\left.-(q+1)(C_0+\varepsilon_0)^2(1+\varepsilon_0)\right)v^{q+1}-C(v+v^q)-C\varepsilon_0^{-1}v^{1-q}.
\end{align*}
if $v>R'_{\varepsilon_0}$. From the condition \eqref{condition:c} and the assumption \eqref{assumtion:c}, we see that there exists $\varepsilon_0\in(0,1)$ such that the coefficient of $v^{q+1}$ satisfies
$$
\frac{1-\varepsilon_0}{n-1}c^2-|Dc|-(C_0+\varepsilon_0)|c|-\frac{(n-1)(C_0+\varepsilon_0)}{K_0}-(q+1)(C_0+\varepsilon_0)^2(1+\varepsilon_0)\geq\frac{\delta}{2}.
$$
Fix such $\varepsilon_0\in(0,1)$. Then, $R'_{\varepsilon_0},\ \varepsilon_0^{-1}$ are fixed as well. Therefore, with this fixed $\varepsilon_0\in(0,1)$, there exist constants $R'>1,\ C>0$ independent of $T\in(0,\infty),\ \eta\in(0,1]$ such that at $(x_1,t_1)$,
$$
0\geq\frac{\delta}{2}v^{q+1}-C(v+v^q)
$$
if $v>R'$. There is, on the other hand, also a constant $R'_0>R'$ independent of $T\in(0,\infty),\ \eta\in(0,1]$ such that
$$
0<\frac{\delta}{2}v^{q+1}-C(v+v^q)
$$
if $v>R'_0$. Therefore, it must hold that $v=v(x_1,t_1)\leq R'_0$, which completes Case 2.
\end{proof}

Next, in order to prove Theorem \ref{thm:local-grad}, we prove \emph{a priori} local gradient estimates, namely the following proposition \ref{prop:exuniquegratime}. 

\begin{proposition}\label{prop:exuniquegratime}
Let $T\in(0,\infty),\ \eta\in(0,1]$. Suppose that a solution $u^{\eta}$ of \eqref{eq} exists and it is of class $C^{2,\sigma}(\overline{\Omega}\times[0,T])\cap C^{3,\sigma}(\Omega\times(0,T])$ for some $\sigma\in(0,1)$. Then $u^{\eta}$ satisfies that
$$
\lVert Du^{\eta}\rVert_{L^{\infty}(\overline{\Omega}\times[0,T])}\leq R_T,
$$
where $R_T>1$ is a constant depending only on $T,\Omega, c, f, \phi, q, u_0$.
\end{proposition}

Note that no assumption on the forcing term $c$ is made, except for being $C^{1,\alpha}$. In the following proof of Proposition \ref{prop:exuniquegratime}, we introduce a time-dependent multiplier.

\begin{proof}[Proof of Proposition \ref{prop:exuniquegratime}]

Now we only assume \eqref{assumtion:c} and \eqref{assumtion:f}. Let $T\in(0,\infty)$, $\eta\in(0,1]$. Let $u=u^{\eta}\in C^{2,\sigma}(\overline{\Omega}\times[0,T])\cap C^{3,\sigma}(\Omega\times(0,T])$ be a solution to \eqref{eq} for some $\sigma\in(0,1)$. Let $w:=v^{q+1}-(q+1)\phi Du\cdot Dh$ on $\overline{\Omega}\times[0,T]$. Let $R_T>1$ denote a constant that may depend on $T\in(0,\infty)$ but not on $\eta\in(0,1)$. As before, $R_T>1$ may vary line by line.

The goal is to prove that $w(x,t)\leq R_T$ for all $(x,t)\in\overline{\Omega}\times[0,T]$. Once we achieve this goal, we complete the proof of Proposition \ref{prop:exuniquegratime} by using the fact that $v^{q+1}-Cv\leq w$ for some constant $C>0$ depending only on $\|h\|_{C^1(\overline{\Omega})},\ \|\phi\|_{C^0(\overline{\Omega})}$ (and $q>0$).

Let $M>1$ be a constant to be determined. Let $(x_0,t_0)\in\argmax_{\overline{\Omega}\times[0,T]}{e^{-Mt}}w(x,t)$. We claim that in both cases of $t_0=0$ and $t_0>0$, $v(x_0,t_0)$ is bounded by a constant $R_T$ that may depend on $T\in(0,\infty)$ but not on $\eta\in(0,1]$. In the case of $t_0=0$, we readily get a local gradient estimate. Indeed,
$$
e^{-Mt}w(x,t)\leq w(x_0,0)\leq R\qquad\text{for all }(x,t)\in\overline{\Omega}\times[0,T]
$$
for some constant $R>1$ depending only on $\|u_0\|_{C^1(\overline{\Omega})},\ \|h\|_{C^1(\overline{\Omega})},\ \|\phi\|_{C^0(\overline{\Omega})}$, which proves our goal. Here, we have used the fact that $\eta\in(0,1]$.

\medskip

It remains the case of $t_0>0$. Let $\rho(x,t)=e^{-Mt}\rho^0(x)$, where $\rho^0(x)$ will be chosen again according to the following cases; again divide into $x_0\in\Omega$ and $x_0\in\partial\Omega$.

\medskip

\noindent {\bf Case 1: $x_0\in\Omega$.}

\medskip

Take $\rho^0\equiv 1$. Since $x_0\in\Omega,\ t_0>0$, and $D\rho=0,\ D^2\rho=0$, we have that
\begin{align*}
0&\geq\frac{1}{(q+1)\rho}\left(\textrm{tr}\{a(Du)D^2\psi\}-\psi_t\right)\\
&\geq\frac{1}{q+1}(\textrm{tr}\{a(Du)D^2w\}-w_t)-\frac{\rho_tw}{(q+1)\rho}
\end{align*}
at $(x_0,t_0)$, where $\psi:=\rho w$ as before.

Following the same argument in Step 1 of the proof of Proposition \ref{prop:exuniquegra}, we see that there exist constants $R>1,\ C>0$ independent of $T\in(0,\infty),\ \eta\in(0,1]$ such that \eqref{maxprinciplesimplifiedr} holds true at $(x_0,t_0)$ for $v>R$. Moreover, since $x_0\in\argmax_{\overline{\Omega}}w(\cdot,t_0)\cap\Omega$ so that $Dw=0$ at $(x_0,t_0)$, \eqref{term3J1} (for some constant $C>0$ independent of $T\in(0,\infty),\ \eta\in(0,1]$), \eqref{secondderivative1}, \eqref{E_iF_i}, \eqref{secondderivative21}, \eqref{E_1} are valid at $(x_0,t_0)$. Therefore, we can follow the estimates in Step 3, Step 4 of the proof of Proposition \ref{prop:exuniquegra} to conclude that for a given $\varepsilon\in(0,1)$, there exists a constant $R_{\varepsilon}>1$ that may depend on $\varepsilon\in(0,1)$ but not on $T\in(0,\infty),\ \eta\in(0,1]$ and a constant $C>0$ independent of $T\in(0,\infty),\ \eta\in(0,1],\ \varepsilon\in(0,1)$ such that \eqref{J1r}, \eqref{J2r} are valid at $(x_0,t_0)$ for $v>R_{\varepsilon}$. We take $\varepsilon=\frac{1}{2}$, and we note that $-\frac{\rho_tw}{(q+1)\rho}=\frac{Mw}{q+1}$. Together with the fact that $w\geq v^{q+1}-(q+1)\|\phi\|_{C^0(\overline{\Omega})}\|h\|_{C^1(\overline{\Omega})}v$, we see that there exists a constant $R>1,\ C>0$ independent of $T\in(0,\infty),\ \eta\in(0,1]$ such that
$$
0\geq\left(\frac{c^2}{2(n-1)}-|Dc|+\frac{M}{q+1}\right)v^{q+1}-C(v+v^q)
$$
at $(x_0,t_0)$ for $v>R$. From the assumption \eqref{assumtion:c}, we can choose a constant $M>1$ independent of $T\in(0,\infty),\ \eta\in(0,1]$ such that
$$
\frac{c(x,z)^2}{2(n-1)}-|Dc(x,z)|+\frac{M}{q+1}>1
$$
for all $(x,z)\in\overline{\Omega}\times\mathbb{R}$. Since there exists a constant $R_0>R$ independent of $T\in(0,\infty),\ \eta\in(0,1]$ such that
$$
0<v^{q+1}-C(v+v^q)
$$
for $v>R_0$, it must hold true that $v(x_0,t_0)\leq R_0$ with the above choice of $M>1$. Using once again the fact that $w\leq v^{q+1}+(q+1)\|\phi\|_{C^0(\overline{\Omega})}\|h\|_{C^1(\overline{\Omega})}v$, we get
$$
e^{-Mt}w(x,t)\leq e^{-Mt_0}w(x_0,t_0)\leq R,\qquad\text{for all }(x,t)\in\overline{\Omega}\times[0,T]
$$
for some constant $R>1$ independent of $T\in(0,\infty),\ \eta\in(0,1]$, which proves our goal in Case 1.

\medskip

\noindent {\bf Case 2: $x_0\in\partial\Omega.$}

\medskip

Since $x_0\in\argmax_{\overline{\Omega}}w(\cdot,t_0)$, we have both $\frac{\partial w}{\partial \Vec{\mathbf{n}}}(x_0,t_0)\geq0$ and $x_0\in\argmax_{\partial\Omega}w(\cdot,t_0)$. From the latter, we see that for a given $\varepsilon_0\in(0,1)$, there exists a constant $R_{\varepsilon_0}>1$ that may depend on $\varepsilon_0\in(0,1)$ but not on $T\in(0,\infty),\ \eta\in(0,1]$ (also not on $x_0\in\partial\Omega$) such that $w>0$ and \eqref{w_n} holds at $(x_0,t_0)$ for $v>R_{\varepsilon_0}$, where $L:=(q+1)(C_0+\varepsilon_0)$.

As in Step 5 of the proof of Theorem \ref{thm:global-grad}, if $C_0<0$, we see that, by taking $\varepsilon_0=\frac{1}{2}\min\{\frac{1}{2},-\frac{1}{2}C_0\}$, there exists a constant $R>1$ independent of $T\in(0,\infty),\ \eta\in(0,1]$ such that $v(x_0,t_0)\leq R$. Here, we have used the fact that $\frac{\partial w}{\partial \Vec{\mathbf{n}}}(x_0,t_0)\geq0$, as in Step 5 of the proof of Theorem \ref{thm:global-grad}. Using the fact that $w\leq v^{q+1}+Cv$ for some constant $C>0$ depending only on $\|h\|_{C^1(\overline{\Omega})},\ \|\phi\|_{C^0(\overline{\Omega})}$, we consequently see that
$$
e^{-Mt}w(x,t)\leq e^{-Mt_0}w(x_0,t_0)\leq R\qquad\text{for all }(x,t)\in\overline{\Omega}\times[0,T], 
$$
and thus,
$$
w(x,t)\leq Re^{MT}\qquad\text{for all }(x,t)\in\overline{\Omega}\times[0,T].
$$
We achieved our goal accordingly when $C_0<0$. Now we assume the other case when $C_0\geq0$.

Let $B=B(x_c,K_0)$ be the open ball with the center $x_c:=x_0-K_0\Vec{\mathbf{n}}(x_0)$ so that $B\subseteq\Omega$ and $\overline{B}\cap(\mathbb{R}^n\setminus\Omega)=\{x_0\}$. For $x\in\overline{B}$, we let
$$
\rho^0(x)=-\frac{L}{2K_0}|x-x_c|^2+\frac{LK_0}{2}+1.
$$
We then extend the function $\rho^0$ on $\overline{B}$ to a function (keeping the same notation $\rho^0$) on $\mathbb{R}^n$ satisfying the requirement that $\rho^0(x)\geq\frac{1}{2}$ for all $x\in\mathbb{R}^n$, and that $\rho^0(x)$ is $C^{\infty}$ on $\mathbb{R}^n$, a nondecreasing function in $|x-x_c|$. Then, $\rho^0(x_0)=1,\ \frac{\partial \rho^0}{\partial \Vec{\mathbf{n}}}(x_0)=-L$. Hence, for $\varepsilon_0\in(0,1)$, there exists a constant $R_{\varepsilon_0}>1$ that may depend on $\varepsilon_0\in(0,1)$ but not on $T\in(0,\infty),\ \eta\in(0,1]$ such that $w>0$ and \eqref{w_n} at $(x_0,t_0)$ are valid if $v>R_{\varepsilon_0}$, and thus that $\frac{\partial (\rho^0w)}{\partial \Vec{\mathbf{n}}}<0$ at $(x_0,t_0)$ if $v>R_{\varepsilon_0}$. 

Since $\rho^0(z)=\rho^0(x_0)=1$ for all $z\in\partial B$, and by the choice of $(x_0,t_0)$, we have
$$
e^{-Mt}\rho^0(z)w(z,t)\leq e^{-Mt_0}\rho^0(x_0)w(x_0,t_0),\qquad\text{for all }(z,t)\in\partial B\times[0,T]. 
$$
Since $\frac{\partial (\rho^0w)}{\partial \Vec{\mathbf{n}}}(x_0,t_0)<0$, we also have
$$
e^{-Mt_0}\rho^0(x_0)w(x_0,t_0)<e^{-Mt_0}\rho^0(x)w(x,t_0)
$$
for some $x\in B$. Combining these two points, we conclude that a maximizer $(x_1,t_1)$ of $e^{-Mt}\rho^0(x)w(x,t)$ on $\overline{B}\times[0,T]$ occurs only inside $B$, i.e., $x_1$ must be inside $B$.

Let $(x_1,t_1)\in\argmax_{\overline{B}\times[0,T]}e^{-Mt}\rho^0(x)w(x,t)$ with $x_1\in B$. If $t_1=0$, then
$$
\max_{\overline{B}\times[0,T]}e^{-Mt}\rho^0(x)w(x,t)=\rho^0(x_1)w(x_1,0)\leq R
$$
for some constant $R>0$ depending only on $\|u_0\|_{C^1(\overline{\Omega})},\ \|h\|_{C^1(\overline{\Omega})},\ \|\phi\|_{C^0(\overline{\Omega})},\ \Omega$. Here, we have used the fact that $\eta\in(0,1],\ \varepsilon_0\in(0,1)$. It consequently yields that for all $(x,t)\in\overline{\Omega}\times[0,T]$,
\begin{align*}
e^{-Mt}\rho^0(x)w(x,t)=\rho^0(x)(e^{-Mt}w(x,t))&\leq\left(\frac{\rho^0(x)}{\rho^0(x_0)}\right)\rho^0(x_0)(e^{-Mt_0}w(x_0,t_0))\\
&\leq R\max_{\overline{B}\times[0,T]}e^{-Mt}\rho^0(x)w(x,t)\leq R
\end{align*}
since $(x_0,t_0)\in\overline{B}\times[0,T]$ and $\frac{\rho^0(x)}{\rho^0(x_0)}\leq R$ for all $x\in\overline{\Omega}$. Here, constants $R>1$ change side by side.
Then, for all $(x,t)\in\overline{\Omega}\times[0,T]$,
$$
w(x,t)\leq \frac{R}{\rho^0(x)}e^{Mt}\leq Re^{MT}
$$
since $\rho^0(x)\geq\frac{1}{2}$ for all $x\in\mathbb{R}^n$, which proves our goal. Now it remains the case when $t_1>0$.

We fix $(x_1,t_1)\in\argmax_{\overline{B}\times[0,T]}e^{-Mt}\rho^0(x)w(x,t)$ with $x_1\in B$, $t_1>0$. Applying the maximum principle to $\psi=\rho w$ at $(x_1,t_1)$, we obtain
\begin{align*}
0&\geq \frac{1}{(q+1)\rho}\left(\textrm{tr}\{a(Du)D^2\psi\}-\psi_t\right)\\
&=\frac{w}{(q+1)\rho}\textrm{tr}\{a(Du)D^2\rho\}+\frac{2}{(q+1)\rho}\textrm{tr}\{a(Du)Dw\otimes D\rho\}\\
&\qquad\qquad\qquad\qquad\qquad\qquad+\frac{1}{q+1}(\textrm{tr}\{a(Du)D^2w\}-w_t)-\frac{\rho_tw}{(q+1)\rho}
\end{align*}
at $(x_1,t_1)$. Following the same computations up to \eqref{maxprinciplerhofinalr} in Step 6 of the proof of Proposition \ref{prop:exuniquegra}, we see that there exist a constant $C>0$ independent of $T\in(0,\infty),\ \eta\in(0,1],\ \varepsilon_0\in(0,1)$ and a constant $R_{\varepsilon_0}>1$ that may depend on $\varepsilon_0\in(0,1)$ but not on $T\in(0,\infty),\ \eta\in(0,1]$ such that, at $(x_1,t_1)$,
\begin{align*}
0&\geq\frac{w}{(q+1)\rho}\textrm{tr}\{a(Du)D^2\rho\}+\left(\frac{2}{(q+1)\rho}\textrm{tr}\{a(Du)Dw\otimes D\rho\}+(q+1)V_1\right)\nonumber\\
&\qquad\qquad\qquad\qquad\qquad\qquad\qquad+J'_1+J'_2-|Dc|v^{q+1}+\frac{M}{q+1}v^{q+1}-C(v+v^q)
\end{align*}
if $v=v(x_1,t_1)>R_{\varepsilon_0}$, with the same definitions of $J'_1,\ J'_2$ as in Step 6 of the proof of Proposition \ref{prop:exuniquegra}. Here, we have used the fact that $-\frac{\rho_tw}{(q+1)\rho}=\frac{Mw}{q+1}$ and that $w\geq v^{q+1}-(q+1)\|\phi\|_{C^0(\overline{\Omega})}\|h\|_{C^1(\overline{\Omega})}v$. Note that $\frac{D\rho}{\rho}=\frac{D\rho^0}{\rho^0},\ \frac{D^2\rho}{\rho}=\frac{D^2\rho^0}{\rho^0}$, and that $x_1\in\argmax_{\overline{\Omega}}\rho^0(\cdot)w(\cdot,t_1)\cap\Omega$. Therefore, we have $wD\rho^0+\rho^0Dw=0$ at $(x_1,t_1)$. Consequently, there exist a constant $C>0$ independent of $T\in(0,\infty),\ \eta\in(0,1],\ \varepsilon_0\in(0,1)$ and a constant $R_{\varepsilon_0}>1$ that may depend on $\varepsilon_0\in(0,1)$ but not on $T\in(0,\infty),\ \eta\in(0,1]$ such that \eqref{term1rho}, \eqref{term2rho}, \eqref{term3J1'}, \eqref{secondderivative1rho}, \eqref{secondderivative21rho} hold true at $(x_1,t_1)$ if $v>R_{\varepsilon_0}$, and thus that \eqref{J1'}, \eqref{J2'r} hold true at $(x_1,t_1)$ if $v>R_{\varepsilon_0}$.

Hence, there exist a constant $C>0$ independent of $T\in(0,\infty),\ \eta\in(0,1],\ \varepsilon_0\in(0,1)$ and a constant $R_{\varepsilon_0}>1$ that may depend on $\varepsilon_0\in(0,1)$ but not on $T\in(0,\infty),\ \eta\in(0,1]$ such that, at $(x_1,t_1)$,
\begin{align*}
0&\geq\left(\frac{1-\varepsilon_0}{n-1}c^2-|Dc|-(C_0+\varepsilon_0)|c|-\frac{(n-1)(C_0+\varepsilon_0)}{K_0}\right.\\
&\qquad\qquad\qquad\qquad\left.-(q+1)(C_0+\varepsilon_0)^2(1+\varepsilon_0)+\frac{M}{q+1}\right)v^{q+1}-C(v+v^q)-C\varepsilon_0^{-1}v^{1-q}.
\end{align*}
if $v>R_{\varepsilon_0}$. Now, take $\varepsilon_0=\frac{1}{2}$, and take $M>1$ large enough, possible due to the assumption \eqref{assumtion:c}, that
$$
\frac{M}{q+1}-|Dc|-(C_0+\frac{1}{2})|c|-\frac{(n-1)(C_0+\frac{1}{2})}{K_0}-\frac{3}{2}(q+1)(C_0+\frac{1}{2})^2>1,
$$
where $c=c(x,z)$, for all $(x,z)\in\overline{\Omega}\times\mathbb{R}$. Since there exists a constant $R_0>R$ independent of $T\in(0,\infty),\ \eta\in(0,1]$ such that
$$
0<v^{q+1}-C(v+v^q)
$$
for $v>R_0$, it must hold true that $v(x_1,t_1)\leq R_0$. Consequently, for all $(x,t)\in\overline{\Omega}\times[0,T]$,
\begin{align*}
e^{-Mt}w(x,t)\leq e^{-Mt_0}w(x_0,t_0)&=e^{-Mt_0}\rho^0(x_0)w(x_0,t_0)\frac{1}{\rho^0(x_0)}\\
&\leq e^{-Mt_1}\rho^0(x_1)w(x_1,t_1)\frac{1}{\rho^0(x_0)}\leq R.
\end{align*}
Here, we have used the fact that $w\leq v^{q+1}+(q+1)\|\phi\|_{C^0(\overline{\Omega})}\|h\|_{C^1(\overline{\Omega})}v$ so that $w(x_1,t_1)\leq R$. Therefore,
$$
w(x,t)\leq Re^{MT}
$$
for all $(x,t)\in\overline{\Omega}\times[0,T]$, which proves our goal in Case 2. This completes the proof.
\end{proof}

Finally, when $\Omega$ is strictly convex, we can recover gradient estimates in \cite{WWX}. The following proof uses a strictly convex $C^2$ defining function of $\Omega$ when we choose a multiplier.

\begin{proof}[Proof of Corollary \ref{cor:strcvxlevelset}]
In order to prove Corollary \ref{cor:strcvxlevelset}, it suffices to verify following, which is a similar statement to Proposition \ref{prop:exuniquegra}; let $\Omega$ be a $C^{3}$ strictly convex domain. Let $T\in(0,\infty)$, $\eta\in(0,1]$, and let $u=u^{\eta}\in C^{2,\sigma}(\overline{\Omega}\times[0,T])\cap C^{3,\sigma}(\Omega\times(0,T])$ be a solution to \eqref{eq} for some $\sigma\in(0,1)$, now with $c\equiv0$. Then, it holds that
$$
\lVert Du^{\eta}\rVert_{L^{\infty}(\overline{\Omega}\times[0,T])}\leq R,
$$
where $R>1$ is a constant independent of $T\in(0,\infty)$ and of $\eta\in(0,1]$.

Let $g$ be a $C^2$ defining function of $\Omega$ such that $g<0$ in $\Omega$, $g=0$ on $\partial\Omega$, $D^2g\geq k_0I_n$ on $\overline{\Omega}$ for some $k_0>0$, $\sup_{\Omega}|Dg|\leq1$, $\frac{\partial g}{\partial \Vec{\mathbf{n}}}=1$ on $\partial\Omega$. Let $\rho=\gamma g+1$, where $\gamma\in\left(0,\frac{1}{2}\min\{1,\|g\|^{-1}_{C^0(\overline{\Omega})}\}\right)$ so that $\frac{1}{2}\leq\rho\leq1$ on $\overline{\Omega}$.

Let $(x_0,t_0)\in\argmax_{\overline{\Omega}\times[0,T]}\rho w$, where $w$ is defined as in the proof of Proposition \ref{prop:exuniquegra}. Again, our goal is to show $v(x_0,t_0)\leq R$, where $R>1$ is a constant independent of $T\in(0,\infty),\ \eta\in(0,1]$. Once it is shown, then we have a global gradient estimate, as
$$
v(x,t)=\frac{1}{\rho(x)}\rho(x)v(x,t)\leq2\rho(x_0)v(x_0,t_0)\leq R\qquad\text{for all }(x,t)\in\overline{\Omega}\times[0,T],
$$
together with the fact that $\rho\geq\frac{1}{2}$ on $\overline{\Omega}$. In the case of $t_0=0$, we readily have that there exists a constant $R>1$ depending only on $\|u_0\|_{C^1(\overline{\Omega})},\ \|h\|_{C^1(\overline{\Omega})},\ \|\phi\|_{C^0(\overline{\Omega})}$ such that $w(x_0,t_0)=w(x_0,0)\leq R$. Using the fact that $v^{q+1}-(q+1)\|\phi\|_{C^0(\overline{\Omega})}\|h\|_{C^1(\overline{\Omega})}v\leq w$, we see that there exists a constant $R>1$ independent of $T\in(0,\infty),\ \eta\in(0,1]$ such that $v(x_0,t_0)\leq R$, which proves our goal.

We assume the remaining case when $t_0>0$. We again divide the proof into two cases, but we consider the case $x_0\in\partial\Omega$ first, and the case $x_0\in\Omega$ next.

\medskip

\textbf{Case 1.} $x_0\in\partial\Omega$.

\medskip

In this case, it holds that $x_0\in\argmax_{\partial\Omega}w(\cdot,t_0)$ since $\rho\equiv1$ on $\partial\Omega$. Therefore, by the argument of Step 5 of the proof of Proposition \ref{prop:exuniquegra}, we see that for a given $\varepsilon_0\in(0,1)$, there exists a constant $R_{\varepsilon_0}>1$ that may depend on $\varepsilon_0\in(0,1)$ but not on $T\in(0,\infty),\ \eta\in(0,1]$ such that $w>0$ and \eqref{w_n} hold true at $(x_0,t_0)$ if $v>R_{\varepsilon_0}$. Since $\Omega$ is strictly convex, we have $C_0<0$. Take $\varepsilon_0=\frac{1}{2}\min\{1,-C_0\}\in(0,1)$, and choose a constant $R=R_{\varepsilon_0}>1$ accordingly. If $v(x_0,t_0)\leq R$, we achieve our goal, and now we assume that $v(x_0,t_0)>R$ so that $w>0$ and\eqref{w_n} are valid at $(x_0,t_0)$. By replacing $R>1$ by a larger one if necessary, we also have that $w(x_0,t_0)>0$ if $v(x_0,t_0)>R$ (from the fact that $w\geq v^{q+1}-(q+1)\|\phi\|_{C^0(\overline{\Omega})}\|h\|_{C^1(\overline{\Omega})}v>0$).

Note that $L=(q+1)(C_0+\varepsilon_0)\leq\frac{1}{2}(q+1)C_0<0$. Since $x_0\in\argmax_{\overline{\Omega}}\rho(\cdot,t_0) w(\cdot,t_0)$, we have $\frac{\partial (\rho w)}{\partial \Vec{\mathbf{n}}}(x_0,t_0)\geq0$. However, if we choose $\gamma\in\left(0,\frac{1}{2}\min\{1,\|g\|^{-1}_{C^0(\overline{\Omega})},-L\}\right)$ so that $\gamma<-L$, then, by \eqref{w_n},
$$
\frac{\partial (\rho w)}{\partial \Vec{\mathbf{n}}}=\rho\frac{\partial w}{\partial \Vec{\mathbf{n}}}+w\frac{\partial \rho}{\partial \Vec{\mathbf{n}}}<Lw+\gamma w<0.
$$
at $(x_0,t_0)$, which contradicts to $\frac{\partial (\rho w)}{\partial \Vec{\mathbf{n}}}(x_0,t_0)\geq0$. Therefore, it must hold true that $v(x_0,t_0)\leq R$, which proves our goal.

\medskip

\textbf{Case 2.} $x_0\in\Omega$.

\medskip

In this case, a maximizer $(x_0,t_0)$ of $\psi:=\rho w$ happens in $B\times(0,T]$, and thus we can apply the maximum principle, which results in \eqref{maxprinciplerhor}, \eqref{maxprincipleexpandedrhor} at $(x_0,t_0)$. Following the same computations as in Step 6 of the proof of Proposition \ref{prop:exuniquegra}, we have \eqref{Vexpansion} at $(x_0,t_0)$. Fix $\varepsilon_0=\frac{1}{2}$. Then, there exist constants $R>1,\ C>0$ independent of $T\in(0,\infty),\ \eta\in(0,1]$ such that \eqref{maxprinciplesimplifiedrhor}, \eqref{maxprinciplerhofinalr} are true at $(x_0,t_0)$ if $v>R$ with the same definitions of $J'_1,\ J'_2$ and $\varepsilon_0=\frac{1}{2}$, $c\equiv0$. Now that we have chosen a multiplier different from the one in the proof of Proposition \ref{prop:exuniquegra}, we estimate the first term and the second term of \eqref{maxprinciplerhofinalr}, which will replace \eqref{term1rho} and \eqref{term2rho}, respectively.

We start with the first term of \eqref{maxprinciplerhofinalr}. Since $\rho=\gamma g+1$ and $D^2\rho\geq\gamma k_0 I_n$, there exists a constant $C>0$ independent of $T\in(0,\infty),\ \eta\in(0,1]$ such that
\begin{align}
&\ \ \ \frac{w}{(q+1)\rho}\textrm{tr}\{a(Du)D^2\rho\}\nonumber\\
&\geq\frac{1}{(q+1)\rho}\gamma k_0\left(\frac{\eta^2}{v^2}+n-1\right)(v^{q+1}-(q+1)\|\phi\|_{C^0(\overline{\Omega})}\|h\|_{C^1(\overline{\Omega})}v)\nonumber\\
&\geq\frac{(n-1)\gamma k_0}{q+1}v^{q+1}-C(v+v^q)\label{term1rhostrcvx}
\end{align}
at $(x_0,t_0)$ for $v>1$. Here, we have used the fact that $\eta\in(0,1]$ and that $\frac{1}{2}\leq\rho\leq1$ on $\overline{\Omega}$.

We estimate the second term of \eqref{maxprinciplerhofinalr}. Since $D\rho=\gamma Dg$ and $|Dg|\leq1,\ \rho\geq\frac{1}{2}$ on $\overline{\Omega}$, we have
$$
0\leq\textrm{tr}\left\{a(Du)\frac{D\rho}{\rho}\otimes\frac{D\rho}{\rho}\right\}\leq\left|\frac{D\rho}{\rho}\right|\leq4\gamma^2.
$$
Choose a constant $R>1$ independent of $T\in(0,\infty),\ \eta\in(0,1]$ such that $|wv^{-q-1}-1|<\frac{1}{2}$ for $v>R$. Then, we see that there exist a constant $R>1,\ C>0$ independent of $T\in(0,\infty),\ \eta\in(0,1]$ such that
\begin{align}
&\ \ \ \frac{2}{(q+1)\rho}\textrm{tr}\{a(Du)Dw\otimes D\rho\}+(q+1)V_1\nonumber\\
&=\frac{(wv^{-1-q}-2)w}{q+1}\textrm{tr}\left\{a(Du)\frac{D\rho}{\rho}\otimes\frac{D\rho}{\rho}\right\}\nonumber\\
&\geq-\frac{6\gamma^2}{q+1}v^{q+1}-Cv\label{term2rhostrcvx}
\end{align}
at $(x_0,t_0)$ for $v>R$.

Following the computations of Step 7 of the proof of Proposition \ref{prop:exuniquegra}, we see that
there exist constants $R>1,\ C>0$ independent of $T\in(0,\infty),\ \eta\in(0,1]$ such that \eqref{term1J1'}, \eqref{J1'} hold at $(x_0,t_0)$ for $v>R$ with $\varepsilon_0=\frac{1}{2},\ c=0$. Note that the left hand side of \eqref{term3J1'} is zero, as $c=0$. As $\rho Dw+wD\rho=0$ at $(x_0,t_0)$, \eqref{secondderivative1rho}, \eqref{secondderivative21rho} are valid at $(x_0,t_0)$, and therefore, there exists constants $R>1,\ C>0$ independent of $T\in(0,\infty),\ \eta\in(0,1]$ such that \eqref{J2'r} holds at $(x_0,t_0)$ if $v>R$ with $\varepsilon_0=\frac{1}{2}$.

All in all, by \eqref{maxprinciplerhofinalr}, \eqref{J1'}, \eqref{J2'r}, \eqref{term1rhostrcvx}, \eqref{term2rhostrcvx}, there exist constants $R>1,\ C>0$ independent of $T\in(0,\infty),\ \eta\in(0,1]$ such that
\begin{align}
0\geq\frac{\gamma}{q+1}((n-1)k_0-6\gamma)v^{q+1}-C(v+v^q)\nonumber
\end{align}
at $(x_0,t_0)$ for $v>R$. Choose $\gamma=\frac{1}{4}\min\left\{1,\|g\|^{-1}_{C^0(\overline{\Omega})},\frac{(n-1)k_0}{6}\right\}\in(0,1)$ so that $\frac{\gamma}{q+1}((n-1)k_0-6\gamma)\geq\frac{3(n-1)k_0}{16(q+1)}>0$. Since there exists a constant $R_0>R$ independent of $T\in(0,\infty),\ \eta\in(0,1]$ such that
$$
0<\frac{3(n-1)k_0}{16(q+1)}v^{q+1}-C(v+v^q)
$$
for $v>R_0$, it must hold that $v=v(x_0,t_0)\leq R_0$, which proves our goal in Case 2. This completes the proof.
\end{proof}


\section{The additive eigenvalue problem}\label{sec:homogenization}
In this section, we prove Theorem \ref{thm:eigengraph}, Theorem \ref{thm:asympgraph} and Theorem \ref{thm:eigenlevelset}. We leave the main reference \cite{WWX}, and we will highlight details that are different from \cite{WWX}. We also refer to \cite[Section 7]{Hung} that go through the limit $k\to0$ first and $\eta\to0$ next.

We consider
\begin{equation}\label{eq:eigenlevelset1}
\begin{cases}
-\sum_{i,j=1}^n\left(\delta^{ij}-\frac{u_iu_j}{\eta^2+|Du|^2}\right)u_{ij}-c(x)\sqrt{\eta^2+|Du|^2}+f(x)=-ku \quad &\text{ in } \Omega,\\
\displaystyle \frac{\partial u}{\partial \Vec{\mathbf{n}}}=\phi(x)v^{1-q}\quad &\text{ on } \partial\Omega,
\end{cases}
\end{equation}
where $k\in(0,1),\ \eta\in(0,1]$ and $v=\sqrt{\eta^2+|Du|^2}$. Note that the choices $\eta=1,\ q>0$ and $\eta=0,\ q=1$ correspond to \eqref{eq:eigengraph} and \eqref{eq:eigenlevelset}, respectively. The case $\eta=0,\ q=1$ will be studied by obtaining estimates uniform in $\eta\in(0,1]$ when $q=1$.

First of all, we start with \emph{a priori} $C^0$ and $C^1$ estimates and get the existence of solutions of \eqref{eq:eigenlevelset1} using the method of continuity with the estimates.

\begin{proposition}\label{prop:C0C1estimate}
Let $\Omega$ be a $C^{\infty}$ bounded domain in $\mathbb{R}^n,\ n\geq2$. Assume that $c\in C^{\infty}(\overline{\Omega})$ satisfies \eqref{condition:c}. Then there exists a unique solution $u\in C^{\infty}(\overline{\Omega})$ of \eqref{eq:eigenlevelset1}. Moreover, we have the following estimate uniform in $k\in(0,1)$ and also in $\eta\in(0,1]$ when $q=1$;
$$
\sup_{\overline{\Omega}}|ku|+\sup_{\overline{\Omega}}|Du|\leq R,
$$
where $R>1$ is a constant independent of $k\in(0,1)$ and also of $\eta\in(0,1]$ when $q=1$.
\end{proposition}
\begin{proof}
We apply Leray-Schauder fixed point theorem to the following family of boundary value problems, parametrized by $\tau\in[0,1]$,
\begin{equation}\label{eq:eigenlevelsetfamily}
\begin{cases}
\tau\left(-\textrm{tr}\{a(Du)D^2u\}-c(x)\sqrt{\eta^2+|Du|^2}+f(x)+ku\right)\\
\quad+(1-\tau)\left(-\textrm{tr}\{a(Du)D^2u\}-c(x)\sqrt{\eta^2+|Du|^2}+\eta c(x)+ku\right)=0 \quad &\text{ in } \Omega,\\
\displaystyle \frac{\partial u}{\partial \Vec{\mathbf{n}}}=\tau\phi(x)v^{1-q}\quad &\text{ on } \partial\Omega,
\end{cases}
\end{equation}
where $a(p):=I_n-\frac{p\otimes p}{\eta^2+|p|^2}$ for $p\in\mathbb{R}^n$. When $\tau=0$, $u\equiv0$ is a solution, and we need to find a solution when $\tau=1$. By Leray-Schauder fixed point theorem, the existence of a solution $u$ when $\tau=1$ can be shown by establishing \emph{a priori} $C^0$ and $C^1$ estimates, uniform in $\tau\in[0,1]$,
$$
\sup_{\overline{\Omega}}|ku|+\sup_{\overline{\Omega}}|Du|\leq R,
$$
which is also uniform in $k\in(0,1)$, and also in $\eta\in(0,1]$ when $q=1$.


Let $u\in C^2(\overline{\Omega})\cap C^3(\Omega)$ be a solution of \eqref{eq:eigenlevelsetfamily}. We first get \emph{a priori} $C^0$ estimate, as it is used to obtain \emph{a priori} $C^1$ estimate. A $C^0$ estimate can be obtained as before. Consider a smooth function $g$ on $\overline{\Omega}$ that has a so large positive slope in the outward normal direction on the boundary that $\left(\sqrt{\eta^2+|Dg|^2}\right)^{q-1}\frac{\partial g}{\partial \Vec{\mathbf{n}}}>\sup_{\overline{\Omega}}|\phi|$ on $\partial\Omega$. Note again that we are dealing with $\eta=1,\ q>0$ for the graph case and with $\eta\in(0,1],\ q=1$ for the level-set case.

We prove \emph{a priori} $C^0$ estimate, and we first check that $g-u$ attains a minimum inside $\Omega$ for \emph{a priori} $C^0$ estimate. Suppose not, and say $x_0\in\partial\Omega$ is a minimizer of $g-u$. Then, at $x_0\in\partial\Omega$, we have $0<\frac{\partial g}{\partial \Vec{\mathbf{n}}}\leq\frac{\partial u}{\partial \Vec{\mathbf{n}}}$ and $D'g=D'u$. The latter follows from $\nabla'g=\nabla'u$ at $x_0$ and Lemma \ref{lem:boundary}, in the notations introduced in Step 5 of the proof of Proposition \ref{prop:exuniquegra}. Using the fact that for a fixed $a\in\mathbb{R}$, the function $\left(\sqrt{\eta^2+a^2+b^2}\right)^{q-1}b$ is monotonically increasing in $b>0$ when $\eta=1,\ q>0$ and also when $\eta\in(0,1],\ q=1$, we see that
$$
\left(\sqrt{\eta^2+|Dg|^2}\right)^{q-1}\frac{\partial g}{\partial \Vec{\mathbf{n}}}\leq\left(\sqrt{\eta^2+|Du|^2}\right)^{q-1}\frac{\partial u}{\partial \Vec{\mathbf{n}}}=\phi(x_0)
$$
at $x_0\in\partial\Omega$. This contradicts with the choice of a function $g$.

Let $x_0\in\Omega$ be a minimizer of $g-u$. Applying the maximum principle at $x_0$ to $g-u$, i.e., $Dg(x_0)=Du(x_0),\ D^2g(x_0)\geq D^2u(x_0)$, we see that, at $x_0$,
\begin{align*}
C\geq \textrm{tr}\{a(Dg)D^2g\}&\geq \textrm{tr}\{a(Du)D^2u\}\\
&=ku-c\sqrt{\eta^2+|Du|^2}+\tau f+(1-\tau)\eta c\\
&=ku-c\sqrt{\eta^2+|Dg|^2}+\tau f+(1-\tau)\eta c\\
&\geq ku-C,
\end{align*}
for some constant $C>0$ depending only on $\Omega,\ g,\ f,\ c$. Here, we have used the fact that $\tau,\eta\in[0,1]$ and the assumptions \eqref{assumtion:c}, \eqref{assumtion:f}. Therefore, for all $x\in\overline{\Omega}$,
$$
ku(x)\leq kg(x)-kg(x_0)+ku(x_0)\leq R
$$
for some constant $R>1$ uniform in $\tau\in[0,1],\ k\in(0,1)$, and also in $\eta\in(0,1]$ when $q=1$. Similarly, we can get a lower bound of $ku(x)$.

A $C^1$ estimate can be established similarly as in the proof of Proposition \ref{prop:exuniquegra}, but now with $\widetilde{c}(x,z):=c(x)$, $\widetilde{f}(x,z):=\tau f(x)+(1-\tau)\eta c(x)+kz$ and $\widetilde{\phi}(x):=\tau\phi(x)$ for $x\in\overline{\Omega},\ z\in\mathbb{R}$. Equation \eqref{eq:eigenlevelsetfamily} can be written as
\begin{equation}\label{eq:eigenlevelsetfamilysimpler}
\begin{cases}
\textrm{tr}\{a(Du)D^2u\}+\widetilde{c}(x,u)v-\widetilde{f}(x,u)=0\quad &\text{ in } \Omega,\\
\displaystyle \frac{\partial u}{\partial \Vec{\mathbf{n}}}=\widetilde{\phi}(x)v^{1-q}\quad &\text{ on } \partial\Omega.
\end{cases}
\end{equation}
The force $\widetilde{c}(x,z)=c(x)$ is in $C^{1,\alpha}(\overline{\Omega})$ and satisfies \eqref{assumtion:c}, \eqref{condition:c}. Also, $\widetilde{\phi}(x)$ is in $C^3(\overline{\Omega})$ with a $C^3$ norm uniform in $\tau\in[0,1]$. Moreover, $\widetilde{f}(x,u)=\tau f(x)+(1-\tau)\eta c(x)+ku$ is \emph{a priori} in $C^{1,\alpha}(\overline{\Omega}\times\mathbb{R})$ and \emph{a priori} satisfies \eqref{assumtion:f} with a constant $C>0$ independent of $\tau\in[0,1],\ k\in(0,1)$ and of $\eta\in(0,1]$ when $q=1$.

We now prove a \emph{a priori} $C^1$ estimate. Throughout the remaining part of the proof, $R>1,\ C>0$ denote constants, which may vary from line to line, independent of $\tau\in[0,1],\ k\in(0,1)$ and also of $\eta\in(0,1]$ when $q=1$. Let $h$ be a function in $C^3(\overline{\Omega})$ such that $h\equiv C,\ Dh=\Vec{\mathbf{n}}$ on the boundary $\partial\Omega$ for some constant $C$. Let $v=\sqrt{\eta^2+|Du|^2}$ and let $w=v^{q+1}-(q+1)\widetilde{\phi} Du\cdot Dh$ on $\overline{\Omega}$.

The proof is similar to that of Proposition \ref{prop:exuniquegra}. We use the idea and the estimate from the proof of Proposition \ref{prop:exuniquegra}, and we highlight the difference coming from not having the time derivative involved.

Let $x_0\in\argmax_{\overline{\Omega}}w$. The goal is to show that $v(x_0)\leq R$ for some constant $R>1$ independent of $\tau\in[0,1],\ k\in(0,1)$ and of $\eta\in(0,1]$ when $q=1$. We again divide the proof into two cases when $x_0\in\Omega$ and when $x_0\in\partial\Omega$.

\medskip

\noindent {\bf Case 1: $x_0\in\Omega$.}

\medskip

At $x_0$, we apply the maximum principle to $w$ to obtain
$$
0\geq\frac{1}{q+1}\textrm{tr}\{a(Du)D^2w\},
$$
which leads to
$$
0\geq\textrm{tr}\{a(Du)D(v^qDv)\}-\textrm{tr}\{a(Du)D^2(\widetilde{\phi}Du\cdot Dh)\}
$$
at $x_0$. Write $0=G+\widetilde{c}v-\widetilde{f}$, where $G:=\textrm{tr}\{a(Du)D^2u\}$. Then,
$$
(v^{q-1}Du-\widetilde{\phi}Dh)\cdot(DG+D(\widetilde{c}v-\widetilde{f}))=0,
$$
and thus, we have \eqref{maxprincipleexpandedr} at $x_0$ with $\widetilde{c},\widetilde{f},\widetilde{\phi}$ instead of $c,f,\phi$.

We proceed the same estimate as in Case 1 of the proof of Proposition \ref{prop:exuniquegra}, except for the part we remark here that with $\alpha=\sqrt{a} D^2u$, $\beta=\sqrt{a}$,
\begin{align*}
\textrm{tr}\{a(Du)(D^2u)^2\}=\|\alpha\|^2&\geq\frac{\textrm{tr}\{\alpha\beta^{\textrm{Tr}}\}^2}{\|\beta\|^2}=\frac{G^2}{n-1+\frac{\eta^2}{v^2}}\nonumber\\
&=\left(\frac{1}{n-1}-\frac{\eta^2}{v^2(n-1)\left(n-1+\frac{\eta^2}{v^2}\right)}\right)(\widetilde{c}v-\widetilde{f})^2\\
&\geq\frac{1}{n-1}\widetilde{c}^2v^2-Cv
\end{align*}
for some constant $C>0$ depending only on $\|f\|_{C^0(\overline{\Omega}\times\mathbb{R})},\|c\|_{C^0(\overline{\Omega}\times\mathbb{R})}$. We have used \emph{a priori} $C^0$ estimate, the assumptions \eqref{assumtion:c}, \eqref{assumtion:f} and the fact that $\tau\in[0,1]$, $\eta\in(0,1]$ when $q=1$.
Therefore, there exists a constant $R>1$ independent of $\tau\in[0,1],\ k\in(0,1)$ and of $\eta\in(0,1]$ when $q=1$ such that whenever $v>R$, \eqref{term1J1} holds.

Also, we have \eqref{term3J1}, \eqref{secondderivative1}, \eqref{secondderivative21} at $x_0$ for some constant $C>0$ independent of $\tau\in[0,1],\ k\in(0,1)$ and of $\eta\in(0,1]$ when $q=1$, since $Dw=0$ at $x_0$. Following the same argument in Case 1 of the proof of Proposition \ref{prop:exuniquegra}, i.e., as in \eqref{J1r}, \eqref{J2r}, we see that for $\varepsilon\in(0,1)$ there exist constants $R>1,\ C>0$ independent of $\tau\in[0,1],\ k\in(0,1)$ and of $\eta\in(0,1]$ when $q=1$ such that
$$
0\geq\frac{\delta}{2}v^{q+1}-C(v+v^q)
$$
at $x_0$ if $v>R_{\varepsilon}$. As there is a constant $R_0>R$ independent of $\tau\in[0,1],\ k\in(0,1)$ and not of $\eta\in(0,1]$ when $q=1$ such that
$$
0<\frac{\delta}{2}v^{q+1}-C(v+v^q)
$$
if $v>R_0$, it must hold that $v=v(x_0)\leq R_0$, which finishes Case 1.

\medskip

\noindent {\bf Case 2: $x_0\in\partial\Omega$.}

\medskip

We see that Step 5 of the proof of Proposition \ref{prop:exuniquegra} carries over verbatim, since the time $t=t_0$ is fixed throughout the step, and since $x_0$ is a maximizer of $w$ on $\overline{\Omega}$. Therefore, for each $\varepsilon_0\in(0,1)$, there exists $R_{\varepsilon_0}>1$ that may depend on $\varepsilon_0$ but not on $\tau\in[0,1],\ k\in(0,1)$ and of $\eta\in(0,1]$ when $q=1$ such that $w>0$ and
$$
\frac{\partial w}{\partial \Vec{\mathbf{n}}}<Lw
$$
at $x_0$ for $v>R_{\varepsilon_0}$, where $L:=(q+1)(C_0+\varepsilon_0)$. We also see that if $C_0<0$, then $v(x_0)\leq R$ for some constant $R>1$ independent of $\tau\in[0,1],\ k\in(0,1)$ and of $\eta\in(0,1]$ when $q=1$, by the argument at the end of Step 5 of the proof of Proposition \ref{prop:exuniquegra}, and thus, we achieve the goal in this case. Therefore, we assume that $C_0\geq0$, and thus that $L\geq0$.

Let $B=B(x_c,K_0)$ be the open ball with the center $x_c:=x_0-K_0\Vec{\mathbf{n}}(x_0)$ so that $B\subseteq\Omega$ and $\overline{B}\cap(\mathbb{R}^n\setminus\Omega)=\{x_0\}$. Let $\psi:=\rho w$, with
$$
\rho(x):=-\frac{L}{2K_0}|x-x_c|^2+\frac{LK_0}{2}+1
$$
as before. Then, $\rho(x_0)=1$, $\frac{\partial \rho}{\partial \Vec{\mathbf{n}}}(x_0)=-L$, and thus,
$$
\frac{\partial \psi}{\partial \Vec{\mathbf{n}}}=\rho\frac{\partial w}{\partial \Vec{\mathbf{n}}}+w\frac{\partial \rho}{\partial \Vec{\mathbf{n}}}=\frac{\partial w}{\partial \Vec{\mathbf{n}}}+(-L)w<0,\qquad\text{at }x_0.
$$
Since $\rho(z)w(z)\leq\rho(x_0)w(x_0)$ for all $z\in\partial B$ from $\rho\equiv1$ on $\partial B$, and since $\frac{\partial \psi}{\partial \Vec{\mathbf{n}}}(x_0)<0$, we derive that $x_1\in B$ for $x_1\in\argmax_{\overline{B}}\psi$. As in Step 5 of the proof of Proposition \ref{prop:exuniquegra}, we see that there exists a constant $C>0$ depending only on $\|\phi\|_{C^0(\overline{\Omega})},\ \|h\|_{C^1(\overline{\Omega})}$ such that the condition $v(x_0)>R_{\varepsilon_0}$ with $R_{\varepsilon_0}>(8C)^{\frac{1}{q+1}}$ implies the condition $v(x_1)>\left(\frac{1}{4C}\right)^{\frac{1}{q+1}}R_{\varepsilon_0}=:R'_{\varepsilon_0}$. Writing $R'_{\varepsilon_0}=\left(\frac{1}{4C}\right)^{\frac{1}{q+1}}R_{\varepsilon_0},$ $R_{\varepsilon_0}=\left(4C\right)^{\frac{1}{q+1}}R'_{\varepsilon_0}$ (and also for $R,\ R'$ similarly), we can state equivalently that if $v(x_1)\leq R'_{\varepsilon_0}$, then $v(x_0)\leq \max\left\{R_{\varepsilon_0},(8C)^{\frac{1}{q+1}}\right\}$. Accordingly, we change our goal from verifying $v(x_0)\leq R$ to proving $v(x_1)\leq R'$.

Fix $x_1\in\argmax_{\overline{B}}\psi\cap B$. At $x_1$,
\begin{align*}
0&\geq\frac{1}{(q+1)\rho}\textrm{tr}\{a(Du)D^2\psi\}\\
&=\frac{w}{(q+1)\rho}\textrm{tr}\{a(Du)D^2\rho\}+\frac{2}{(q+1)\rho}\textrm{tr}\{a(Du)Dw\otimes D\rho\}+\frac{1}{q+1}\textrm{tr}\{a(Du)D^2w\},
\end{align*}
Since $D\psi=\rho Dw+wD\rho=0$ at $x_1$, we have \eqref{Vexpansion} at $x_1$. Also, since
$$
(v^{q-1}Du-\widetilde{\phi}Dh)\cdot(DG+D(\widetilde{c}v-\widetilde{f}))=0,
$$
there exist a constant $R'_{\varepsilon_0}>1$ that may depend on $\varepsilon_0\in(0,1)$ but not on $\tau\in[0,1],\ k\in(0,1)$ and not on $\eta\in(0,1]$ when $q=1$ and a constant $C>0$ independent of $\tau\in[0,1],\ k\in(0,1)$ and of $\eta\in(0,1]$ when $q=1$ such that \eqref{maxprinciplerhofinalr} holds true at $x_1$ for $v>R'_{\varepsilon_0}$. Following the same computations in Step 7 of the proof of Proposition \ref{prop:exuniquegra}, we see that there exist a constant $R'_{\varepsilon_0}>1$ that may depend on $\varepsilon_0\in(0,1)$ but not on $\tau\in[0,1],\ k\in(0,1)$ and not on $\eta\in(0,1]$ when $q=1$ and a constant $C>0$ independent of $\tau\in[0,1],\ k\in(0,1)$ and of $\eta\in(0,1]$ when $q=1$ such that \eqref{term1rho}, \eqref{term2rho}, \eqref{J1'}, \eqref{secondderivative1rho}, \eqref{secondderivative21rho}, \eqref{J2'r} at $x_1$ for $v>R'_{\varepsilon_0}$.

All in all, choosing $\varepsilon_0\in(0,1)$ as in Step 8 of the proof of Proposition \ref{prop:exuniquegra}, we see that there exist constants $R'>1,\ C>0$ independent of $\tau\in[0,1],\ k\in(0,1)$ and of $\eta\in(0,1]$ when $q=1$ such that
$$
0\geq\frac{\delta}{2}v^{q+1}-C(v+v^q)
$$
at $x_1$ if $v>R'$. There is, on the other hand, also a constant $R'_0>R'$ independent of $\tau\in[0,1],\ k\in(0,1)$ and of $\eta\in(0,1]$ when $q=1$ such that
$$
0<\frac{\delta}{2}v^{q+1}-C(v+v^q)
$$
if $v>R'_0$. Therefore, it must hold that $v=v(x_1)\leq R'_0$, which completes Case 2.

\medskip

All in all, we have obtained \emph{a priori} $C^0$ and $C^1$ estimates for $u\in C^2(\overline{\Omega})\cap C^3(\Omega)$ solving \eqref{eq:eigenlevelset1}, and thus the existence of a solution $u\in C^{\infty}(\overline{\Omega})$ of \eqref{eq:eigenlevelset1} by Leray-Schauder fixed point theorem. The higher regularity and the fixed point theorem are referred to \cite{LT}.

\medskip

For the rest of the proof, we refer to the proof of \cite[Theorem 4.2]{WWX} for more details and the uniqueness upto an additive constant.
\end{proof}

Take $\eta=1,\ q>0$ to prove Theorem \ref{thm:eigengraph} and Theorem \ref{thm:asympgraph}.

\begin{proof}[Proof of Theorem \ref{thm:eigengraph} and Theorem \ref{thm:asympgraph}]
For each $k\in(0,1)$, let $u_k$ be the solution of \eqref{eq:eigen} with $\eta=1,\ q>0$. Then the function $w_k=u_k-\frac{\int_{\Omega}u_k}{|\Omega|}$ solves
\begin{equation}\label{eq:eigengraph3}
\begin{cases}
-a(Dw_k):D^2w_k-c(x)\sqrt{\eta^2+|D w_k|^2}+f(x)=-kw_k-k\frac{\int_{\Omega}u_k}{|\Omega|} \quad &\text{ in } \Omega,\\
\displaystyle \frac{\partial w_k}{\partial \Vec{\mathbf{n}}}=\phi(x)\left(\sqrt{1+|Dw_k|^2}\right)^{1-q}\quad &\text{ on } \partial\Omega.
\end{cases}
\end{equation}
Then we have that $\sup|w_k|+\sup|Dw_k|\leq R$. By Schauder theory, there is an exponent $\alpha\in(0,1)$ such that $\|w_{k}\|_{C^{2,\alpha}(\overline{\Omega})}\leq R$. Therefore, $w_k\to w$ in $C^{2,\alpha'}$ for some $\alpha'\in(0,\alpha)$, and $-kw_k-k\frac{\int_{\Omega}u_k}{|\Omega|}\to-\lambda$ where $(\lambda,w)$ solves \eqref{eq:eigengraph}.

See the proof of \cite[Theorem 4.2]{WWX} for more details and the uniqueness upto an additive constant. The proof of Theorem \ref{thm:asympgraph} goes the same as that of \cite[Theorem 5.1]{WWX}.
\end{proof}

Now, we study \eqref{eq:eigenlevelset} by vanishing viscosity procedure $\eta\to0$ when $q=1$.

\begin{proposition}
Let $\Omega$ be a $C^{\infty}$ bounded domain in $\mathbb{R}^n,\ n\geq2$. Let $\eta\in(0,1]$. Assume $c\in C^{\infty}(\overline{\Omega})$ satisfies \eqref{condition:c}. Then, there exists a unique $\lambda_{\eta}\in\mathbb{R}$ such that there exists a solution $w\in C^{\infty}(\overline{\Omega})$ of
\begin{equation}\label{eq:eigenlevelset2}
\begin{cases}
-\sum_{i,j=1}^n\left(\delta^{ij}-\frac{w_iw_j}{\eta^2+|Dw|^2}\right)w_{ij}-c\sqrt{\eta^2+|Dw|^2}+f=-\lambda_{\eta} \quad &\text{ in } \Omega,\\
\displaystyle \frac{\partial w}{\partial \Vec{\mathbf{n}}}=\phi(x)\quad &\text{ on } \partial\Omega.
\end{cases}
\end{equation}
Moreover, a solution $w$ is unique upto an additive constant, and we have the following estimate uniform in $\eta\in(0,1]$;
\begin{align}\label{uniformineta}
|\lambda_{\eta}|+\sup_{\overline{\Omega}}|Dw|\leq R,
\end{align}
where $R>0$ is a constant not depending on $\eta\in(0,1]$.
\end{proposition}
\begin{proof}
We proceed the same limit process as $k\to0$ as in the proof of Theorem \ref{thm:eigengraph}. Note that the estimates are uniform in $\eta\in(0,1]$ when $q=1$.
\end{proof}

\begin{proof}[Proof of Theorem \ref{thm:eigenlevelset}]
Fix $x_0\in\Omega$. For each $\eta\in(0,1]$, let $(\lambda_{\eta},w_{\eta})$ be a pair that solves \eqref{eq:eigenlevelset2} with $w_{\eta}(x_0)=0$. By \eqref{uniformineta} and Arzela-Ascoli Theorem, as $\eta\to0$, we can find a subsequence of $(\lambda_{\eta},w_{\eta})$ such that $\lambda_{\eta}$ converges to $\lambda\in\mathbb{R}$, and $w_{\eta}$ converges to a Lipschitz function $w$ uniformly on $\overline{\Omega}$. By the stability of viscosity solutions, we see that $(\lambda,w)$ solves \eqref{eq:eigenlevelset}.

Let $u$ be the unique viscosity solution of \eqref{eq:levelset}. Then for some constant $C>0$, $w(x)-C+\lambda t$ and $w(x)+C+\lambda t$ are a subsolution and supersolution of \eqref{eq:levelset}, respectively. By the comparison principle (Proposition \ref{prop:comp}) for \eqref{eq:levelset}, we have
$$
w(x)-C+\lambda t\leq u(x,t)\leq w(x)+C+\lambda t.
$$
Therefore, we can draw the conclusion that $\lambda=\lim_{t\to\infty}\frac{u(x,t)}{t}$ and that the convergence is uniform in $x\in\overline{\Omega}$. The uniqueness of such a number $\lambda\in\mathbb{R}$ follows from the uniqueness of a solution $u$ of \eqref{eq:levelset} and the limit $\lambda=\lim_{t\to\infty}\frac{u(x,t)}{t}$.
\end{proof}

\section{Radially symmetric cases}\label{sec:radial}

In this section, we study the radially symmetric setting of \eqref{eq:levelset}. We find the Lagrangian, the optimal control formula and a counterexample of the condition \eqref{condition:c} in Subsection \ref{subsec:4.1}, and we define the \emph{Aubry set}, prove the comparison principle on the Aubry set and prove Theorem \ref{thm:radprofile} in Subsection \ref{subsec:4.2}. We mention an example of nonuniqueness for \eqref{eq:levelset} when $0<q<1$ at the end of this section. We leave the reference \cite{GMT, JKMT} for the analysis of the radially symmetric setting, and \cite{Hung} for Aubry sets.

We always assume here that, by abuse of notations,
\begin{equation}\label{con:radial}
\begin{cases}
\hspace{6.25mm}\Omega = B(0,R) &\text{ for some } R>0,\\
\hspace{2mm}c(x) = c(r) &\text{ for } |x|=r \in [0, R],\\
\hspace{1.25mm}f(x) = f(r) &\text{ for } |x|=r \in [0, R],\\
\hspace{1.5mm}\phi(x) = \phi(r) &\text{ for } |x|=r \in [0, R],\\
u_0(x) = u_0(r)  &\text{ for } |x|=r \in [0, R].
\end{cases}
\end{equation}
Here, $R>0$ is a fixed positive number, $c \in C^1([0,R], [0,\infty))$, $f\in C^1([0,R])$ and $u_0 \in C^2([0,R])$ with $u_0'(R)=\phi(R)$ are given. The function $\phi(x)$ can be understood as the constant $\phi(R)$.

\subsection{The optimal control formula and a counterexample}\label{subsec:4.1}

Equation \eqref{eq:levelset} becomes 
\begin{equation}\label{eq:radial}
\begin{cases}
\varphi_t-\frac{n-1}{r}\varphi_r-c(r)\lvert \varphi_r\rvert+f(r)=0 \quad &\text{ in } (0,R)\times(0,\infty),\\
\displaystyle \hspace{41mm} \varphi_r(R)=\phi(R)\quad &\text{} \\
\hspace{40mm}\varphi(r,0)=u_0(r) \quad &\text{ for } r\in[0,R].
\end{cases}
\end{equation}

Note that this is a first-order Hamilton-Jacobi equation with a concave Hamiltonian.
The associated Lagrangian $L=L(r,q)$ to the Hamiltonian $H(r,p)=-\frac{n-1}{r}p-c(r)\lvert p\rvert+f(r)$ is
\begin{align*}
L(r,q)&=\inf_{p\in\mathbb{R}}\left\{p\cdot q-\left(-\frac{n-1}{r}p-c(r)\lvert p\rvert+f(r)\right)\right\}\\
&=\inf_{p\in\mathbb{R}}\left\{\left(q+\frac{n-1}{r}\right)p+c(r)\lvert p\rvert-f(r)\right\}\\
&=\left\{\begin{array}{ll}
-f(r),\quad\quad\quad\hspace{1.5mm} \textrm{if}\ \left\lvert q+\frac{n-1}{r}\right\rvert\leq c(r),\\
-\infty,\quad\quad\qquad\hspace{0.5mm}\textrm{otherwise}.
\end{array}\right.
\end{align*}
Therefore, we have the following representation formula for $\phi=\phi(r,t)$
\begin{equation}\label{repnformula}
\varphi(r,t)=\sup\left\{\int_0^t\left(-f(\eta(s))+\phi(\eta(s))l(s)\right)ds+u_0(\eta(t)):\ (\eta,v,l)\in\mathrm{SP}(r)\right\},
\end{equation}
where we denote by $\mathrm{SP}(r)$ the Skorokhod problem. See \cite[Section 4.5]{GMT2} for the derivation of the formula.
For a given $r\in(0,R],\ v\in L^{\infty}([0,t])$, the Skorokhod problem seeks to find a solution $(\eta,l)\in \mathrm{Lip}((0,t))\times L^{\infty}((0,t))$ such that
$$
\left\{\begin{array}{ll}
\eta(0)=r,\quad\quad\hspace{1.5mm} \eta([0,t])\subset(0,R],\\
l(s)\geq0\quad\quad\quad \textrm{for almost every}\ s>0,\\
l(s)=0\quad\quad\quad \textrm{if}\ \eta(s)\neq R,\\
\left|-v(s)+\frac{n-1}{\gamma(s)}\right|\leq c(\gamma(s)),\\
v(s)=\dot{\eta}(s)+l(s)n(\eta(s)),
\end{array}\right.
$$
and the set $\mathrm{SP}(r)$ collects all the associated triples $(\eta,v,l).$
Here, $n(R)=1$ is the outward normal vector  to $(0,R)$ at $R$. 
See \cite[Theorem 4.2]{I2} for the existence of solutions of the Skorokhod problem and \cite[Theorem 5.1]{I2} for the representation formula.
See \cite{GMT} for a related problem on the large time behavior and the large time profile.

We remark that at $\eta(s)\neq R,$ we have
\begin{equation}\label{etadot1}
\frac{n-1}{\eta(s)}-c(\eta(s))\leq\dot{\eta}(s)\leq\frac{n-1}{\eta(s)}+c(\eta(s)),
\end{equation}
and at $\eta(s)=R$,
\begin{equation*}
\frac{n-1}{R}-c(R)\leq\dot{\eta}(s)+l(s)n(R)\leq\frac{n-1}{R}+c(R).
\end{equation*}
This implies that
\begin{equation}\label{etadot2}
\frac{n-1}{R}-c(R)\leq l(s)\leq\frac{n-1}{R}+c(R).
\end{equation}

We will find the eigenvalue $\lambda=\lim_{t\to\infty}\frac{\varphi(r,t)}{t}$ in terms of given functions $c$, $f$ and a constant $\phi(R)$ when the force $c$ satisfies \eqref{condition:c}. Before that, let us see that $\lim_{t\to\infty}\frac{\varphi(r,t)}{t}$ is not constant in $r\in[0,R]$ when $c$ does not satisfy \eqref{condition:c} with the following example.

\begin{example}\label{ex}
We consider a case when $c(r)$ is of the form
$$
c(r)\left\{\begin{array}{ll}
<\frac{n-1}{a},\quad\quad\quad\hspace{1mm} &0\leq r< a,\\
=\frac{n-1}{r},\quad\quad\quad\hspace{1.5mm} &a\leq r\leq b,\\
>\frac{n-1}{b},\quad\quad\quad\hspace{1.5mm} &b< r\leq R,
\end{array}\right.
$$
for some $0<a<b<R$. Let $u_0\equiv0,\ \phi(R)=0$. By \eqref{etadot1}, a curve $\eta(s)$ with $(\eta,v,l)\in\mathrm{SP}(r)$
\begin{itemize}
\item can stay still or go right when $a\leq\eta(s)\leq b$,
\item must go right when $\eta(s)<a$
\item can move both left and right when $\eta(s)>b$.
\end{itemize}
Then, by \eqref{repnformula},
$$
\lim_{t\to\infty}\frac{\varphi(r,t)}{t}=\left\{\begin{array}{ll}
\sup\{-f(s):\ s\geq a\},\quad\quad\quad\hspace{1mm} &r\leq a,\\
\sup\{-f(s):\ s\geq r\},\quad\quad\quad\hspace{1.5mm} &a\leq r\leq b,\\
\sup\{-f(s):\ s\geq b\},\quad\quad\quad\hspace{1.5mm} &r\geq b.
\end{array}\right.
$$
We see that the limit is not constant in $r\in[0,R]$ for a suitable choice of $f$. For instance, take a smooth function $f(r)$ such that
$$
f(r)\left\{\begin{array}{ll}
=1,\quad\quad\quad\hspace{1mm} &0\leq r< a,\\
\in(0,1),\quad\quad\quad\hspace{1.5mm} &a\leq r\leq b,\\
>0,\quad\quad\quad\hspace{1.5mm} &b< r\leq R.
\end{array}\right.
$$
\end{example}

In the above example, the force $c$ does not satisfy \eqref{condition:c}; at $r\in(a,b)$,
$$
\frac{1}{n-1}c(r)^2-|Dc(r)|=\frac{1}{n-1}\left(\frac{n-1}{r}\right)^2-\frac{n-1}{r^2}=0.
$$
Therefore, the condition \eqref{condition:c} is sharp.

\subsection{Aubry set, the comparison principle and the large-time behavior}\label{subsec:4.2}

From now on, we assume that $c$ is coercive, i.e., $c$ satisfies \eqref{condition:c}. Then there is at most one $r$, which we call $r_{cr}$ if  it exists, such that $c(r)=\frac{n-1}{r}$. Otherwise, there would exist two points $a<b$ where the curves $c(r)$ and $\frac{n-1}{r}$ cross. At $r=b$,
$$
\frac{1}{n-1}c(b)^2-|Dc(b)|\leq\frac{1}{n-1}\left(\frac{n-1}{b}\right)^2-\frac{n-1}{b^2}=0,
$$
since $Dc(b)\leq\frac{d}{dr}\bigr\rvert_{b}\left(\frac{n-1}{r}\right)=-\frac{n-1}{b^2}<0.$ If $c(r)<\frac{n-1}{r}$ for all $r\leq R$, we let $r_{cr}:=\infty$.

In the both cases of $r_{cr}<\infty$ and $r_{cr}=\infty$, by \eqref{repnformula} and \eqref{etadot2}, we obtain
\begin{equation*}
\lambda=\sup\left\{-f(r)+\delta(r-R)\phi(R)\left(\frac{n-1}{R}+\mathrm{sgn}(\phi(R))c(R)\right):\ r\geq r_{cr}\textrm{ or }r=R\right\},  
\end{equation*}
which is \eqref{formula:eigenvalue}.

We define the \emph{Aubry set} $\widetilde{\mathcal{A}}$ by
\begin{align*}
\widetilde{\mathcal{A}}:=\left\{r\geq r_{cr}:\ \textrm{the supremum of \eqref{formula:eigenvalue} is attained}\right\}\quad\text{if }r_{cr}<\infty.
\end{align*}
Note that if $r_{cr}<\infty$, then the function $-f(r)+\delta(r-R)\phi(R)\left(\frac{n-1}{R}+\mathrm{sgn}(\phi(R))c(R)\right)$ is upper semicontinuous on the interval $[r_{cr},R]$. Thus, $\widetilde{\mathcal{A}}$ is well-defined, and it is a nonempty closed subset of $[0,R]$. If $r_{cr}=\infty$, we let $\widetilde{\mathcal{A}}=\{R\}$.

Let
\begin{equation}\label{eq:radialeigen}
\begin{cases}
\lambda-\frac{n-1}{r}w_r-c(r)\lvert w_r\rvert+f(r)=0 \quad &\text{ in } (0,R)\times(0,\infty),\\
\displaystyle \hspace{39.5mm} w_r(R)=\phi(R)
\end{cases}
\end{equation}
be the stationary problem of \eqref{eq:radial}. Here, we are assuming that $c$ satisfies \eqref{condition:c}, and thus, the eigenvalue $\lambda$ is given as in \eqref{formula:eigenvalue}.

The propositions in \cite[Section 2]{GMT} follow for \eqref{eq:radial} with little changes. Here, we state \cite[Lemma 2.4]{GMT} and \cite[Theorem 2.5]{GMT} for problem \eqref{eq:radialeigen}.

\begin{proposition}\label{prop:compprior}
Let $w^1,w^2$ be two solutions of \eqref{eq:radialeigen}. Assume that $w^1(r_0)=w^2(r_0)$ and $w^1(M)=w^2(M)$, where $r_0:=\min\{r:r\in\widetilde{\mathcal{A}}\}$ and $M:=\max\{r:r\in\widetilde{\mathcal{A}}\}$. Then $w^1=w^2$ on $[r_{cr},r_0]\cup[M,R]$.
\end{proposition}
\begin{proof}
The only part that changes is where we prove $w^1=w^2$ on $[M,R]$. To prove this, we may assume without loss of generality that $0<r_{cr}<R$ and $M<R$. We claim that $w^1$and $w^2$ cannot have a corner from below in $(M,R)$ so that they agree on $[M,R]$ by \eqref{eq:radialeigen}.

Suppose not, i.e., there would exist $i\in\{1,2\}$ and $y\in[M,R)$ such that
$$
(w^i)_r(r)=\frac{-r(-f(r)-\lambda)}{rc(r)+(n-1)}\qquad\text{for all }r\geq y.
$$
At $r=R$,
$$
\phi(R)=(w^i)_r(R)=\frac{-R(-f(R)-\lambda)}{Rc(R)+(n-1)}.
$$
This means that $\phi(R)\geq0$. However, from the assumption that $R\notin\widetilde{\mathcal{A}}$, we have
$$
-f(R)+\phi(R)\left(\frac{n-1}{R}+c(R)\right)<\lambda,
$$
or,
$$
-f(R)-\lambda<-\phi(R)\left(\frac{n-1}{R}+c(R)\right)<0.
$$
This yields a contradiction, as
$$
\phi(R)=\frac{-R(-f(R)-\lambda)}{Rc(R)+(n-1)}>\frac{-R}{Rc(R)+(n-1)}\cdot\left(-\phi(R)\left(\frac{n-1}{R}+c(R)\right)\right)=\phi(R).
$$
\end{proof}

This proposition implies the following proposition of the uniqueness set property of the Aubry set $\widetilde{\mathcal{A}}$.

\begin{proposition}
The following hold;

\noindent (i) If $w^1,w^2$ are solutions of \eqref{eq:radialeigen} such that $w^1=w^2$ on $\widetilde{\mathcal{A}}$, then $w^1=w^2$ on $[0,R]$.

\noindent (ii) If $w^1$ and $w^2$ are a subsolution and a supersolution of \eqref{eq:radialeigen}, respectively, and if $w^1\leq w^2$ on $\widetilde{\mathcal{A}}$, then $w^1\leq w^2$ on $[0,R]$.
\end{proposition}

Now we prove Theorem \ref{thm:radprofile} based on the uniqueness set property of the Aubry set.
\begin{proof}[Proof of Theorem \ref{thm:radprofile}]
Since we already found the eigenvalue $\lambda$, defined the Aubry set $\widetilde{\mathcal{A}}$ and the number $r_{cr}$ in the preceding discussions, it suffices to prove the asymptotic behavior and to find the large time profile in this proof.

The proof follows almost the same as that of \cite[Theorem 1.1]{GMT}, but we put a extra care on the boundary $r=R$. Following the proof of \cite[Theorem 1.3]{GMT}, we can prove (ii) of Theorem \ref{thm:radprofile} once we prove (i) of Theorem \ref{thm:radprofile}. Thus, it suffices to show that $\varphi(r,t)-\lambda t$ converges as $t\to\infty$ uniformly in $r\in[0,R]$. 

The first case we consider is when $r_{cr}=\infty$. Note that by \eqref{etadot1} every admissible curve $\eta=\eta(s)$, i.e., $(\eta,v,l)\in\mathrm{SP}(r)$ for some $v,l,r$, satisfies
\begin{align}\label{repnformula1}
\dot{\eta}(s)\geq\frac{n-1}{\eta(s)}-c(\eta(s)).
\end{align}

Then $\eta$ always moves to the right with minimal speed $\delta>0$ for some $\delta>0$. Therefore, using the formula \eqref{repnformula},
\begin{equation*}
\varphi(r,t)-\lambda t=\sup\left\{\int_0^t\left(-f(\eta(s))+\phi(\eta(s))l(s)-\lambda\right)ds+u_0(\eta(t)):\ (\eta,v,l)\in\mathrm{SP}(r)\right\}
\end{equation*}
does not change as $t$ varies after $t>\frac{R}{\delta}$.

The second case is when $r_{cr}<\infty$. We claim that for any $r\in\widetilde{\mathcal{A}}$, and for any $t_1\leq t_2$, we have
$$
\varphi(r,t_1)-\lambda t_1\leq\varphi(r,t_2)-\lambda t_2.
$$
Let us write the Skorokhod problem in \eqref{repnformula1} as $\mathrm{SP}(r,t)=\mathrm{SP}(r)$ to show the dependence in $t$. Then a triple $(\eta,v,l)\in\mathrm{SP}(r,t_1)$ induces a triple $(\widetilde{\eta},\widetilde{v},\widetilde{l})\in\mathrm{SP}(r,t_2)$ by means of
$$
(\widetilde{\eta},\widetilde{v},\widetilde{l})(s)=\left\{\begin{array}{ll}
(\eta,v,l)(0),\qquad\qquad\qquad\hspace{3mm}\text{for }0\leq s\leq t_2-t_1,\\
(\eta,v,l)(s-(t_2-t_1)),\qquad\text{for }t_2-t_1\leq s\leq t_2.
\end{array}\right.
$$
This yields
\begin{multline*}
\int_0^{t_2}\left(-f(\widetilde{\eta}(s))+\phi(\widetilde{\eta}(s))\widetilde{l}(s)-\lambda\right)ds+u_0(\widetilde{\eta}(t_2))=\\
\int_0^{t_1}\left(-f(\eta(s))+\phi(\eta(s))l(s)-\lambda\right)ds+u_0(\eta(t_1)),
\end{multline*}
and this is because $r\in\widetilde{\mathcal{A}}$ so that the integrand above is zero while $(\widetilde{\eta},\widetilde{v},\widetilde{l})\in\mathrm{SP}(r,t_2)$ stays still upto $s=t_2-t_1$. This argument of embedding $\mathrm{SP}(r,t_1)$ into $\mathrm{SP}(r,t_2)$ gives, together with \eqref{repnformula1}, that $\varphi(r,t_1)-\lambda t_1\leq\varphi(r,t_2)-\lambda t_2$.

The rest proof follows the same as that of \cite[Theorem 1.1]{GMT}. We also refer to \cite{DS}
\end{proof}

We give an example of nonuniqueness of \eqref{eq:levelset} when $0<q<1$ before we end the section.

\begin{example}\label{ex2}
Consider
\begin{equation}\label{eq:radialeigenpatho}
\begin{cases}
\lambda-\frac{n-1}{r}w_r-c(r)\lvert w_r\rvert+f(r)=0 \quad &\text{ in } (0,R)\times(0,\infty),\\
\displaystyle \hspace{39.5mm} w_r(R)=\phi(R)|w_r(R)|^{1-q},
\end{cases}
\end{equation}
where $0<q<1$. Let $\phi(R)=1$. We also let $f\equiv0$, $c\equiv0$. Then $c$ is coercive by Corollary \ref{cor:strcvxlevelset}.

By the definition of viscosity solutions, we see that the condition $w_r(R)=\phi(R)|w_r(R)|^{1-q}$ is satisfied if $w_r(R)=\mathrm{sgn}(\phi(R))|\phi(R)|^{\frac{1}{q}}$ in the classical sense. Then, one can check that
$\lambda_1=\frac{n-1}{R}$, $w^1(r)=\frac{r^2}{2R}$ solve \eqref{eq:radialeigenpatho}.

Also, if the boundary condition $v_r(R)=0$ is true in the classical sense, then the condition $w_r(R)=\phi(R)|w_r(R)|^{1-q}$ is satisfied in the viscosity sense. Then $\lambda_2=0$, $w^2\equiv C$, where $C$ is a constant, solve \eqref{eq:radialeigenpatho}.

Therefore, we have two distinct eigenvalues admitting a solution, which result in two different solutions $\varphi^i(r,t)=\lambda_it+w^i(x)$, $i=1,2$, of
\begin{equation}\label{eq:radialpatho}
\begin{cases}
\varphi_t-\frac{n-1}{r}\varphi_r-c(r)\lvert \varphi_r\rvert+f(r)=0 \quad &\text{ in } (0,R)\times(0,\infty),\\
\displaystyle \hspace{41mm} \varphi_r(R)=\phi(R)|\varphi_r(R)|^{1-q}\quad &\text{} \\
\hspace{40mm}\varphi(r,0)=u_0(r) \quad &\text{ for } r\in[0,R].
\end{cases}
\end{equation}
\end{example}


\section*{Appendix A}\label{sec:appendix}
In this appendix, we provide the definition of viscosity solutions of \eqref{eq:levelset} and give the results on the comparison principle and the stability under the conditions \eqref{assumtion:c}, \eqref{assumtion:f} on $c,f$, respectively.

\medskip

Let $F:\overline{\Omega}\times\mathbb{R}\times\mathbb{R}^n \setminus \{0\}\times \mathcal{S}_n \rightarrow \mathbb{R}$ be such that
$$
F(x,z,p,X)=\text{trace}\left(\left(I-\frac{p\otimes p}{|p|^2}\right)X\right)+c(x,z)|p|-f(x,z),
$$
where $\mathcal{S}_n$ is the set of square symmetric matrices of size $n$. Together with the assumption that $c_z\leq0,\ f_z\geq0$, we see that $-F$ is degenerate elliptic and proper, i.e.,
$$
-F(x,z,p,X)\leq -F(x,w,p,Y)\quad\textrm{whenever }Y\leq X,\ z\leq w.
$$

Define the lower and upper semicontinuous envelopes of $F$ by, for $(x,z,p,X)\in\overline{\Omega}\times\mathbb{R}\times\mathbb{R}^n\times\mathcal{S}_n$,
$$
F_*(x,z,p,X)=\liminf_{(y,w,q,Y)\rightarrow(x,z,p,X)}F(y,w,q,Y),
$$
and
$$
F^*(x,z,p,X)=\limsup_{(y,w,q,Y)\rightarrow(x,z,p,X)}F(y,w,q,Y),
$$
respectively.

\begin{definition} A function $u:\overline{\Omega}\times [0,\infty)\rightarrow \mathbb{R}$ is said to be a viscosity subsolution (a viscosity supersolution, resp.)of \eqref{eq:levelset}  if
\begin{itemize}
\item $u$ is upper semicontinuous (lower semicontinuous, resp.);

\item for all $x\in \overline{\Omega}$, $u^*(x,0)\leq u_0(x)$ ($u_*(x,0)\geq u_0(x)$, resp.);

\item for any function $\varphi\in C^2(\overline{\Omega}\times[0,\infty))$, if $(\hat{x},\hat{t})\in \overline{\Omega}\times(0,\infty)$ is a maximizer (a minimizer, resp.) of $u-\varphi$, then, at $(\hat{x},\hat{t})$,
\begin{equation*}
\begin{cases}
\varphi_t(\hat{x},\hat{t})-F^*(\hat{x},u(\hat{x},\hat{t}),D\varphi(\hat{x},\hat{t}),D^2\varphi(\hat{x},\hat{t}))\leq 0\quad &\text{ if } \hat{x}\in\Omega,\\
\min\left\{\varphi_t(\hat{x},\hat{t})-F^*(\hat{x},u(\hat{x},\hat{t}),D\varphi(\hat{x},\hat{t}),D^2\varphi(\hat{x},\hat{t})),\frac{\partial\varphi}{\partial\Vec{\mathbf{n}}}(\hat{x},\hat{t})-\phi(\hat{x},\hat{t})\right\}\leq 0\quad &\text{ if } \hat{x}\in\partial\Omega.\\
\end{cases}
\end{equation*}
\begin{equation*}
\left(
\begin{cases}
\varphi_t(\hat{x},\hat{t})-F_*(\hat{x},u(\hat{x},\hat{t}),D\varphi(\hat{x},\hat{t}),D^2\varphi(\hat{x},\hat{t}))\geq 0\quad &\text{ if } \hat{x}\in\Omega,\\
\max\left\{\varphi_t(\hat{x},\hat{t})-F_*(\hat{x},u(\hat{x},\hat{t}),D\varphi(\hat{x},\hat{t}),D^2\varphi(\hat{x},\hat{t})),\frac{\partial\varphi}{\partial\Vec{\mathbf{n}}}(\hat{x},\hat{t})-\phi(\hat{x},\hat{t})\right\}\geq 0\quad &\text{ if } \hat{x}\in\partial\Omega,\ \text{resp.}\\
\end{cases}
\right)
\end{equation*}

\end{itemize}
\noindent A function $u:\overline{\Omega}\times [0,\infty)\rightarrow \mathbb{R}$ is a viscosity solution of \eqref{eq:levelset} if $u$ is both its viscosity subsolution and its viscosity supersolution.
\end{definition}

\begin{proposition}[Comparison principle for  \eqref{eq:levelset}]\label{prop:comp}
Let $\Omega$ be a bounded domain in $\mathbb{R}^n$ with $C^3$ boundary $\partial\Omega$. Suppose that $c,f$ satisfy \eqref{assumtion:c}, \eqref{assumtion:f}, respectively.
Let $u$ be a subsolution and $v$ be a supersolution of  \eqref{eq:levelset}, respectively. 
Then, $u^*\leq v_*$ in $\overline{\Omega}\times[0,\infty)$.
\end{proposition}

We can follow \cite{B} with slight modifications for the comparison principle of viscosity solutions of \eqref{eq:levelset}. We also refer to \cite{CIL, GS}.

\begin{lemma}\label{lem:stable}
Suppose that $u^{\eta}$ is the unique solution of \eqref{eq} for each $\eta>0$, and there exists $u\in C(\ol \Omega \times [0,\infty))$ such that
$$
u^{\eta}\rightarrow u,\quad \text{as}\ \eta\rightarrow 0,
$$
uniformly on $\overline{\Omega}\times[0,T)$ for each $T>0$.
Then $u$ is the unique viscosity solution of \eqref{eq:levelset}.
\end{lemma}

We refer to \cite{CIL} for Lemma \ref{lem:stable}.

\section*{Appendix B}\label{sec:appendixB}

In this appendix, we provide a reason of why a priori gradient estimates (Propositions \ref{prop:exuniquegra} and \ref{prop:exuniquegratime}) yield the existence of solutions to \eqref{eq}. We leave \cite{MT} as the main reference.

Let $T\in(0,\infty)$, $X=C^{1,\alpha}(\Omega\times(0,T))$. For a given $w\in X$, we consider the following linear parabolic equation with a source term
\begin{equation}\label{appendixB1}
\begin{cases}
u_t=\textrm{tr}\left\{a(Dw)D^2u\right\}+c(x,w)\sqrt{\eta^2+|Dw|^2}-f(x,w) \quad &\text{ in } \Omega\times(0,T),\\
\displaystyle \frac{\partial u}{\partial \Vec{\mathbf{n}}}=\phi(x)(\sqrt{\eta^2+|Dw|^2})^{1-q}\quad &\text{ on } \partial\Omega\times[0,T),\\
u(x,0)=u_0(x) \quad &\text{ on } \overline{\Omega}.
\end{cases}
\end{equation}
Then, for any $w\in X$, there exists a unique solution $u_w\in C^{2,\alpha'}(\Omega\times(0,T))\subseteq X$ to \eqref{appendixB1} for some $\alpha'\in(0,\alpha)$ with
$$
\|u_w\|_{C^{2,\alpha'}(\Omega\times(0,T))}\leq C_1,
$$
where $C_1>0$ is a constant depending only on $n,\alpha,\|w\|_X,\|u_0\|_{C^{2,\alpha}(\Omega)}$ and on the constants in \eqref{assumtion:c}, \eqref{assumtion:f} (see \cite[Theorem 4.5.2]{LSU}).

Define a map $A:X\to X$ with $Aw=u_w$. Then $A$ is a continuous and compact map. To apply Schauder fixed point theorem, it suffices to prove that the set
$$
S=\{u\in X:u=\sigma Au\textrm{ for some $\sigma\in[0,1]$}\}
$$
is bounded in $X$. Then, $A$ admits a fixed point $u\in C^{2,\alpha'}(\Omega\times(0,T))$, and moreover, $u\in C^{1,\alpha'}(\overline{\Omega}\times[0,T])$ (see \cite{LSU,L2}) since $c,f\in C^{1,\alpha}(\overline{\Omega}\times\mathbb{R})$ and are bounded. Therefore, $u$ becomes a solution to \eqref{eq}, and the regularity of the solution $u$ is improved so that $u\in C^{3,\alpha'}(\Omega\times(0,T))\cap C^{2,\alpha'}(\overline{\Omega}\times[0,T])$ for some $\alpha'\in(0,\alpha)$ from the Schauder theory.

Let $u\in S$. Then, for some $\sigma\in[0,1]$, $u$ solves
\begin{equation}\label{appendixB2}
\begin{cases}
u_t=\textrm{tr}\left\{a(Du)D^2u\right\}+\sigma c(x,u)\sqrt{\eta^2+|Du|^2}-\sigma f(x,u) \quad &\text{ in } \Omega\times(0,T),\\
\displaystyle \frac{\partial u}{\partial \Vec{\mathbf{n}}}=\sigma\phi(x)(\sqrt{\eta^2+|Du|^2})^{1-q}\quad &\text{ on } \partial\Omega\times[0,T),\\
u(x,0)=\sigma u_0(x) \quad &\text{ on } \overline{\Omega}.
\end{cases}
\end{equation}
By Proposition \ref{prop:exuniquegratime}, we have that
$$
\|Du\|_{L^{\infty}(\Omega\times[0,T))}\leq C_2
$$
where $C_2>0$ is a constant depending only on $T,\Omega,c,f,\phi,q,u_0$. Here, we have used the fact that $\sigma\in[0,1]$. By interior Schauder estimates, we also have that
$$
\|Du\|_{C^{\alpha}(\Omega\times(0,T))}\leq C_3
$$
where $C_3>0$ is a constant depending only $T,\Omega,n,\alpha,c,f,\phi,q,u_0$. This yields that the set $S$ is bounded in $X$, and therefore, we obtain the existence.

Now, we apply Proposition \ref{prop:exuniquegra} to the obtained solution to conclude Theorem \ref{thm:global-grad}.

\section*{Appendix C}\label{sec:appendixC}

In this section, we provide the proof of Lemma \ref{lem:boundary} and that of \ref{lem:afternotation}.

\begin{proof}[Proof of Lemma \ref{lem:boundary}]

We take a copy of the space $(\mathbb{R}^n,x)$ with the coordinate $x$ given in the hypothesis of this lemma, and relabel the coordinate $x$ by $y$. We also relabel $x_0$ by $y_0$. We now construct a $C^2$ map $g$ from $(\mathbb{R}^n,y)$ to $(\mathbb{R}^n,x)$ around $y_0$ as follows.

Take an open neighborhood $U_1$ of $y_0=(0,\cdots,0)$ in $\mathbb{R}^n$ and a $C^3$ function $\varphi$ defined on $\{y'=(y_1,\cdots,y_{n-1}):(y',0)\in U_1\}$ such that $y=(y',y_n)\in\partial\Omega$ if and only if $y_n=\varphi(y')$. Then, the $y_{\ell}-$axis lies along an eigenvector corresponding to the eigenvalue $\kappa_{\ell}$ of the matrix $D^2\varphi(y_0)$, $\ell=1,\cdots,n-1$, respectively. Define the map $g:U_1\to \mathbb{R}^n$ by
$$
g(y',y_n)=(y',\varphi(y'))-\Vec{\mathbf{n}}(y',\varphi(y'))y_n.
$$
Then, $g$ is a $C^2$ function on $U_1$. Moreover, with respect to the coordinates $y$ on the domain $U_1\subseteq\mathbb{R}^n$ and $x$ on the codomain $\mathbb{R}^n$, the Jacobian $Jg$ at $(0,\cdots,0,y_n)$, $|y_n|<\sigma$, is the diagonal matrix, as
$$
Jg(0,\cdots,0,y_n)=
\begin{bmatrix}
1-\kappa_{1}y_n & &0 \\
& \ddots & \\
0& & 1-\kappa_{n}y_n
\end{bmatrix},
$$
where $\sigma>0$ is a positive number such that $\{(0,\cdots,0,y_n):|y_n|<\sigma\}\subseteq U_1$ and that $\sigma^{-1}>\max\{|\kappa_1|,\cdots,|\kappa_{n-1}|\}$. In particular, $Jg(0,\cdots,0)$ is the identity matrix, and therefore, by Inverse Function Theorem, there are an open neighborhood $U$ of $(0,\cdots,0)$ in $U_1(\subseteq\mathbb{R}^n)$ and an open neighborhood $V$ of $(0,\cdots,0)$ in $\mathbb{R}^n$ such that $g:U\to V$ is a $C^2$ diffeomorphism from $U$ onto $V$. We take a smaller number $\sigma>0$ if necessary so that $\{(0,\cdots,0,y_n):|y_n|<\sigma\}\subseteq U$ and that $\sigma^{-1}>\max\{|\kappa_1|,\cdots,|\kappa_{n-1}|\}$

By the chain rule, we obtain (iii), and then we obtain (iv) by differentiating (iii) in $y_n$ when $\zeta,\overline{\zeta}$ are $C^2$ functions. For (i), (ii), we refer to \cite[Lemma 14.16]{GL}. 
\end{proof}

We next give the proof of Lemma \ref{lem:afternotation}.

\begin{proof}[Proof of Lemma \ref{lem:afternotation}]
From $a(p)=I_n-\frac{p\otimes p}{\eta^2+|p|^2}$, we see that, for each $\ell=1,\cdots,n$,
$$
a_{p^{\ell}}(Du)=-\frac{1}{\eta^2+|Du|^2}\left(e_{\ell}\otimes Du+Du\otimes e_{\ell}\right)+\frac{2u_{\ell}}{(\eta^2+|Du|^2)^2}Du\otimes Du,
$$
where $e_{\ell}$ is the $\ell$-th element of the standard basis of $\mathbb{R}^n$. Thus,
$$
D_pa\odot\xi=-\frac{1}{\eta^2+|Du|^2}\left(\xi\otimes Du+Du\otimes \xi\right)+\frac{2Du\cdot\xi}{(\eta^2+|Du|^2)^2}Du\otimes Du.
$$
Together with the fact that $\textrm{tr}\{(p\otimes q)M\}=p\cdot(Mq)=q\cdot (Mp)$ for vectors $p,q\in\mathbb{R}^n$ and a symmetric matrix $M$, we obtain
\begin{align*}
v\textrm{tr}\{(D_p(Du)\odot\xi)D^2u\}&=-\frac{2}{\eta^2+|Du|^2}\textrm{tr}\{(\xi\otimes Du)vD^2u\}+\frac{2Du\cdot\xi}{(\eta^2+|Du|^2)^2}\textrm{tr}\{(Du\otimes Du)vD^2u\}\\
&=-\frac{2}{\eta^2+|Du|^2}\xi\cdot(vD^2uDu)+\frac{2Du\cdot \xi}{(\eta^2+|Du|^2)^2}Du\cdot(vD^2uDu)\\
&=-\frac{2}{\eta^2+|Du|^2}\xi\cdot(v^2Dv)+\frac{2Du\cdot \xi}{(\eta^2+|Du|^2)^2}Du\cdot(v^2Dv)\\
&=-2\xi\cdot Dv+\frac{2Du\cdot\xi}{\eta^2+|Du|^2}Du\cdot Dv.
\end{align*}
We have used the fact that $vDv=D^2uDu$. Now use the fact that $(p_1\cdot p_2)(q_1\cdot q_2)=\textrm{tr}\{(p_1\otimes q_1)(p_2\otimes q_2)\}$ for $p_1,p_2,q_1,q_2\in\mathbb{R}^n$. Then,
\begin{align*}
v\textrm{tr}\{(D_p(Du)\odot\xi)D^2u&=-2\left(\xi\cdot Dv-\frac{(Du\cdot\xi)(Du\cdot Dv)}{\eta^2+|Du|^2}\right)\\
&=-2\left(\textrm{tr}\{I_n(\xi\otimes Dv)\}-\frac{\textrm{tr}\{(Du\otimes Du)(\xi\otimes Dv)\}}{\eta^2+|Du|^2}\right)\\
&=-2\left(\textrm{tr}\left\{\left(I_n-\frac{Du\otimes Du}{\eta^2+|Du|^2}\right)(\xi\otimes Dv)\right\} \right)\\
&=-2\textrm{tr}\{a(Du)(\xi\otimes Dv)\},
\end{align*}
and therefore, \eqref{auxiliaryr} is proved.
\end{proof}

\section*{Acknowledgements}
The author would like to express his gratitude sincerely to the reviewers for many valuable comments.

\section*{Data availability}
Data sharing not applicable to this article as no datasets were generated or analyzed during the current study.

\end{document}